\documentclass[10pt,oneside,reqno]{amsart}
\usepackage{textcomp}
\usepackage{amsmath} 
\usepackage{mathrsfs}
\usepackage{amsthm}
\usepackage[linktocpage]{hyperref}
\usepackage{amsfonts,graphics,amsthm,amsfonts,amscd,latexsym}
\usepackage{epsfig}
\usepackage{flafter}
\usepackage{mathtools}
\usepackage{comment}
\usepackage{stmaryrd}
\usepackage{enumitem}
\usepackage{textcomp}
\usepackage{epigraph} 
\usepackage{longtable}
\usepackage{tikz-cd}
\hypersetup{
colorlinks=true,
linkcolor=red,
citecolor=green,
filecolor=blue,
urlcolor=blue
}
\usepackage{tikz}
\usetikzlibrary{graphs,positioning,arrows,shapes.misc,decorations.pathmorphing}

\tikzset{
>=stealth,
every picture/.style={thick},
graphs/every graph/.style={empty nodes},
}

\tikzstyle{vertex}=[
draw,
circle,
fill=black,
inner sep=1pt,
minimum width=5pt,
]
\usepackage[position=top]{subfig}
\usepackage[alphabetic]{amsrefs}
\usepackage{amssymb}
\usepackage{color}

\calclayout

\usetikzlibrary{decorations.pathmorphing}
\tikzstyle{printersafe}=[decoration={snake,amplitude=0pt}]

\newcommand{\mult}{\operatorname{mult}}

\newcommand{\ord}{\operatorname{ord}}

\newcommand{\oo}{\mathcal{O}}

\newcommand{\pp}{\mathbb{P}}

\newcommand{\coreg}{\mathrm{coreg}\,}

\def\O#1.{\mathcal {O}_{#1}}
\def\pr #1.{\mathbb P^{#1}}
\def\af #1.{\mathbb A^{#1}}
\def\ses#1.#2.#3.{0\to #1\to #2\to #3 \to 0}
\def\xrar#1.{\xrightarrow{#1}}
\def\K#1.{K_{#1}}
\def\bA#1.{\mathbf{A}_{#1}}
\def\bM#1.{\mathbf{M}_{#1}}
\def\bL#1.{\mathbf{L}_{#1}}
\def\bB#1.{\mathbf{B}_{#1}}
\def\bK#1.{\mathbf{K}_{#1}}
\def\subs#1.{_{#1}}
\def\sups#1.{^{#1}}

\DeclareMathOperator{\reg}{reg}

\newcommand{\OOO}{\mathscr{O}}

\usepackage{tikz}
\usetikzlibrary{matrix,arrows,decorations.pathmorphing}

\newtheorem{theorem}{Theorem}[section]
\newtheorem{lemma}[theorem]{Lemma}
\newtheorem{proposition}[theorem]{Proposition}
\newtheorem{corollary}[theorem]{Corollary}

\theoremstyle{definition}
\newtheorem{definition}[theorem]{Definition}
\newtheorem{example}[theorem]{Example}

\newtheorem{problem}[theorem]{Problem}
\newtheorem{question}[theorem]{Question}

\newtheorem{remark}[theorem]{Remark}

\theoremstyle{remark}

\numberwithin{equation}{section}
\numberwithin{subsection}{theorem}
\numberwithin{subsubsection}{subsection}

\usepackage[all]{xy}
\newcounter{rownumber}[figure]
\newcounter{rownumber-irr}[figure]
\newcounter{rownumber-p1}[figure]
\begin{document}

\title{Coregularity of smooth Fano threefolds}
\begin{abstract}
We study the coregularity of smooth Fano threefolds. We prove that for $100$ out of $105$ families of smooth Fano threefolds, a general member in the family has coregularity $0$; moreover, for $92$ families out of these $100$, any member in the family has coregularity $0$; for the remaining $5$ families, we obtain some partial results. In particular, we show that there exist families of smooth Fano threefolds whose general elements have positive coregularity.



\end{abstract}

\address{\emph{Artem Avilov}
\newline
\textnormal{National Research University Higher School of Economics, Russian Federation, Laboratory of Algebraic Geometry, NRU HSE.}
\newline
\textnormal{\texttt{v07ulias@gmail.com}}}

\address{\emph{Konstantin Loginov}
\newline
\textnormal{Steklov Mathematical Institute of Russian Academy of Sciences, Moscow, Russia.}
\newline
\textnormal{National Research University Higher School of Economics, Russian Federation, Laboratory of Algebraic Geometry, NRU HSE.}
\newline
\textnormal{Laboratory of AGHA, Moscow Institute of Physics and Technology.}
\newline
\textnormal{\texttt{loginov@mi-ras.ru}}}

\address{\emph{Victor Przyjalkowski}
\newline
\textnormal{Steklov Mathematical Institute of Russian Academy of Sciences, Moscow, Russia.}
\newline
\textnormal{National Research University Higher School of Economics, Russian Federation, Laboratory of Mirror Symmetry, NRU HSE.}
\newline
\textnormal{\texttt{victorprz@mi-ras.ru, victorprz@gmail.com}}}

\maketitle
\setcounter{tocdepth}{1} 
\tableofcontents
\section*{Introduction}
A smooth complex projective variety $X$ is called Fano if its anti-canonical class $-K_X$ is ample. 
A natural way to study Fano varieties is by looking at its (pluri-)anti-canonical linear system $|-lK_X|$ for $l\geq 1$. In dimension not greater than $3$ the existence of a smooth element in the linear system $|-K_X|$ was established by V. Shokurov. This allowed to apply inductive arguments to reduce the study of Fano threefolds to the study of its smooth anti-canonical elements. This approach was used to obtain the classification of smooth Fano varieties in dimension not greater than~$3$. In particular, in dimension $3$, smooth Fano varieties belong to $105$ deformation families, see \cite{IP99}. 

It is also fruitful to look at singular elements in $|-lK_X|$. To measure ``how singular'' such elements could be, the following invariants were introduced. The \emph{dual complex} of a reduced simple normal crossing divisor~$D$ on a smooth variety is a topological space that captures the combinatorial complexity of this divisor, see Definition \ref{def-dual-complex}. Using the resolution of singularities, this definition can be generalized to the case of log canonical pairs.

The notion of \emph{regularity} of a log canonical pair was introduced by V. Shokurov in \cite{Sh00}. By definition, it is the dimension of the dual complex of the boundary divisor of this pair. 
As a generalization, \emph{regularity of a Fano variety} was defined in \cite{Mo22}. By definition, it is the maximum of regularities of log canonical complements on a given Fano variety, see Definition \ref{defin-regularity}. For convenience, the dual notion of \emph{coregularity} is introduced. By definition, 
\[
\mathrm{coreg}(X)=\dim X - 1 - \mathrm{reg}(X).
\]
Roughly speaking, the regularity measures how far a given Fano variety is from being \emph{exceptional}. The latter means that any log canonical complement is in fact Kawamata log terminal. For technical reasons, for any $l\geq 1$ we introduce the $l$-th regularity of a Fano variety, which by definition is the maximum of dimensions of dual complexes of log canonical $l$-complements on a given Fano variety. The $l$-th coregularity is defined as
$
\mathrm{coreg}_l(X)=\dim X - 1 - \mathrm{reg}_l(X).
$

The study of coregularity of Fano varieties drew a lot of attention recently, see \cite{Mo22}, \cite{FFMP22}, \cite{ABBdVILW23}, \cite{Du22} and references therein. It is expected that Fano varieties of coregularity~$0$ should enjoy some good properties. However, not much is known about the coregularity of concrete families of Fano varieties, see Section \ref{sec-dP} for the results in dimension $2$. Another example is the case of toric Fano varieties: it is easy to check that they have coregularity~$0$.  

In this paper, we deal with smooth Fano varieties of dimension $3$. We use the notation for the classification of smooth Fano threefolds as in \cite{IP99}, see also Table \ref{sec-the-table}. Our main result is as follows:
\begin{theorem}
\label{intro-main-theorem}
Let $X$ be a smooth Fano threefold. Put 
\[
\aleph=\{1.1, 1.2\}, \quad \quad \quad \beth=\{1.3, 1.4\}, \quad \quad \quad \gimel=\{1.5\}, \quad \quad \quad \daleth=\{1.6, 1.7, 1.8, 1.9, 1.10, 1.11, 2.1, 10.1\}.
\]
Then the following holds.
\begin{enumerate}
\item
If $X$ is any smooth threefold that belongs to any family except for the families in $\aleph, \beth, \gimel, \daleth$, then $\mathrm{coreg}(X)=0$.
\item
If $X$ is a general member in one of the families $\aleph$, we have $\mathrm{coreg}(X)\geq 1$.
\item
If $X$ is a general member in one of the families $\beth$ we have $\mathrm{coreg}_1(X)=2$.
\item
If $X$ is a general member in the family $\gimel$, we have $\mathrm{coreg}(X)\leq 1$,
\item
If $X$ is a general member in one of the families $\daleth$, we have $\mathrm{coreg}(X)=0$,
\end{enumerate}
\end{theorem}
Consequently, for $100$ out of $105$ families of smooth Fano threefolds, a general member in the family has coregularity $0$; moreover, for $92$ families out of these $100$, any member in the family has coregularity $0$; for the remaining $5$ families, we obtain some partial results. Namely, for the families in $\aleph$, that is for  
a sextic double solid and a quartic, we show that a general member in the family has coregularity at least $1$. For the families in $\beth$, that is for the 
intersection of a quadric and a cubic and the intersection of three quadrics, our methods allowed to compute only the first regularity $\mathrm{reg}_1$. 
Finally, for a general Fano variety in the family~$\gimel$, that is for a general Fano variety of genus $6$, we prove that the coregularity is at least $1$. The following corollary is straightforward.
\begin{corollary}
\label{int-main-cor}
Let $X$ be a smooth Fano threefold. 
If $(-K_X)^3\geq 12$ then for a general $X$ we have $\mathrm{coreg}(X)=~0$. If $(-K_X)^3\geq 24$ then for any $X$ we have $\mathrm{coreg}(X)=0$.
\end{corollary}
For now, we do not know whether the bounds in Corollary \ref{int-main-cor} are sharp. Our computations show that in many cases the second regularity $\mathrm{reg}_2(X)$ (see Definition \ref{defin-regularity}) is equal to $0$ for general elements in the families~$\beth$, 
which by Theorem \ref{thm-1-2-complements} implies that the coregularity is positive for general elements of these families. However, we do not know how to compute the coregularity for a general element in the family $\gimel$. 
The following problem was suggested by V. Shokurov:
\begin{problem}
For any $i\geq 0$, find a (reasonable) function $f_i\colon \mathbb{N}\to \mathbb{R}$ with the following property. 
Assume that $X$ is a klt Fano variety of dimension $n$ and the inequality  
$(-K_X)^n>f_i(n)$ holds. Then $\mathrm{coreg}(X)\leq i$.
\end{problem}

To prove Theorem \ref{intro-main-theorem}, we use the following strategy. In the cases when the coregularity is equal to $0$, we explicitly construct the anti-canonical boundary divisor which evaluates it. It turns out that it is not hard to do in the case when either the index, or the Picard number of a given Fano threefold is at least two. Therefore, the most interesting varieties are Fano threefolds of the main series, that is, Fano varieties of index~$1$ and Picard rank $1$. In this case, we apply the method of the double projection due to Iskovskikh, which gives a natural birational tranformation to another, more simple, Fano threefold. This allows to construct a boundary of coregularity $0$ on this simpler Fano threefold and then pull it back to our initial variety. A kind of surprise is that, unlike in the case of del Pezzo surfaces, the coregularity is not always equal to $0$ for a general element in all families of Fano threefolds. To show this, we use the following crucial result

\begin{theorem}[{\cite[Theorem 4]{FFMP22}}]
\label{thm-1-2-complements}
If $X$ is a smooth Fano variety with $\mathrm{coreg}(X)=0$ then either $\mathrm{coreg}_1(X)=0$ or $\mathrm{coreg}_2(X)=0$.
\end{theorem}

In other words, if the coregularity is $0$, then it is evaluated either on the elements of the anti-canonical, or of the bi-anti-canonical linear system. Thus, to show that coregularity is greater or equal to $1$, it is enough to analyze the singularities of the elements that belong to either of these two linear systems. The proof of Theorem \ref{intro-main-theorem} is summarized in Table \ref{sec-the-table}. 
We propose the following
\begin{question}
\label{question:reg 1}
Let $X$ be a smooth Fano variety. Is it true that $\mathrm{reg}(X)=\mathrm{reg}_1(X)$?
\end{question}
From the proof of Theorem \ref{intro-main-theorem} it follows that the answer to this question is affirmative for all Fano threefolds possibly except those in the families $\aleph, \beth, \gimel$, and for general Fano threefolds in the families from~$\daleth$. 

We provide some motivation for the problem of computing the coregularity of Fano varieties from the point of view of mirror symmetry. The normal crossing anti-canonical complements play an important role in mirror symmetry considerations. 
%
Let us consider a smooth Fano variety $X$ 
together with a normal crossing anti-canonical complement
$D=D_1+\ldots+D_r$ and an open Calabi--Yau variety $U=X\setminus D$. Then the homotopy type of the dual intersection
complex of $D$ is an invariant of $U$. Also, the homology groups of the dual complex of $D$ determine
the weight $0$ part of the Deligne's canonical mixed Hodge structure on the 
cohomology groups of $U$ with compact support. Note that the dimension $k$ of the maximal cell of the dual complex is the depth of the
weight filtration on the mixed Hodge structure.

On the other hand, the mirror symmetry predicts the existence of a mirror dual open Calabi--Yau variety~$U^\vee$.
One can also define Deligne's mixed Hodge structure for $U^\vee$. Moreover, one can consider
the affinization maps $U\to \mathrm{Spec}\,(\mathbb{C}[U])$ and $U^\vee\to \mathrm{Spec}\,(\mathbb{C}[U^\vee])$
and the \emph{perverse Leray filtrations} defined by these maps, see~\cite{CM10, HKP20}.
Under some natural mild assumptions the depth of this filtration for $U^\vee$ is $k$, see~\cite[Remark 1.4]{HKP20}.
Moreover, the \emph{Mirror P=W Conjecture} claims that the dimensions of the associated graded quotients for the
weight filtration for $U$ corresponds to the dimensions of the associated graded quotients for perverse Leray filtration
for $U^\vee$, and vice versa. For more details and precise statements, see \cite{HKP20}.
Thus, the most interesting case is when the dimension 
of the base $U^\vee\to \mathrm{Spec}\,(\mathbb{C}[U^\vee])$
is maximal possible (say, from the point of view of counting vanishing Lagrangian cycles), which by mirror symmetry corresponds to the case when the Fano variety $X$ has the dual complex of $D$ of maximal possible dimension, and so the coregularity of $X$ is equal to $0$.

\begin{question}
\label{question: lc mirror}
Is it possible to generalize the construction of perverse Leray filtration and Mirror P=W Conjecture
for lc pairs?
\end{question}

Answers on Questions~\ref{question:reg 1} and~\ref{question: lc mirror} may be interesting for the Gross--Siebert program of constructing mirror duals to Fano varieties (see~\cite{GrossSiebert}). We recall that this program suggests to construct a mirror of a Fano variety $X$ as a spectrum
of a certain ring whose Krull dimension is equal to a number of components of a simple normal crossing
decomposition of an anti-canonical divisor.  However, in private 
communication M.\,Gross suggested that the approach can be generalized and this ring can be constructed by a log resolution of an anti-canonical divisor on a deformation of $X$,
so that the dimension of the ring is equal to the first regularity of the deformation.

The paper is organized as follows.
In Section \ref{sec-prelim}, we collect the main definitions and fix the notation. 
In Section~\ref{sec-dP}, we compute coregularity of smooth del Pezzo surfaces, that is, Fano varieties of dimension $2$. 
In Section \ref{sec-singularity-theory}, we recall some notions and facts from the singularity theory. 
In Section \ref{sec-preparations}, we prove some preparatory results on the dual complexes of certain pairs. 
In Section \ref{sec-Fanos-large-index}, we start to analyze the Fano threefolds. We explicitly construct boundaries on such varieties which evaluate coregularity $0$. 
More precisely, we consider Fano threefolds of index at least $2$.
In Section \ref{sec-prime-Fano-threefolds}, we consider Fano threefolds of the main series, that is, Fano threefolds of Picard rank $1$ and index $1$. 
In Section \ref{sec-sextic-double-solid}, we deal with the case of a sextic double solid, that is, a Fano threefold in the family $1.1$. 
In Section \ref{sec-quartic}, we consider the case of a quartic hypersurface, that is, a Fano threefold in the family $1.2$. 
In Section \ref{sec-23}, an intersection of a quadric and a cubic, that is, a Fano threefold in the family $1.3$, is analyzed. 
In Section \ref{sec-222}, we study an intersection of three quadrics, that is, a Fano threefold in the family $1.4$.
In Section \ref{sec-Fanos-that-are-not-blow-ups}, we work with Fano threefolds with the Picard rank at least $2$ that cannot be obtained as a blow up of some other Fano variety. 
In sections \ref{sec-Fano-with-big-rho} and \ref{sec-Fano-with-big-rho2}, we work with Fano threefolds with the Picard rank $2$ and at least $3$, respectively, that can be obtained as a blow up of some other Fano variety. 
Finally, in Section \ref{sec-the-table}, we organize all our results in one table.  

\textbf{Acknowledgements.}
The work of K.\,Loginov was performed at the Steklov International Mathematical Center and supported by the Ministry of Science and Higher Education of the Russian Federation (agreement no. 075-15-2022-265), and supported by
the HSE University Basic Research Program, and the Simons Foundation. He is a Young Russian Mathematics award winner and would like to thank its sponsors and jury. 
The work of V.\, Przyjalkowski was performed at the Steklov International Mathematical Center and supported by the Ministry of Science and Higher Education of the Russian Federation (agreement no. 075-15-2022-265). The authors thank Alexander Kuznetsov, Dmitry Mineyev, Joaqu\'in Moraga, Yuri Prokhorov and Vyacheslav Shokurov for helpful conversations, and Constantin Shramov for reading the draft of the paper.

\section{Preliminaries}
\label{sec-prelim}
In what follows, all varieties are
projective and defined over $\mathbb{C}$ unless  stated otherwise. We
use the language of the minimal model program (the MMP for short), see
e.g. \cite{KM98}.

\subsection{Contractions} By a \emph{contraction} we mean a surjective
morphism $f\colon X \to Y$ of normal varieties such that $f_*\OOO_X =
\OOO_Y$. In particular, $f$ has connected fibers. A
\emph{fibration} is defined as a contraction $f\colon X\to Y$ such
that $\dim Y<\dim X$.

\subsection{Pairs and singularities} A \emph{pair} $(X, B)$ consists
of a normal variety $X$ and a boundary $\mathbb{Q}$-divisor $B$ with
coefficients in $[0, 1]$ such that $K_X + B$ is $\mathbb{Q}$-Cartier.
Let $\phi\colon W \to X$ be a
log resolution of $(X,B)$ and let $K_W +B_W = \phi^*(K_X +B)$.
The \emph{log discrepancy} of a prime divisor $D$ on $W$ with respect
to $(X, B)$ is $1 - \mathrm{coeff}_D B_W$ and it is denoted by $a(D, X, B)$. We
say $(X, B)$ is lc (resp. klt) if $a(D, X, B)$
is $\geq 0$ (resp.~$> 0$) for every $D$. We say that the pair is \emph{plt}, if $a(D, X, B)>0$ holds for any $\phi$-exceptional divisor $D$ and for any log resolution $\phi$. 
We say that the pair is \emph{dlt}, if $a(D, X, B)>0$ hold for any $\phi$-exceptional divisor $D$ and for some log resolution $\phi$.

\subsection{Complements and log Calabi-Yau pairs}
\label{sect-log-CY}
Let $(X, B)$ be an lc pair. 
We say $(X,B)$ is \emph{log Calabi-Yau pair} (or \emph{log CY} for short) if $K_X + B
\sim_{\mathbb{Q}} 0$. In this case, $B$ is called a $\mathbb{Q}$-complement of $K_X$. If $N(K_X + B)\sim 0$ for some $N$, we say that $B$ is an $N$-complement of $K_X$. 
\subsection{Dual complex and coregularity}
Let $D=\sum D_i$ be a Cartier divisor on a smooth variety $X$. Recall that $D$ has \emph{simple normal crossings} (snc for short), if all the components $D_i$ of $D$ are smooth, and any point in $D$ has an open neighborhood in the analytic topology that is analytically equivalent to the union of coordinate hyperplanes.

\begin{definition}
\label{def-dual-complex}
\emph{The dual complex}, denoted by $\mathcal{D}(D)$, of a simple normal crossing divisor $D=\sum_{i=1}^{r} D_i$ on a smooth variety $X$ is a CW-complex constructed as follows. 
The simplices $v_Z$ of $\mathcal{D}(D)$ are in bijection with irreducible components $Z$ of the intersection $\bigcap_{i\in I} D_i$ for any subset $I\subset \{ 1, \ldots, r\}$, and the dimension of $v_Z$ is equal to $\#I-1$.
The gluing maps are constructed as follows. 
For any subset $I\subset \{ 1, \ldots, r\}$, let $Z\subset \bigcap_{i\in I} D_i$ be any irreducible component, and for any $j\in I$ let $W$ be a unique component of $\bigcap_{i\in I\setminus\{j\}} D_i$ containing $Z$. Then the gluing map is the inclusion of $v_W$ into $v_Z$ as a face of $v_Z$ that does not contain the vertex $v_i$. 

Note that the dimension of $\mathcal{D}(D)$ does not exceed $\dim X-1$. If $\mathcal{D}(D)$ is empty, we say that $\dim \mathcal{D}(D)=-1$. In what follows, for a divisor $D$ by $D^{=1}$ we denote the sum of the components of $D$ that have coefficient $1$ in it. For a lc log CY pair $(X, D)$, we define $\mathcal{D}(X, D)$ as $\mathcal{D}(D_Y^{=1})$ where $f\colon (Y, D_Y)\to (X, D)$ is a log resolution of $(X, D)$, so that the formula
\[
K_{Y} + D_Y= f^*(K_X + D)
\]
is satisfied. It is known that the PL homeomorphism class of $\mathcal{D}(D_Y^{=1})$ does not depend on the choice of a log resolution
, see \cite{dFKX17}. 
\end{definition}

For more results on the topology of dual complexes of Calabi-Yau pairs, see \cite{KX16}. In the case of log Fano pairs, see also \cite{Lo19} and \cite{LM20}.

\begin{definition}[{\cite[7.9]{Sh00}},{\cite{Mo22}}]
\label{defin-regularity}
Let $X$ be a klt Fano variety of dimension $n$. By the \emph{regularity} $\mathrm{reg}(X, D)$ of an lc log CY pair $(X, D)$ we mean the number $\dim \mathcal{D}(X, D)$. For $l\geq 1$, we define the $l$-th \emph{regularity} of $X$ by the formula 
\[
\mathrm{reg}_l(X) = \max \{ \mathrm{reg}(X, D)\ |\ D\in \frac{1}{l} |-lK_X|\}.
\]
Then the \emph{regularity} of $X$ is 
\[
\mathrm{reg}(X) = \max_{l\geq 1} \{\mathrm{reg}_l(X)\}.
\]
Note that $\mathrm{reg}(X)\in \{-1, 0,\ldots, \dim X-1\}$ where by convention we say that the dimension of the empty set is $-1$. The \emph{coregularity} of an lc log CY pair $(X, D)$ is defined as the number $n-1-\mathrm{reg}(X, D)$. Also, we define:
\[
\mathrm{coreg}_l(X) = n-1-\mathrm{reg}_l(X),\quad \quad \quad \mathrm{coreg}(X) = n-1-\mathrm{reg}(X).
\]
Clearly, one has $\mathrm{reg}_l(X)\leq \mathrm{reg}_{kl} (X)$ for any $k,l\geq1$.
\end{definition}

We recall the following property of the coregularity.

\begin{proposition}[{\cite[Proposition 3.28]{Mo22}}]
\label{proposition:coreg of fibration}
Let $\phi\colon X\to Y$ be a fibration. Assume that $X$ is of Fano
type (this holds, for example, if $X$ is a Fano variety). Then $\coreg(X)\ge\coreg(Y)$.
In particular, if $\coreg(X)=0$, then $\coreg(Y)=0$.
\end{proposition}

\section{Del Pezzo surfaces}
\label{sec-dP}
As a warm-up, we compute the coregularity of smooth del Pezzo surfaces, that is, of smooth Fano varieties of dimension $2$. By the degree of a smooth del Pezzo surface $X$ we mean the number $d=(-K_X)^2$. It is well-known that $1\leq d\leq 9$.


\begin{theorem}[{\cite[Theorem 3.2]{Mo22}}]
\label{theorem:del Pezzo degree 2 and more}
Let $X$ be a smooth del Pezzo of degree $d\geq 2$.
Then $\coreg(X)=0$.
\end{theorem}
\begin{proof}
First we show that for $d\geq 2$ in the linear system $|-K_X|$ there always exists a nodal curve. Indeed, let us blow up $d-2$ general points on $X$ to obtain a morphism $f\colon X'\to X$ where $X'$ is a del Pezzo surface of degree $2$. The anti-canonical linear system $|-K_{X'}|$ defines a double cover $g\colon X'\to \mathbb{P}^2$ ramified in a smooth quartic curve $C'\subset \mathbb{P}^2$. The preimage $g^{-1}(L')$ of a general tangent line $L'$ to $C'$ is a nodal curve in $|-K_{X'}|$. Then $C=f(C')$ is the desired nodal curve on $X$. Blowing up the node of $C$ on $X$ we obtain a boundary that evaluates coregularity $0$ for $X$. 
\end{proof}

\begin{remark}
\label{rem-curve-dP1}
We check that on a general del Pezzo surface $X$ of degree $1$ there exists an element $C$ in $|-K_X|$ which is a nodal curve. Indeed, since this condition is open in the space of all del Pezzo surfaces of degree $1$ it is enough to show that there exists a del Pezzo surface $X$ of degree $1$ such that all $12$ singular elements in $|-K_X|$ are nodal curves. Such a surface is given by the equation
\begin{equation}
X = \{ w^2 + z^3 + zf_4(x, y) + f_6(x, y) = 0\} \subset \mathbb{P}(1,1,2,3)
\end{equation}
where $(x, y, z, w)$ are the coordinates with weights $(1,1,2,3)$, the polynomials $f_4$ and $f_6$ are general and homogeneous of degree $4$ and $6$, respectively. 

On the other hand, there exist del Pezzo surfaces of degree $1$ such that $|-K_X|$ contains either smooth or cuspidal curves. Indeed, consider a family of surfaces
\[
X = \{ w^2 + z^3 + f_6(x, y) = 0\} \subset \mathbb{P}(1,1,2,3)
\]
where $f_6(x, y)$ is a general homogeneous polynomial of degree $6$. Then $|-K_X|$ has exactly $6$ singular elements, all of which are cuspidal curves. 
\end{remark}

\begin{proposition}
\label{prop:delPezzo1}
A general del Pezzo surface of degree $1$ has coregularity $0$. For a
special del Pezzo surface $X$ of degree $1$, one has $\mathrm{reg}_1(X)=0$ and $\mathrm{reg}_2(X)=0$. As a consequence, $\mathrm{coreg}(X)=1$.
\end{proposition}
\begin{proof}
Let $X$ be a general del Pezzo surface of degree $1$. By Remark \ref{rem-curve-dP1}, there exists
a nodal curve $C\in|-K_X|$. Then it is clear that the pair $(X, C)$ is
lc. Blow up the node to obtain a pair $(Y, C'+E)$ which evaluates
$\coreg (Y)=0$. 

Now let $X$ be any del Pezzo surface of degree $1$. Assume that there is no nodal curve in $|-K_X|$. We prove that
$\coreg(X)=1$ in this case. Note that $X$ is a blow up of a projective plane at
$8$ points in a general position (that means that no $3$ of them lie on
a line, no $6$ of them lie on a conic and there is no cubic curve
passing through all of them and singular at one of them). The proper
transform of a cubic curve passing through all of these points is a
$1$-complement, so the regularity of a del Pezzo surface of degree $1$ is
at least~$0$. 

Due to Theorem \ref{thm-1-2-complements} we need to study only
$1$- and $2$-complements on $X$. Let $B$ be a boundary
such that $(X, B)$ is a log CY pair and $2(K_{X}+B)\sim 0$.
Let $\pi\colon X\to \pp^{2}$ be a blow up of the $8$ points and let $B_{2}$
be a push-forward of $B$. Then $2(K_{\pp^{2}}+B_{2})\sim 0$ and
$K_{X}+B=\pi^{*}(K_{\pp^{2}}+B_{2})$. Thus $B_{2}$ has multiplicity at
least~$1$ at each point blown-up by $\pi$.

First we consider the case when $B$ is a $1$-complement. 
If $B_{2}$ is a sum of three lines, then every line contains at most
two blown-up points, so there are $6$ of them in total. Since $B$ is an effective divisor, this case does not occur. If
$B_{2}$ is a sum of a line and a conic, then the conic contains at
most $5$ blown-up points and there are $7$ of them in total. Thus this case does not occur as well. 
Finally, if $B_{2}$ is a cubic curve, then the multiplicity of $B_{2}$
at the blown-up points is equal to $1$ and $B$ is a proper
transform of $B_{2}$. Since the pair $(X, B)$ is lc, this shows that $B$ is smooth. This shows that $\mathrm{reg}_1(X)=0$. 

Now we consider the case when $B$ is a $2$-complement (which is not a $1$-complement). It means that all
coefficients of $B_{2}$ are positive numbers of the form $\frac{a}{2},
a\in\mathbb{N}$. Assume that one of these coefficients is greater than
$\frac{1}{2}$. Then we have $B_{2}=C+B_{3}$ where $C$ is a curve of
degree $2$ or $1$ and $B_{3}$ is not a $\mathbb{Z}$-divisor. If $\deg
C=2$, then $B_{3}=\frac{1}{2}C_{2}$, where $C_{2}$ is a conic
(possibly reducible). The curve $C$ contains at most $5$ blown-up points and in any case the divisor $B_{3}$ cannot have multiplicity
$1$ at three remaining points. Since $B$ is an effective divisor, this case does not occur.

If $\deg C=1$, then $C$ contains at most two blown-up points and
$B_{3}$ must have multiplicity at least~$1$ at the six remaining points.
Let $C_{2}$ be a conic passing through five of them. If $C_{2}$ is not
a component of $B_{3}$, then $4=B_{3}\cdot C_{2}\geq 5$, a
contradiction. Thus $B_{3}=\frac{1}{2}C_{2}+B_{4}$. But again
$2=B_{4}\cdot C_{2}\geq \frac{5}{2}$, a contradiction.

So we know that $B_{2}=\frac{1}{2}B_{3}$ where $B_{3}$ is a reduced
divisor of degree $6$. We know that $B_{3}$ is singular at every point
of the blow up (and maybe somewhere else). If all singularities of $B_2$ are
double points then $B$ is a proper transform of $B_{2}$, the pair $(X, B)$ is dlt and $\mathcal{D}(X, B)$ is empty. 

Assume
that $B_{3}$ contains two of the blown up points of multiplicity $3$ or a
blown up point of multiplicity $4$. Then we can consider a cubic
curve $C$ passing through all the blown up points and a general point of
$B_{3}$. If $C$ is not a component of $B_{3}$ then we have $18=C\cdot
B_{3}\geq 19$, which is impossible. So in this case $B_{3}$ is a union
of two cubic curves $B'$ and $B''$ passing through the blown-up points. 
The same argument easily excludes the case
of at least $9$ singular points such that one of them is at least
triple.

So we have the last case: $B_{3}$ has one triple point and $7$ double
points. Then $B=\frac{1}{2}(\widetilde{B}+E)$ where $\widetilde{B}$ is
the proper transform of $B_{3}$ and $E$ is an exceptional divisor over
the triple point. There exists a curve $C\in |-K_{X}|$ passing through the
intersection point of $\widetilde{B}$ and $E$. Then one can calculate
that $1=C\cdot \widetilde{B}$, so any such point of $B$ is a double
point. Thus again the pair $(X, B)$ is dlt and $\mathcal{D}(X, B)$ is empty.

Now, Theorem \ref{thm-1-2-complements} implies that $\mathrm{coreg}(X)=1$.
\end{proof}

\begin{remark}
Note that Theorem 3.2 in~\cite{Mo22} claims, in particular, that any
smooth del Pezzo surface of degree $1$ has coregularity $1$.
\end{remark}

Theorem~\ref{theorem:del Pezzo degree 2 and more} and
Proposition~\ref{prop:delPezzo1} imply that a general smooth del Pezzo surface has coregularity $0$.


\begin{corollary}
\label{corollary:P1 x del Pezzo}
Let $X=\mathbb{P}^1\times S$ where $S$ is a smooth del Pezzo surface
of degree at least $2$. Then $\coreg(X)=~0$.
\end{corollary}
\begin{proof}
Let $p_1,p_2$ be two distinct points on
$\mathbb{P}^1$ and let $C$ be a nodal anti-canonical curve on $S$ which exists by the proof of Theorem \ref{theorem:del Pezzo degree 2 and more}. Let $p\in C$ be a node. Put $D=D_1+D_2+D_3$ where 
$D_1=p_1\times S$, $D_2=p_2\times S$, $D_3=\mathbb{P}^1\times C$. Then $K_X+D\sim 0$, the pair $(X, D)$ is lc and blowing up $\mathbb{P}^1\times p$ on $X$ we see that $\dim \mathcal{D}(X, D)=2$, and so $\mathrm{coreg}(X)=0$. 
\end{proof}

Analogously, we obtain

\begin{corollary}
\label{lemma:coreg P1xS1}
Let $X=\mathbb{P}^1\times S$ where $S$ is a smooth del Pezzo surface
of degree $1$. Then for general $X$, we have $\coreg(X)=0$.
\end{corollary}

Let $X=\mathbb{P}^1\times S$ where $S$ is a smooth del Pezzo surface of degree $1$. Then Proposition~\ref{proposition:coreg of fibration} applied to the fibration $X=\mathbb{P}^1\times S\to S$ and Proposition \ref{prop:delPezzo1} imply that for a special $X$ one has $\coreg(X)=1$.


\section{Singularity theory}
\label{sec-singularity-theory}
Recall some special types of surface singularities.

\subsection{Simple elliptic and cusp singularities}
\label{subsec-simple-ell-cusp}
By the classification (see e.g. \cite{KM98}), striclty lc Gorenstein surface singularities are either simple elliptic or cusp singularities. 
\label{subsec-cusp-elliptic}
Recall that a normal surface $D$ has a \emph{simple elliptic singularity} at the point $0\in D$, if $D$ is strictly lc at $0$ and in the minimal resolution $f\colon \widetilde{D}\to D$ the preimage of $0$ is a smooth elliptic curve. The classification implies that for a simple elliptic singularity we have either $\mult_0 D\geq 3$, or after an analytic coordinate change the local equation of $0\in D$ is
\begin{equation}
\label{eq-simple-elliptic-sing-1}
x_1^2 + q(x_2, x_3) = 0\ \quad \quad \text{or}\ \quad \quad x_1^2+x_2^3+ax_2x_3^4+bx_3^6=0\
\end{equation}
where $q$ is homogeneous of degree $4$ and $a,b\in\mathbb{C}$.

\label{hypersurface-cusp-sing}
A normal surface $D$ has a \emph{cusp singularity} $0\in D$, if $D$ is strictly lc at $D$ and in the minimal resolution $f\colon \widetilde{D}\to D$ the preimage of $0$ is either a cycle of smooth rational curves, or an irreducible rational curve with one node. By the classification, any two-dimensional hypersurface cusp singularity up to analytic change of coordinates has the following form:
\begin{equation}
\label{eq-hypersurface-cusp-sing}
T_{p,q,r}:\quad x_1^p + x_2^q + x_3^r + x_1 x_2 x_3 = 0
\end{equation}
where $1/p+1/q+1/r< 1$ and $p\leq q\leq r$.

The following lemma is well-known.

\begin{lemma}[cf. {\cite[11.1]{AGV85}}]
\label{lem-corank-second-differenetial}
Let $f(x_1, \ldots, x_n)$ be a germ of a holomorphic function such that $0\in \mathbb{C}^n$ is its critical point. Assume that the quadratic term in the Taylor expansion near $0$ has rank $k$. Then up to a holomorphic change of coordinates we have
\[
f(x_{1},\ldots, x_n) = x_1^2 + \ldots + x_k^2 + g(x_{k+1},\ldots, x_n)
\]
where $g(x_{k+1},\ldots, x_n)$ is a germ of a holomorphic function whose quadratic term is equal to zero.
\end{lemma}

Our main goal in this section is to obtain the characterization of strictly lc $1$- and $2$-complements in terms of their local equations. In what follows, if $f=0$ is an analytic equation near a point $0\in \mathbb{C}^n$, we denote by $f_k$ its $k$-th term in the Taylor expansion $f=\sum_{i\geq0} f_i$ near the point $0$. 
The next lemma is inspired by \cite[Proposition 2.1.1]{CPW14}.
\begin{lemma}
\label{lem-striclty-lc-1-complement}
Let $0\in D=\{f=0\}$ be a germ of a reduced irreducible normal surface in $\mathbb{C}^3$. 
Assume that the pair $(\mathbb{C}^3, D)$ is strictly lc at $0\in\mathbb{C}^3$. Then  
\begin{itemize}
\item
either $f_2=0$, 
\item
or up to a change of coordinates $f_2=x_1^2$ and $(f|_{x_1=0})_3=0$,
\item
or up to a change of coordinates $f_2=x_1^2$, $(f|_{x_1=0})_3=x_2^3$ and $(f|_{x_1=x_2=0})_4=0$. 
\end{itemize}
\end{lemma}
\begin{proof}
Assume that the pair $(\mathbb{C}^3, D)$ is strictly lc at $0\in\mathbb{C}^3$. Of course, this implies that $f_0=f_1=0$. Also, if the rank of $f_2$ is at least two, then by Lemma \ref{lem-corank-second-differenetial} we see that $D$ has du Val singularity at $0$, so $(\mathbb{C}^3, D)$ cannot be strictly lc. Hence either $f_2$ vanishes, and we are done, or $f_2$ has rank $1$, so up to change of coordinates we have $f_2=x_1^2$. 
By the classification of strictly lc singularities \ref{subsec-simple-ell-cusp}, we see that $D$ has either simple elliptic or cusp singularities. In the case of cusp singularities, looking at the equation \eqref{eq-hypersurface-cusp-sing}, we observe that either in the triple $(p,q,r)$ we have $p, q, r\geq 3$, so $\mult_0 D\geq 3$, and hence $f_2=0$, or $\mult_0 D=2$, so we may assume that $p=2$, $q\geq 3$ and $r\geq 5$. Thus up to a change of coordinates we have two cases: either $f_2=x_1^2$, $(f|_{x_1=0})_3=x_2^3$, or $f_2=x_1^2$, $(f|_{x_1=0})_3=0$, 
and $(f|_{x_1=x_2=0})_4=0$.

Now consider the case of simple elliptic singularities. Again, $\mult_0 D\geq 3$ implies $f_2=0$. If $\mult_0 D=2$, in the first case of \eqref{eq-simple-elliptic-sing-1} we have $f_2=x_1^2$ and $(f|_{x_1=0})_3=0$, while in the second case $f_2=x_1^2$, $(f|_{x_1=0})_3=x_2^3$ and $(f|_{x_1=x_2=0})_4=0$. The claim follows.
\end{proof}

We will prove an analogous statement for boundaries with half-integer coefficients.

\begin{proposition}
\label{lem-striclty-lc-2-complement}
Let $0\in D=\{f=0\}$ be a germ of a reduced surface in $\mathbb{C}^3$.
Assume that the pair $(\mathbb{C}^3, 1/2D)$ is strictly lc at $0\in\mathbb{C}^3$. Then $f_2=0$ and
\begin{itemize}
\item
either $f_3=0$, 
\item
or up to a change of coordinates $f_3=x_1^2x_2$ and $f_4$ is not divisible by~$x_1$,
\item
or up to a change of coordinates $f_3=x_1^3$ and $f_4$ is not divisible by~$x_1^2$.
\end{itemize}
\end{proposition}



Recall that the largest real number $\lambda$ such that the pair $(X, \lambda D)$ is log canonical is called a \emph{log canonical threshold} of $(X, D)$. We denote it by $\mathrm{lct}(X, D)$. We will use the following estimate.

\begin{lemma}[{\cite[Proposition 8.13]{Ko97}}]
\label{lem-weights}
Let $f$ be a holomorphic function near $0\in \mathbb{C}^{n}$. Assign rational weights $w(x_{i})$ to the coordinates $x_1,\ldots, x_n$ in $\mathbb{C}^{n}$ and let $w(f)$ be the weighted multiplicity of $f$. Let $f_{w}$ denote the weighted homogeneous leading term of $f$. Then
\[
\mathrm{lct}(\mathbb{C}^n, \{f=0\})\leq\frac{\sum w(x_{i})}{w(f)}
\]
where $w(f)$ is the weight of $f_w$ with respect to $w(x_{i})$. 
If $\{f_{w}=0\}\subset \mathbb{C}^{n}$ is lc outside $0$, then the equality holds.
\end{lemma}

\begin{lemma}
\label{lem-sum-lct}
Let $0\in D$ be a germ of a reduced hypersurface in $\mathbb{C}^n$ such that $\mult_0 D=2$. Then the pair $(\mathbb{C}^n, 1/2 D)$ is klt near $0$.
\end{lemma}
\begin{proof}
By Lemma \ref{lem-corank-second-differenetial} we may assume that 
\[
D=\{x_1^2+\ldots+x_k^2+g(x_{k+1},\ldots, x_{n})=0\}\subset \mathbb{C}^{n}
\] 
where $k\geq 1$ and $\mult_0 g\geq 3$. Put
\[
D_1=\{x_1^2+\ldots+x_k^2=0\}\subset\mathbb{C}^k, \quad \quad \quad D_2=\{g(x_{k+1},\ldots, x_{n})=0\}\subset \mathbb{C}^n.
\]
First assume that $g$ is a zero polynomial. Then since $D$ is reduced, we see that $k\geq 2$, so $D_1$ has an ordinary double point at the origin, and $\mathrm{lct}(\mathbb{C}^n, D) = 1$. So we may assume that $g$ is a non-zero polynomial. By \cite[Proposition 8.21]{Ko97} we have 
\[
\mathrm{lct}(\mathbb{C}^{n}, D)) 
= \min \{1,\mathrm{lct}(\mathbb{C}^{k},D_1) + \mathrm{lct}(\mathbb{C}^{n-k},D_2)\}.
\]
We may assume that $\mathrm{lct}(\mathbb{C}^{k},D_1) + \mathrm{lct}(\mathbb{C}^{n-k},D_2)\leq 1$, otherwise $\mathrm{lct}(\mathbb{C}^{n}, D))=1$ and the pair $(X, 1/2)$ is klt. Hence
\[
\mathrm{lct}(\mathbb{C}^{n}, D)) \geq 1/2+ \mathrm{lct}(\mathbb{C}^{n-k},D_2)>1/2
\]
 where we used the fact that $\mathrm{lct}(\mathbb{C}^k, D_1)\geq1/2$ which is straightforward. 
 This shows that the pair $(X, 1/2D)$ is klt as claimed.
\end{proof}

\begin{lemma}
\label{lem-kuwata}
Let $0\in D=\{f=0\}$ be a germ of a reduced surface in $\mathbb{C}^3$ such that $\mult_0 D=3$.  
Let $f_3$ be a degree $3$ term of $f$, and let $C_3$ be a curve in $\mathbb{P}^2$ given by the equation $f_3=0$. Then if $C_3$ does not contain a double line then the pair $(\mathbb{C}^3, 1/2 D)$ is klt near~$0$.
\end{lemma}
\begin{proof}
Assume that the cubic term $f_3$ is not divisible by a square of a linear form. Consider the subset $C_3=\{ f_3=0\}\subset \mathbb{P}^2$. Assume that $C_3$ is lc. This happens if $C_3$ is either a smooth cubic curve, or a nodal curve, or the union of a line and a conic that intersect transversally, or the union of three lines that form a triangle. Then the subset $\{f_3=0\}\subset \mathbb{C}^3$ is lc outside the origin, so assigning the weights $(1,1,1)$ to the coordinates $x_1,x_2,x_3$ by Lemma \ref{lem-weights} we conclude that the pair $(\mathbb{C}^3, D)$ is lc, so the pair $(\mathbb{C}^3, 1/2D)$ is klt as claimed.

Now assume the curve $C_3$ is not lc, then $C_3$ is one of the following: a cuspidal cubic, a line tangent to a conic, or three lines passing through a point. In these cases the log canonical thresholds of the pair $(\mathbb{C}^3, D)$ were estimated in the proof of the main theorem in \cite{Ku99} and in all these cases they are greater than $1/2$, so the pair $(\mathbb{C}^3, 1/2D)$ is klt. This completes the proof.
\end{proof}

\begin{corollary}
\label{cor-curve-x2y}
\label{cor-curve-x3}
Let $0\in C=\{f=0\}\subset \mathbb{C}^2$ be a curve germ where 
\[
f = x_1^2 x_2 + f_4(x_1,x_2) + f_{\geq 5}(x_1,x_2)
\]
with $\deg f_4=4$ and $\mult_0 (f_{\geq 5})\geq 5$. Assume that $f_4|_{x_1=0}\neq 0$. 
Then the pair $(\mathbb{C}^2, 1/2C)$ is klt. 
\end{corollary}
\begin{proof}
Observe that the condition $f_4|_{x_1=0}\neq 0$ is invariant under analytic changes of coordinates which preserve the cubic term of the equation.
Indeed, this follows from the fact that  
$f_4|_{x_1=0} = (f|_{x_1=0})_4$ where the latter means the quartic (and thus the leading) term of the polynomial $f|_{x_1=0}$. 
We can write down the equation of $C$ in the form
\begin{equation}
\label{eq-cor-curve-x2y}
f = x_1^2 x_2 + ax_2^4 + b x_1x_2^3 + x_1^2 r(x_1, x_2) + s(x_1, x_2) = 0
\end{equation}
where $r(x_1, x_2)$ is homogeneous of degree $2$ (or is a zero polynomial), $\mult_0 s(x_1,x_2)\geq 5$, $a,b\in\mathbb{C}$ and $a\neq 0$. 
After an analytic coordinate change 
we obtain the equation 
\[
x_1^2 x_2 + ax_2^4 + s(x_1, x_2) = 0
\]
with $\mult_0 s(x_1,x_2)\geq 5$ and $a\neq 0$. Assign the weights $w=(3,2)$ to the coordinates $x_1, x_2$. Then we have $f_w=x_1^2 x_2 + ax_2^4$ and by Lemma \ref{lem-weights} we have $\mathrm{lct}(\mathbb{C}^2, C) = 5/8 > 1/2$. 
Hence the pair $(\mathbb{C}^2, 1/2C)$ is klt.
\end{proof}

\begin{corollary}
\label{cor-curve-x3}
Let $0\in C=\{f=0\}\subset \mathbb{C}^2$ be a curve germ where 
\[
f = x_1^3 + f_4(x_1,x_2) + f_{\geq 5}(x_1,x_2)
\]
with $\deg f_4=4$ and $\mult_0 (f_{\geq 5})\geq 5$. Assume that $f_4$ is not divisible by $x_1^2$ (in particular, $f_4\neq 0$).
Then the pair $(\mathbb{C}^2, 1/2C)$ is klt. 
\end{corollary}
\begin{proof}
Observe that the condition that $f_4$ is not divisible by $x_1^2$ is invariant under analytic changes of coordinates which preserve the cubic term of the equation. Indeed, after an analytic change of coordinates 
we may assume that our equation has the form
\[
x_1^3 + ax_2^4 + b x_1x_2^3 + s(x_1, x_2) = 0.
\]
where $\mult_0 s(x_1,x_2)\geq 5$, $a,b\in\mathbb{C}$ and either $a\neq 0$ or $b\neq0$. 
If $a\neq 0$, assign the weights $w=(4,3)$ to the coordinates $x_1,x_2$. Then we have $f_w=x_1^3 + ax_2^4$ and by Lemma \ref{lem-weights} we have $\mathrm{lct}(\mathbb{C}^2, C) = 7/12 > 1/2$. If $a=0$ and $b\neq 0$ assign the weights $w=(3, 2)$ to the coordinates $x_1, x_2$. Then we have $f_w=x_1^3 + ax_1x_2^3$ and by Lemma \ref{lem-weights} we have $\mathrm{lct}(\mathbb{C}^2, C) = 5/9 > 1/2$. Hence in both cases the pair $(\mathbb{C}^2, 1/2C)$ is klt.
\end{proof}

\begin{corollary}
\label{cor-surface-x2y}
Let $0\in D\subset \mathbb{C}^3$ be a germ of a reduced surface $D=\{f=0\}$ where 
\[
f = x_1^2 x_2 + f_4(x_1,x_2,x_3) + f_{\geq 5}(x_1,x_2,x_3) 
\]
with $\deg f_4=4$ and $\mult_0 f_{\geq 5}\geq 5$. Assume that $f_4$ is not divisible by $x_1$ (in particular, $f_4\neq 0$). 
Then the pair $(\mathbb{C}^3, 1/2S)$ is klt.
\end{corollary}
\begin{proof}
We may write down the equation of $f$ in the form
\[
f = x_1^2 x_2 + p_4(x_2, x_3) + x_1 q_3(x_2, x_3) + x_1^2 r_2(x_2, x_3) + s(x_1, x_2, x_3) = 0
\]
where $\deg r_2(x_2, x_3) = 2$ (or $r_2(x_2, x_3)=0$), $\mult_0 s(x_1, x_2, x_3)\geq 5$ and either $p_4(x_2, x_3)\neq 0$ and $\deg p_4(x_2, x_3)=~4$, or $q_3(x_2, x_3)\neq 0$ and $\deg q_3(x_2, x_3)=3$. Let $H=\{\lambda x_2=\mu x_3\}$ be a plane where $\lambda, \mu\in\mathbb{C}$. Note that the pair $(\mathbb{C}^3, H + 1/2D)$ is plt if and only if the pair $(H, 1/2D|_H)$ is klt. Then the equation becomes (assume without loss of generality that $\mu=1$)
\[
f = x_1^2 x_2 + x_2^4 p_4(1,\lambda) + x_1 x_2^3 q_3(1,\lambda) + x_1^2 x_2^2 r_2(1,\lambda) + s(x_1, x_2, \lambda x_2) = 0.
\]
This equation has the form \eqref{eq-cor-curve-x2y} for some $\lambda$ and $\mu$, and the claim follows from Corollary \ref{cor-curve-x2y}.
\end{proof}

Analogously, we obtain

\begin{corollary}
\label{cor-surface-x2y}
Let $\in D\subset \mathbb{C}^3$ be a germ of a reduced surface $D=\{f=0\}$ where 
\[
f = x_1^3 + f_4(x_1,x_2,x_3) + f_{\geq 5}(x_1,x_2,x_3)
\]
with $\deg f_4=4$ and $\mult_0 f_{\geq 5}\geq 5$. Assume that $f_4$ is not divisible by $x_1^2$ (in particular, $f_4\neq 0$). 
Then the pair $(\mathbb{C}^3, 1/2S)$ is klt.
\end{corollary}

Now, Proposition \ref{lem-striclty-lc-2-complement} follows from 
Lemma \ref{cor-surface-x2y} and Lemma \ref{lem-kuwata}.

\section{Preparations}
\label{sec-preparations}
In this section, we prove some technical results that we will use later.
\begin{lemma}
\label{lem-smooth-section}
Let $C$ be a reduced curve on a smooth threefold $X\subset \mathbb{P}^N$ such that for any singular point $Q$ of~$C$ we have $\dim T_Q C = 2$. Also assume that $C$ is a scheme-theoretic intersection of all elements from a (not necessarily complete) linear system $\mathcal{H}$. Then there exists a smooth element $D\in \mathcal{H}$.
\end{lemma}
\begin{proof}
By Bertini theorem, a general element $D\in \mathcal{H}$ is smooth outside $C$. Hence it is enough to prove that $D$ is smooth at all the points of $C$. 
Fix a smooth point $P\in C$. There exist two elements $D_1, D_2\in \mathcal{H}$ such that the intersection $D_1\cap D_2$ is smooth at $P$. Let $D_1$ and $D_2$ locally are given by the equations $f_1=0$ and $f_2=0$, respectively. Then the differentials $d_P f_1$ and $d_P f_2$ are linearly independent at $P$. Hence, to be singular at $P$ is a codimension $2$ condition on the elements of $\mathcal{H}$. Since the curve $C$ is one-dimensional, a general element $D$ in $\mathcal{H}$ is smooth at all smooth points of $C$.

Now fix a singular point $Q\in C$. By assumption, $\dim T_Q C = 2$. Hence all elements of $\mathcal{H}$ cannot be singular at $Q$, so there exists an element $D'\in\mathcal{H}$ such that $T_Q D'=T_Q C$. This means that for a general element $D\in\mathcal{H}$ one has $T_Q D=T_Q C$. In particular, a general element in $\mathcal{H}$ is smooth at $Q$. The same argument shows that a general element in $\mathcal{H}$ is smooth at all singular points of $C$, and the claim follows.
\end{proof}

As an immediate corollary, we see that there exists a smooth hyperplane section passing through a line $L$ on a smooth projective threefold $X$.

\begin{lemma}
\label{lem-line-conic}
\label{reducible_cubic_surface}
Let $D_1$ and $D_2$ be two normal surfaces with at worst du Val singularities in a smooth threefold~$X$. Assume that the scheme-theoretical intersection $D_1\cap D_2$ is a reducible curve $C+L$ where both $C$ and $L$ are smooth rational curves that belong to the smooth locus of $D_1$ and $D_2$, and $C$ intersects $L$ transversally. 
Also assume that one of the following holds:
\begin{enumerate}
\item
$L\cdot D_1=1, \quad N_{L/X}=\oo\oplus\oo\quad \text{or} \quad N_{L/X}=\oo(-1)\oplus\oo(1)$.
\item
$L\cdot D_1=2, \quad N_{L/X}=\oo\oplus\oo(1)$.
\item
$L\cdot D_1=3, \quad N_{L/X}=\oo(1)\oplus\oo(1)$.
\end{enumerate}
Then $\dim \mathcal{D}(X, D) = 2$ where $D=D_1+D_2$.
\end{lemma}
\begin{proof}
Observe that $D$ is not snc. However, the pair $(X, D)$ where $D=D_1+D_2$ is lc by inversion of adjunction. Blow up $L$ to obtain a morphism $g\colon X'\to X$. Denote by $D'_i$ the strict transform of~$D_i$ and by~$E$ the $g$-exceptional divisor. We consider three cases. 
\begin{enumerate}
\item
Assume that $N_{L/X}=\oo(-a)\oplus\oo(a)$ where $a\in\{0, 1\}$, so $E\simeq \mathbb{F}_{2a}$. Let $s$ be a $(-2a)$-curve and $f$ be a fiber on $E$. We have $E|_E = -s-cf$ for some $c$. From the equation
\[
0=-\deg N_{L/X}=E^3 = (E|_{E})^2 = (-s-c f)^2 = -2a +2c
\]
we obtain $c=a$. Using $g^* D_1 = D'_1 + E$, compute 
\[
f = (L\cdot D_1)f = E g^*D_1 = E^2 + E D'_1 = - s - a f + E D'_1, 
\]
thus \[D'_1|_E\sim s+(a+1)f.\] 
\item
Assume that $N_{L/\mathbb{P}}=\oo\oplus\oo(1)$, so $E\simeq \mathbb{F}_1$. Let $s$ be a $(-1)$-curve and $f$ be a ruling on $E$. 
Analogously to the above considered case, we obtain \[D'_1|_E\sim s+2f.\]
\item
Assume that $N_{L/\mathbb{P}}=\oo(1)\oplus\oo(1)$, so $E\simeq \mathbb{P}^1\times \mathbb{P}^1$. Let $s$ and $f$ be two different rulings on $E$. 
Analogously, we obtain \[D'_1|_E\sim s+4f.\]
\end{enumerate}
Observe that in all three cases $D'_1|_E$ is ample. Hence $D'_1$ and $D'_2$ intersect on $E$. Note that the pair \[(X', D'_1+D'_2+E)\] is snc and admits a zero-dimensional stratum, so the claim follows.
\end{proof}

\begin{lemma}
\label{lem-nodal-curve}
Let $D_1$ and $D_2$ be two normal surfaces with at worst du Val singularities in a smooth threefold~$X$. Assume that the scheme-theoretical intersection $D_1\cap D_2$ is a rational curve $C$ of arithmetic genus $1$ with one node $P\in C$. Also assume that $C$ belongs to the smooth locus of $D_1$ and $D_2$. Then $\dim \mathcal{D}(X, D) = 2$ where $D=D_1+D_2$.
\end{lemma}
\begin{proof}
The pair $(X, D)$ where $D=D_1+D_2$ is lc by inversion of adjunction. Blow up $P$ to obtain a morphism $f\colon X'\to X$ with the exceptional divisor $E\simeq \mathbb{P}^2$. Note that the strict transforms $D'_1$ and $D'_2$ intersect in a reducible curve $C'+L$ where $L\subset E$ is a line, $C'$ is a strict transform of $C$ and the components $C'$ and $L$ intersect transversally in two distinct points. Note that $L$ is a $(-1)$-curve on both $D'_i$. Then the pair $(X', D')$ where $D'=D'_1+D'_2$ is the log pullback of the pair $(X, D)$.

Note that $L$ is a complete intersection of $E$ and $D'_1$ (or $E$ and $D'_2$), hence \[N_{L/X'}=\oo(-1)\oplus\oo(1).\] Blow up $L$ on $X'$ to obtain a morphism $g\colon X''\to X'$ with the exceptional divisor $E_1\simeq \mathbb{F}_2$. We have $E_1|_{E_1}=-s-cf$ for some $c$ where $s$ is a $(-2)$-section and $f$ is a fiber on $E$. Compute 
\[
0=-\deg N_{L/X'}=E_1^3 = E_1|_{E_1}^2 = ( -s-cf)^2.
\]
Hence $c=1$. 
Also,
\[
L\cdot D'_1 = L (f^* D_1 - E) = 1.
\]
Then we have 
\[
f = (L\cdot D'_1)f = E_1\cdot g^*D'_1 = E_1^2 + E_1\cdot D''_1 = - s - f + E_1\cdot D''_1, 
\]
thus \[D''_1|_{E_1}\sim s+2f.\] Hence $D''_1|_{E_1}$ is a section of $E\simeq \mathbb{F}_2$ with self-intersection $2$. Thus the pair \[(X'', D''_1+D''_2+E_1)\] is snc and admits a zero-dimensional stratum, and the claim follows.
\end{proof}

\section{Fano threefolds of index $\geq 2$}
\label{sec-Fanos-large-index}

Recall that the index $i(X)$ of a Fano variety $X$ is a maximal natural number $m$ such that $-K_X\sim mH$ where $H$ is an element of the Picard group. 
We start to compute the coregularity of smooth Fano threefolds. First we make the following simple observartion:

\begin{remark}
\label{remark:toric}
For a toric Fano variety, taking the torus-invariant boundary, we see
that the coregularity is equal to $0$. Here is the list of toric Fano
threefolds:
\[
1.17,\ 2.33,\ 2.34,\ 2.35,\ 2.36,\ 3.25,\ 3.26,\ 3.27,\ 3.28,\ 3.29,\
3.30,\ 3.31,\ 4.9,\ 4.10,\ 4.11,\ 4.12,\ 5.2,\ 5.3.
\]
\end{remark}

As a consequence of Remark \ref{remark:toric}, we obtain 
\begin{lemma}
\label{lemma:Pn}
Let $X=\mathbb{P}^n$.
Then $\coreg(X)=0$.
\end{lemma}

\begin{lemma}
\label{lemma:quadrics}
Let $X$ be a smooth quadric in $\mathbb{P}^{n+1}$.
Then $\coreg(X)=0$.
\end{lemma}

\begin{proof}
Observe that $-K_X\sim nH$ where $H$ is a hyperplane section of $X$.
Consider a pair $(X, D=\sum_{i=1}^n H_i)$ where $H_i\in|H|$ are
general hyperplane sections. Note that $D$ has simple normal
crossings and it has a zero-dimensional stratum. This shows that
$\coreg(X)=0$.
\end{proof}

Let $X$ be a del Pezzo threefold, that is, a smooth Fano threefold with $i(X)=2$. By the degree of $X$ we mean the number $d=H^3$ where $-K_X\sim 2H$. We assume that $\rho(X)=1$. Then it is known that $1\leq d\leq
5$. 
By a line on $X$ we mean a smooth rational curve whose
intersection with $H$ equals $1$. It is known that for any line $L\subset X$ on a del Pezzo threefold one has  $N_{L/X}=\oo(a)\oplus \oo(-a)$ for $a\in\{0, 1\}$, see e.g. \cite[Proposition 2.2.8.]{KPS18}. 
\begin{lemma}
\label{lemma:Fano del Pezzo}
Let $X$ be a del Pezzo threefold with $3\leq d\leq
5$. Then $\coreg (X)=0$.
\end{lemma}
\begin{proof}
Let $X$ be a del Pezzo threefold with $-K_X\sim 2H$ and $3\leq
d=H^3\leq 5$. Note that in this case $H$ is very ample, so it defines an embedding of $X$ into the projective space $\mathbb{P}^{d+1}$. 
Let $L$ be a line in $X$ on some
smooth prime divisor $H_1$ such that $H_1\in |H|$. Such a line exists
since $H_1$ is a smooth del Pezzo surface of degree $d$. We construct a boundary $D=H_1+H_2$ where $H_2$ is a general divisor with
$L\subset H_2$ and $H_2\in |H|$. 
By Lemma \ref{lem-smooth-section}, we can pick the surface
$H_2$ to be smooth. 
Note that $L$ is a $(-1)$-curve on both surfaces $H_1$ and $H_2$.
We show that for a general divisor $H_2$, the
restriction $H_2|_{H_1}$ is a union of a line $L$ and a smooth
rational curve $C$ that intersects~$L$ transversally. 
Indeed, consider an exact sequence
\begin{equation}
0\to H^0(X, \OOO_X) \to H^0(X, \OOO_X(H)) \to
H^0(H_1, \OOO_{H_1}(-K_{H_1})) \to 0
\end{equation}
where by adjunction we have $H|_{H_1}\sim -K_{H_1}$ and the last term is $0$ by the Kodaira vanishing theorem. 
It follows that there exists an element $H_2\in |H|$ such that $H_1\cap
H_2$ is a given anti-canonical curve in $H_1$. Since $H_1$ is a del Pezzo surface of degree $d\geq 2$, a general reducible anti-canonical curve is a union of a line and a smooth rational curve such that they intersect transversally in two points. Hence for a general $H_2$, the curve $C$ intersects $L$ transversally in two points. 
It follows that the pair $(X, D)$ is lc by inversion of
adjunction on $H_1$. 
Then we can apply Lemma \ref{lem-line-conic}, which shows that $\mathrm{coreg}(X)=0$.
\end{proof}

\begin{lemma}
\label{lemma:Fano del Pezzo2}
Let $X$ be del Pezzo threefold with $d=2$. Then $\coreg (X)=0$.
\end{lemma}
\begin{proof}
The linear system $|H|$ defines a double cover $g\colon X\to\mathbb{P}^3$ ramified in a smooth quartic surface
$Q\subset \mathbb{P}^3$. Let~$H_1$ be a general hyperplane section of $X$, and let $L\subset H_1$. Note that $L$ is a $(-1)$-curve on $H_1$. Put $H'_1=g(H_1)$. Observe that the restriction $g|_{H_1}\colon H_1\to H'_1\simeq \mathbb{P}^2$  is a double cover ramified in a smooth quartic curve $Q\cap H'_1$. Then $L'=g(L)$ is a line in $\mathbb{P}^3$ which is bitangent to $Q$. In particular, $L'\not\subset Q$. Consider a pencil of planes
passing through $L'$. Since $H'_1$ is a plane passing through $L'$
which corresponds to a smooth hyperplane section $H_1$ on $X$, then the preimage $H_2$ of a general plane $H'_2$ passing through $L'$ is smooth as well. Thus, $H_1\cap H_2$ is a union of two smooth rational curves intersecting transversally in two points, so we can apply Lemma \ref{lem-line-conic} and conclude that $\mathrm{coreg}(X)=0$.
\end{proof}

\begin{lemma}
\label{lem-dP1-nodal-element}
Let $X$ be a general del Pezzo threefold of degree $1$. Put $-K_X\sim 2H$. Let $H_1\in |H|$ be a general element which is a
smooth del Pezzo surface of degree~$1$. Then there exists an irreducible curve $C\in |-K_{H_1}|$ such that it has arithmetic genus $1$ and one nodal singular point. 
\end{lemma}
\begin{proof}
Observe that for a del Pezzo threefold $X$ of degree $1$ to have an element in $|H|$ with at worst nodal singularities is an open property in the family of all del Pezzo threefolds of degree $1$. Hence it is enough to show that there exists one such del Pezzo threefold $X$ with an element $H_1\in |H|$ where $-K_X\sim 2H$ such that $|-K_{H_1}|$ contains a nodal curve. Let $X\subset \mathbb{P}(1,1,1,2,3)$ be given by the equation
\[
w^2 + z^3 + f_4(x_1, x_2, x_3) z + f_6(x_1, x_2, x_3) = 0
\]
where $(x_1, x_2, x_3, z, w)$ have weights $(1,1,1,2,3)$, and $f_4$ and $f_6$ are general homogeneous polynomials of degree $4$ and $6$, respectively. Let $H_1$ be given by the condition $x_3=0$, so the equation for $H_1$ in $\mathbb{P}(1,1,2,3)$ is
\[
w^2 + z^3 + f_4(x_1, x_2, 0) z + f_6(x_1, x_2, 0) = 0,
\]
and $f_4(x_1, x_2, 0)$, $f_6(x_1, x_2, 0)$ also are general homogeneous polynomials of degree $4$ and $6$, respectively. 
Then by Remark \ref{rem-curve-dP1} we see that there exists a nodal curve $C\in |-K_{H_1}|$.
\end{proof}

\begin{lemma}
\label{lemma:Fano del Pezzo1}
Let $X$ be a general del Pezzo threefold with $d=1$. Then $\coreg (X)=0$.
\end{lemma}
\begin{proof}
Let $H_1$ be a general element in $|H|$. By Lemma \ref{lem-dP1-nodal-element}, there exists an element $C$ in $|-K_{H_1}|$ which is an irreducible nodal curve. Consider an exact sequence
\[
0\to H^0(X, \OOO_X) \to H^0(X, \OOO_X(H)) \to
H^0(H_1, \OOO_{H_1}(-K_{H_1}))\to 0
\]
where by adjunction we have $H|_{H_1}\sim -K_{H_1}$. It follows that there exists an element $H_2\in |H|$ such that $H_1\cap
H_2=C$. Note that in the linear system $|H-C|$ of hyperplane sections passing through $C$, there exists a smooth element $H_1$. It follows that a general element $H_2$ in this linear system is smooth as well. Put $D = H_1+H_2$. The pair $(X, D)$ is lc by inversion of
adjunction on $H_1$. 
Applying Lemma \ref{lem-nodal-curve} we conclude that
$\coreg(X)=0$.
\end{proof}

\section{Fano threefolds of the main series} 
\label{sec-prime-Fano-threefolds}
In this section, we work in the following setting. Let $X$ be a smooth Fano threefold with $\rho(X)=1$ and $i(X)=1$. They are called Fano threefolds of the main series. In what follows, by the \emph{genus} of a smooth Fano threefold $X$ of the main series we mean the number 
\[
g(X) = h^0(X, \oo(-K_X))-2.
\] 
It is known that $2\leq g\leq 12$, $g\neq 11$. By \cite[Proposition 1 and Corollary 1]{Isk89}, for any line on $X$, that is for any smooth rational curve $L$ such that $L\cdot (-K_X)=1$, we have either 
\[
N_{L/X}=\oo\oplus\oo(-1),\quad \quad \text{or} \quad \quad N_{L/X}=\oo(-1)\oplus\oo(2).
\] 
If for all lines on $X$ the second possibility is realized, then $X$ is called \emph{exotic}. It is known that for $g(X)\geq 9$ there exists a unique exotic Fano threefold $X$, namely, the Mukai-Umemura example \cite{MU82}. Moreover, a general smooth Fano threefold of the main series is not exotic, cf. \cite[Theorem 4.2.7]{IP99}. It follows that for a general smooth Fano threefold of the main series, we can always find a line with a normal bundle of the form $\oo\oplus\oo(-1)$. 

In the remainder of this section, we consider Fano threefolds of the main series with $g(X)\geq 6$. By $|H|$ we will denote the linear system of the hyperplane sections in the anti-canonical embedding of $X$, that is, $H\sim -K_X$.


\begin{lemma}
\label{lemma:1-5}
Let $X$ be a general Fano threefold in the family \textnumero\,$1.5$.
Then $\coreg(X)\leq 1$.
\end{lemma}
\begin{proof}
Let $X$ be a general Fano threefold in the family $1.5$. It is known that $X$ can be realized as a section of the Grassmannian $\mathrm{Gr}(2,5)\subset\mathbb{P}^9$ by a subspace of codimension $2$ and a quadric. We have $g(X)=6$ and $X\subset \mathbb{P}^7$. Let $L\subset X$ be a general line on $X$, so $N_{L/X}=\oo\oplus \oo(-1)$. 
By \cite[Theorem 5.3]{Isk78} the variety $X$ is birational to a singular threefold which is the intersection of a quadric and a cubic hypersurfaces in $\mathbb{P}^5$. More precisely, there exists the following diagram 
\[
\begin{tikzcd}
Y \ar[rrd, "\alpha"] \ar[d, "\phi"] &  & \ \\
X = X_{10} \subset
\mathbb{P}^{7} \ar[rr, dashed, "\beta"] & & V_{2\cdot 3} \subset \mathbb{P}^5
\end{tikzcd}
\]
where 
\begin{itemize}
\item
$\phi$ is a blow up of a line $L\subset X$, 
\item
the birational map $\beta$ is given by the linear system $|H-L|$, 
\item
the birational morphism $\alpha$ is given by the linear system $|H-E|$ where $E\simeq \mathbb{F}_1$ is the $\phi$-exceptional divisor. Moreover, $\alpha$ is a small contraction, and contracts precisely the strict transforms of $11$ lines on $X$ intersecting $L$, cf. \cite[6.5]{Ma83}.
\end{itemize}
The image $V=V_{2\cdot 3}$ of $\alpha$ is the intersection of a quadric and a cubic, and it has $11$ ordinary double points $p_1,\ldots, p_{11}$ as singularities. Moreover, $V$ contains the isomorphic image $F=F_3$ of $E$, and $F$ is a surface of degree $3$ in $\mathbb{P}^4\subset \mathbb{P}^5$ that passes through $p_1,\ldots, p_{11}$. 

Let $x_0, \ldots, x_5$ be the coordinates on $\mathbb{P}^5$. We may assume that surface $F\subset \mathbb{P}^4=\{x_5=0\}$ is given by the equations
\[
Q_1 = x_0 x_3 - x_1 x_2 = 0, \quad \quad \quad \quad Q_2 = x_0 x_4 - x_2^2 = 0, \quad \quad \quad \quad Q_3 = x_1 x_4 - x_2 x_3 = 0
\]
and $V$ is given by the equations
\begin{equation}
\label{eq-3-quadric-cubic}
Q_1 L_1 + Q_2 L_2 + x_5 Q= 0,\quad \quad \quad \quad
Q_3 + x_5 L = 0
\end{equation}
where $L_{i}$ and $L$ are linear forms in coordinates $x_0,\ldots, x_5$ and $Q$ is a quadratic form in $x_0,\ldots, x_5$. 

Vice versa, a general smooth Fano threefold $X_{10}\subset \mathbb{P}^7$ can be obtained in the following way. Start from an intersection of a quadric and a cubic $V\subset \mathbb{P}^5$ given by the equations \eqref{eq-3-quadric-cubic}. Then consider its small resolution $\alpha\colon Y\to V$ such that for the strict transform $E$ of $F$ one has $\alpha(E)\simeq F$. Then one can blow down $E$ to obtain a morphism $\phi\colon Y\to X$ which is the desired Fano threefold.

Consider a hyperplane section $D_V$ of $V$ given by $x_5=0$. Note that $D_V$ is reducible. In fact, $D_V=F+R$ where $R$ is a smooth surface of degree $3$ in $\mathbb{P}^4$. 
Consider a birational map $\phi\colon \mathbb{P}^2\dashrightarrow F\subset \mathbb{P}^4$ given by the formula
\[
(y_0: y_1: y_2) \mapsto (y_0^2: y_0y_2: y_0y_1: y_1y_2: y_1^2).
\]
Let $q_0=(0:0:1)$ be the point in $\mathbb{P}^2$ where $\phi$ is not defined. 
Note that $R$ restricts to $F$ as an element of the linear system $|-K_{F} - e|$ where $e$ is the $(-1)$-curve on $F\simeq \mathbb{F}_1$. One checks that $C=\phi^{-1}_*(R|_F)$ is a cubic curve in $\mathbb{P}^2$ given by the equation
\[
y_2 L_1 + y_1 L_2 = 0.
\]
In particular, for general $L_i$ the curve $C$ is smooth. Consider the pair $(V, D_V)$ where $D_V = F + R$. Then the pair $(V, D_V)$ is lc. Let $(Y, D_Y)$ be its strict transform on $Y$. It follows that $\dim \mathcal{D}(Y, D_Y)=1$, and we conclude that $\mathrm{coreg}(Y)=\mathrm{coreg}(X)\leq 1$.
\end{proof}

Note that, in the above proof, taking the forms $L_i$ to be special, namely such that the curve $C$ is a rational curve with one node and using Lemma \ref{lem-nodal-curve}, we can obtain Fano threefolds in the family $1.5$ with $\mathrm{coreg}(X)=0$.

\begin{lemma}
\label{lemma:1-6}
Let $X$ be a general Fano threefold in the family \textnumero\,$1.6$.
Then $\coreg(X)=0$.
\end{lemma}
\begin{proof}
Let $X$ be a general Fano threefold in the family $1.6$. Then $g(X)=7$ and $X\subset \mathbb{P}^8$. By \cite[Theorem 6.1(vii)]{Isk78} the variety $X$ is birational to a singular threefold which is the intersection of $3$ quadric hypersurfaces in $\mathbb{P}^6$. More precisely, there exists the following diagram 
\[
\begin{tikzcd}
Y \ar[rrd, "\alpha"] \ar[d, "\phi"] &  & \ \\
X = X_{12} \subset
\mathbb{P}^{8} \ar[rr, dashed, "\beta"] & & V_{2\cdot 2\cdot 2} \subset \mathbb{P}^6
\end{tikzcd}
\]
where 
\begin{itemize}
\item
$\phi$ is a blow up of a line $L\subset X$, 
\item
$\alpha$ is a small contraction, and it contracts precisely the strict transforms of $8$ lines on $X$ passing through $L$,
\item
the map $\beta$ is given by the linear system $|H-L|$, 
\item 
the  morphism $\alpha$ is given by the linear system $|H-E|$ where $E\simeq \mathbb{F}_1$ is the $\phi$-exceptional divisor. 
\end{itemize}
The image of $\alpha$ is the intersection of $3$ quadrics $V=V_{2\cdot 2\cdot 2}$ which has $8$ ordinary double points $p_1,\ldots, p_8$ as singularities. Moreover, $V$ contains the isomorphic image $F=F_3$ of $E$, and $F$ is a surface of degree $3$ in $\mathbb{P}^4\subset \mathbb{P}^6$ that passes through $p_1,\ldots, p_8$. 

Let $x_0, \ldots, x_6$ be the coordinates on $\mathbb{P}^6$. We may assume that surface $F$ is given by the equations
\[
Q_1=x_0 x_3 - x_1 x_2 = 0, \quad  \quad \quad Q_2=x_0 x_4 - x_2^2 = 0, \quad \quad \quad Q_3=x_1 x_4 - x_2 x_3 = 0.
\]
Then $V$ is given by the equations
\begin{equation}
\label{eq-3-quadrics}
Q_1 + x_5 L_{11} + x_6 L_{12} = 0, \quad \quad \quad \quad
Q_2 + x_5 L_{21} + x_6 L_{22} = 0, \quad \quad \quad \quad
Q_3 + x_5 L_{31} + x_6 L_{32} = 0
\end{equation}
where $L_{ij}$ are linear forms in coordinates $x_0,\ldots, x_6$. Vice versa, a general smooth Fano threefold $X_{12}\subset \mathbb{P}^8$ can be obtained in the following way. Start from an intersection of three quadrics $V\subset \mathbb{P}^6$ given by the equations \eqref{eq-3-quadrics}. Then consider its small resolution $\alpha\colon Y\to V$ such that for the strict transform $E$ of $F$ one has $\alpha(E)\simeq F$. Then one can blow down $E$ to obtain a morphism $\phi\colon Y\to X$ which is the desired Fano threefold.

Consider a pencil of hyperplane sections $\mathcal{H}$ spanned by $x_5$ and $x_6$ on $V$. Elements in $\mathcal{H}$ on $V$ are reducible. In fact, each such element is the union of $F$ and a surface $R$ of degree $5$. Let $\mathcal{R}$ be the pencil of such residual surfaces. One checks that $R\in \mathcal {R}$ restricts to $F$ as an element of the linear system $|-K_{F}-\sum p_i|$ on $F\simeq \mathbb{F}_1$. 
Consider a birational map $\phi\colon \mathbb{P}^2\dashrightarrow F\subset \mathbb{P}^4$ given by the formula
\[
(y_0: y_1: y_2) \mapsto (y_0^2: y_0y_2: y_0y_1: y_1y_2: y_1^2).
\]
One checks that $\mathcal{C}=\phi^{-1}_*(\mathcal{R}|_F)$ is a pencil of cubic curves in $\mathbb{P}^2$ spanned by the curves $C_1$ and $C_2$ given by the equations
\[
y_0 L_{31} + y_1 L_{11} - y_2 L_{21} = 0 \quad \quad \quad \text{and} \quad \quad \quad
y_0 L_{32} + y_1 L_{12} - y_2 L_{22} = 0.
\]
Moreover, the base points of $\mathcal{C}$ are precisely the points $q_0, q_1,\ldots q_8$ in $\mathbb{P}^2$ where $q_0=(0:0:1)$ is the point where $\phi$ is not defined and $\phi(q_i) = p_i$ for $1\leq i\leq 8$. For general $L_{i, j}$, the degenerate curves in this linear systems are irreducible rational curves with one node.

Consider the pair $(V, D_V)$ where $D_V = F + R$ and $R\in \mathcal{R}$ is such an element that $R|_F$ is an irreducible curve with one node, and this node is disjoint from $p_1,\ldots p_8$. Then the pair $(V, D_V)$ is lc by inversion of adjunction on $F$. Let $(Y, D_Y)$ be its strict transform on $Y$. By Lemma \ref{lem-nodal-curve} wee see that $\dim \mathcal{D}(Y, D_Y)=2$. We conclude that $\mathrm{coreg}(Y)=\mathrm{coreg}(X)=0$.
\end{proof}

\begin{lemma}
\label{lemma:1-7}
Let $X$ be a general Fano threefold in the family \textnumero\,$1.7$.
Then $\coreg(X)=0$.
\end{lemma}
\begin{proof}
Let $X$ be a general Fano threefold in the family $1.7$. Then $g(X)=8$ and $X\subset \mathbb{P}^{9}$. Since $X$ is general, we may assume that it is not exotic. Consider a general line $L$ on $X$. It is known that $N_{L/X} = \oo\oplus\oo(-1)$. Then exactly $6$ other lines $L_i$ intersect $L$. Moreover, $N_{L_i/X} = \oo\oplus\oo(-1)$. By 
\cite[Theorem 4.3.3]{IP99}, the variety $X$ can be included in the following diagram 
\[
\begin{tikzcd}
Y \ar[rr, dashed, "\alpha"] \ar[d, "\phi"] & & Y' \ar[d, "\psi"] \\
X = X_{14} \ar[rr, dashed, "\beta"] & & \mathbb{P}^2
\end{tikzcd}
\]
where 
\begin{itemize}
\item
$\phi$ is a blow up of a line $L\subset X$ with the $\phi$-exceptional divisor $E$, 
\item
the map $\beta$ is given by the linear system $|H-2L|$ where $H$ is a hyperplane section, 
\item
$\psi$ is a standard conic bundle with the discriminant curve $\Delta\subset \mathbb{P}^2$ of degree $5$,
\item
the map $\alpha$ is the composition of $6$ Atiyah flops with the centers in $L'_i$ where $L'_{i}$ are the strict transforms of $L_i$ on $Y$. 
\end{itemize}
One has $E\simeq \mathbb{F}_1$. Put $E'=\alpha(E)$.  Also,  
\[
(H_Y-2E)|_E\sim 2s+3f\sim -K_E
\]
where $H_Y$ is a strict transform of $H$ on $Y$, $f$ is a ruling and $s$ is a unique $(-1)$-curve on the Hirzebruch surface $E$. 
Note that the restriction $\psi|_{E'}\colon E'\to \mathbb{P}^2$ is a double cover. 

By the Hurwitz formula, the ramification curve $\Gamma_4\subset \mathbb{P}^2$ of $\psi|_{E'}$ has degree $4$. Since $E'$ is smooth, $\Gamma_4$ is smooth as well. Consider a general line $l\subset\mathbb{P}^2$ that is tangent to $\Gamma_4$. Then the preimage $l_{E'}=\psi|_{E'}^{-1}(l)$ is a rational curve with one node. Put $H'_l=\psi^{-1}(l)$. Note that for a general $l$ the surface $H'_l$ is smooth since $l$ intersects the discriminant curve $\Delta$ of the standard conic bundle $\psi$ transversally. Observe that $H'_l + E' \sim -K_{Y'}$. Since both $H'_l$ and $E'$ are smooth and $l_{E'}=H'_l\cap E'$ is a nodal curve, by inversion of adjunction the pair $(Y', H'_l + E)$ is lc. Apply Lemma \ref{lem-line-conic} to the pair $(Y, F' + H'_1)$ to conclude that $\dim \mathcal{D}(Y, H'_l + E')=2$, and hence $\dim \mathcal{D}(Y, H_l + E)=2$ where $H_l$ is the strict transform of $H'_l$ on $Y$. Consequently, $\mathrm{coreg}(X)=0$.
\end{proof}

\begin{lemma}
\label{lemma:1-8}
Let $X$ be a general Fano threefold in the family \textnumero\,$1.8$.
Then $\coreg(X)=0$.
\end{lemma}
\begin{proof}
Let $X$ be a general Fano threefold in the family $1.8$. Then $g(X)=9$ and $X\subset \mathbb{P}^{10}$. Since $X$ is general, we may assume that it is not exotic. By
\cite[Theorem 4.3.7(iii)]{IP99} the variety $X$ can included in the following diagram 
\[
\begin{tikzcd}
Y \ar[rr, dashed, "\alpha"] \ar[d, "\phi"] &  & Y' = \mathrm{Bl}_\Gamma\mathbb{P}^3 \ar[d, "\psi"]  \\
X = X_{16} \subset
\mathbb{P}^{10} \ar[rr, dashed, "\beta"] & & \mathbb{P}^3 
\end{tikzcd}
\]
where 
\begin{itemize}
\item
$\phi$ is the blow up of a line on $X$, 
\item
$\alpha$ is a composition of flops, 
\item
the map $\beta$ is given by the linear system $|H-2L|$ where $H$ is a hyperplane section, 
\item 
$\psi$ is blow up of a smooth non-hyperelliptic curve $\Gamma\subset\mathbb{P}^3$ of degree $7$ and genus $3$. 
\end{itemize}
It is known that $\Gamma$
lies on a unique irreducible cubic
surface $F=F(\Gamma)$ with at worst du Val singularities provided that the line on $X$ is general. 
Consider the
following boundary on $\mathbb{P}^3$: $D = F + H_1$ where $H_1$ is a
plane that intersects $F$ in a union of a line $L$ and a conic $C$ that intersect transversally. By inversion of adjunction the $(\mathbb{P}^3, D)$ is lc.  We have
\begin{equation}
\psi^*(K_{\mathbb{P}^3} + F + H_1) = K_{Y'} + F' + H'_1 \sim 0, \quad \ K_{Y} + F'' + H''_1 \sim 0, \quad \ \phi_*(K_{Y} + F'' + H''_1) = K_X + H \sim 0
\end{equation}
where $F'$ and $H'_1$ are the strict transform of $F$ and $H_1$ on $Y'$, $F''$ and $H''_1$ are their strict transforms on~$Y$. It follows that the pair $(X, H)$ is lc. Apply Lemma \ref{lem-line-conic} to the pair $(Y, F' + H'_1)$ to conclude that $\dim \mathcal{D}(Y, F' + H'_1)=2$. Consequently, $\mathrm{coreg}(X)=0$.
\end{proof}

\begin{lemma}
\label{lemma:1-9}
Let $X$ be a general Fano threefold in the family \textnumero\,$1.9$.
Then $\coreg(X)=0$.
\end{lemma}
\begin{proof}
Let $X$ be a general Fano threefold in the family $1.9$. Then $g(X)=10$ and $X\subset \mathbb{P}^{11}$. By
\cite[Theorem 4.3.7(ii)]{IP99} the variety $X$ can included in the following diagram 
\[
\begin{tikzcd}
Y \ar[rr, dashed, "\alpha"] \ar[d, "\phi"] &  & Y' = \mathrm{Bl}_\Gamma Q \ar[d, "\psi"]  \\
X = X_{18} \subset
\mathbb{P}^{11} \ar[rr, dashed, "\beta"] & & Q
\end{tikzcd}
\]
where 
\begin{itemize}
\item
$\phi$ is the blow up of a line on $X$, 
\item
$\alpha$ is a composition of flops, 
\item
the map $\beta$ is given by the linear system $|H-2L|$ where $H$ is a hyperplane section, 
\item 
$\psi$ is the blow up of a smooth curve $\Gamma\subset Q$ of degree $7$ and genus $2$ on  a smooth quadric $Q\subset\mathbb{P}^4$. 
\end{itemize}
It is known that $\Gamma$ lies on a unique irreducible surface $F$ such that $F\sim 2H_Q$ where $H_Q= \OOO_Q(1)$. 
Consider the
following boundary on $Q$: $D = F + H_1$ where $H_1$ is a
hyperplane section that intersects $F$ in a union of a line $L$ and a twisted cubic $C$. Let $F'$ and $H'_1$ be the strict transform of $F$ and $H_1$ on $Y'$, and let $Y''$ and $H''_1$ be their strict transforms on $Y$. Apply Lemma \ref{lem-line-conic} to the pair $(Y, F'' + H''_1)$ to conclude that $\dim \mathcal{D}(Y, F' + H'_1)=2$. Consequently, $\mathrm{coreg}(X)=0$.
\end{proof}

\begin{lemma}
\label{lemma:1-10}
Let $X$ be a general Fano threefold in the family \textnumero\,$1.10$.
Then $\coreg(X)=0$.
\end{lemma}
\begin{proof}
Let $X$ be a Fano threefold in the family $1.10$ different from the Mukai-Umemura example, so it is not exotic. Then $g(X)=12$ and $X\subset \mathbb{P}^{13}$. By
\cite[Theorem 4.3.7(i)]{IP99} the variety $X$ can included in the following diagram 
\[
\begin{tikzcd}
Y \ar[rr, dashed, "\alpha"] \ar[d, "\phi"] &  & Y' = \mathrm{Bl}_\Gamma V_5 \ar[d, "\psi"]  \\
X = X_{22} \subset
\mathbb{P}^{13} \ar[rr, dashed, "\beta"] & & V_5
\end{tikzcd}
\]
where 
\begin{itemize}
\item
$\phi$ is the blow up of a line on $X$, 
\item
$\alpha$ is a composition of flops, 
\item
the map $\beta$ is given by the linear system $|H-2L|$ where $H$ is a hyperplane section, 
\item 
$\psi$ is blow up of a smooth rational curve $\Gamma\subset V_5$ of degree $5$ where $V_5$ is a del Pezzo threefold of degree $5$. 
\end{itemize}
It is known that $\Gamma$ lies on a unique irreducible surface $F$ such that $F\sim H_{V_5}$ where $H_{V_5}=\OOO_{V_5}(1)$. 
Consider the boundary $D$ on $V_5$ where $D = F + H_1$ and $H_1$ is a
hyperplane section that intersects $F$ in a union of a line $L$ and a smooth rational curve $C$ of degree $4$. Let $F'$ and $H'_1$ be the strict transform of $F$ and $H_1$ on $Y'$, and let $Y''$ and $H''_1$ be their strict transforms on $Y$. Apply Lemma \ref{lem-line-conic} to the pair $(Y, F' + H'_1)$ to conclude that $\dim \mathcal{D}(Y, F' + H'_1)=2$. Consequently, $\mathrm{coreg}(X)=0$.
\end{proof}


\section{Sextic double solid}
\label{sec-sextic-double-solid}
In this section, we consider smooth Fano threefolds in the family \textnumero\,1.1. Our goal is to prove the following 
\begin{proposition}
\label{cor-gen-sextic-double-solid}
Let $X$ be a general sextic double solid, that is, a smooth Fano threefold in the family \textnumero\,1.1. Then $\mathrm{coreg}(X)\geq 1$.
\end{proposition}

Consider the product $\mathcal{S} = \mathbb{P}^n\times \mathbb{P}\mathrm{H}^0(\mathbb{P}^n, \OOO_{\mathbb{P}^n}(d))$ together with the natural projections $p\colon \mathcal{S} \to \mathbb{P}^n$ and $q\colon \mathcal{S} \to \mathbb{P}\mathrm{H}^0(\mathbb{P}^n, \OOO_{\mathbb{P}^n}(d))$. We have $\dim \mathbb{P}\mathrm{H}^0(\mathbb{P}^n, \OOO_{\mathbb{P}^n}(d))={n+d \choose n}-1$. 
For an element $F\in \mathrm{H}^0(\mathbb{P}^n, \OOO_{\mathbb{P}^n}(d))$, put $X_F = \{ F = 0 \}\subset \mathbb{P}^n$.
Define the incidence subvariety
\[
\mathcal{I}=\{ (P, F)\in\mathcal{S}\ |\ F(P)=0\ \text{and}\ X_F\ \text{is smooth}\,\}.
\]
Note that the fiber $p|_{\mathcal{I}}^{-1}(P)$ has dimension ${n+d \choose n}-2$.

\begin{lemma}
\label{lem-codimention-n}
In the notation as above, let $S\subset \mathcal{I}$ be a non-empty algebraic subset.
Assume that for any point $P\in \mathbb{P}^n$ the subset $p|_{S}^{-1}(P)$ has codimension at least $n$ in $p|_{\mathcal{I}}^{-1}(P)$. Then $q(S)$ is a proper subset in $\mathbb{P}\mathrm{H}^0(\mathbb{P}^n, \OOO_{\mathbb{P}^n}(d))$. 
\end{lemma}
\begin{proof}
We have $\dim p|_{S}^{-1}(P) \leq {n+d \choose n}-2-n$. It follows that $\dim S\leq {n+d \choose n}-2$. Thus, $\dim q(S)\leq {n+d \choose n}-2$, and $q(S)$ is a proper subset in $\mathbb{P}\mathrm{H}^0(\mathbb{P}^n, \OOO_{\mathbb{P}^n}(d))$. The claim follows.
\end{proof}

The following simple lemma is well-known.
\begin{lemma}
\label{lem-section-double-points}
For a general hypersurface $S$ of degree $d$ in $\mathbb{P}^n$ with $n\geq 2$, any hyperplane section $H$ of~$S$ has at worst isolated singular points of multiplicity $2$.
\end{lemma}
\begin{proof}
We may assume that $S$ is smooth. For $d=1$ or $n=2$, the statement is clear. So we assume that $d\geq 2$ and $n\geq 3$.
 We may assume that in an affine chart $\mathbb{C}^n$ near a point $P\in S$ the equation of $X$ has the form 
\[
f (x_1,\ldots, x_n) = x_n + f_2(x_1,\ldots, x_n) +\ldots + f_d(x_1,\ldots, x_n)= 0
\]
where $f_i$ are homogeneous polynomials of degree $i$. Then a hyperplane section $H\cap S$ is singular at $S$ if and only if $H$ is given by the equation $x_n=0$. Observe that
the vanishing of $f_2(x_1,\ldots, x_{n-1}, 0)$ is a condition of codimension $n(n-1)/2$ for the coefficients of $f$. Note that for $n\geq 3$ we have $n(n-1)/2\geq n$, so we can apply Lemma \ref{lem-codimention-n} to conclude that for a general $S$ and for any plane $H$ in $\mathbb{P}^n$, the intersection $H\cap S$ has at worst isolated singular points of multiplicity $2$.
\end{proof}

\begin{lemma}
\label{prop-sextic-reg1}
Let $X$ be a general sextic double solid. Then $\mathrm{reg}_1(X)=0$.
\end{lemma}
\begin{proof}
The anti-canonical linear system $|-K_X|$ defines a double cover $\pi\colon X\to \mathbb{P}^3$ ramified in a smooth surface $S\subset \mathbb{P}^3$ of degree $6$. Since we assume that $X$ is general it follows that $S$ is general as well. For a divisor $D\in |-K_X|$ we have that $H=\pi(D)$ is a plane in $\mathbb{P}^3$. By Lemma \ref{lem-section-double-points}, for a general $S$ the intersection $R_H = H\cap S$ has at worst isolated singular points of multiplicity $2$. Consider a double cover $\pi|_{D}\colon D\to H$ which is ramified along $R_H$. The Hurwitz formula  
yields 
\[
K_X + D = \pi^*(K_{\mathbb{P}^3} + 1/2 S + H ) \sim 0. 
\]
By adjunction,
\[
K_D = \pi^*(K_H + 1/2 S|_H). 
\]
Since $S|_H$ is a reduced (cf. \cite{Ish82}) curve with at worst double points, by Lemma \ref{lem-sum-lct} we have that the pair $(H, 1/2 S|_H)$ is klt, hence $(D, 0)$ is klt (see e.g. \cite[Proposition 20.3]{CKM88}). By inversion of adjunction, $(X, D)$ is plt. This shows that the dual complex is a point, and $\mathrm{reg}_1(X)=0$ for a general $X$.
\end{proof}

\begin{lemma}
\label{prop-sextic-reg2}
Let $X$ be a general sextic double solid. Then $\mathrm{reg}_2(X)=0$.
\end{lemma}
\begin{proof}
Since $\mathrm{reg}_1(X)=0$, we have $\mathrm{reg}_2(X)\geq 0$.  
Let $Q$ be a quadric in $\mathbb{P}^3$. In the notation as in the proof of Lemma \ref{prop-sextic-reg1}, write 
\[
K_X + 1/2 D = \pi^*(K_{\mathbb{P}^3} + 1/2 S + 1/2Q )\sim 0
\]
where $D=\pi^{-1}(Q)$. 
First assume that $Q$ is reducible, so $Q=H_1+H_2$ where $H_1\neq H_2$. 
Then $D=D_1+D_2$ where $\pi(D_i)=H_i$. By Lemma \ref{prop-sextic-reg1}, we know that $(X, D_i)$ is plt for $i=1,2$. Hence $(X, 1/2D)$ is plt as well. This shows that the dual complex of this pair is empty.

Now assume that $Q$ is an irreducible quadric. It is enought to show that the pair $(\mathbb{P}^3, 1/2S + 1/2Q)$ is klt. Since $S$ is smooth, the pair $(\mathbb{P}^3, S)$ is plt. One easily checks that the pair $(\mathbb{P}^3, Q)$ is plt is well. Then the pair $(\mathbb{P}^3, 1/2S+1/2Q)$ is plt, and hence klt. This shows that the dual complex of this pair is empty. Consequently, $\mathrm{reg}_2(X)=0$.
\end{proof}

The proof of Proposition \ref{cor-gen-sextic-double-solid} follows from Lemmas \ref{prop-sextic-reg1}, \ref{prop-sextic-reg2} and Theorem \ref{thm-1-2-complements}.

\section{Quartic hypersurface}
\label{sec-quartic}
In this section, we analyze smooth Fano threefolds in the family \textnumero\,1.2. A general element in this family is a quartic hypersurface in $\mathbb{P}^4$. Our goal is to prove the following 

\begin{proposition}
\label{quartic-coreg-0}
For a general quartic hypersurface $X$ in $\mathbb{P}^4$, 
we have $\mathrm{coreg}(X)\geq 1$.
\end{proposition}

\begin{lemma}
\label{lem-no-cusps}
For a general quartic hypersurface $X$ in $\mathbb{P}^4$, 
we have $\mathrm{reg}_1(X)=0$.
\end{lemma}
\begin{proof}
Let $H$ be a hyperplane section of $X$. By \cite{Ish82}, $H$ is normal. 
We show that for a general $X$, the surface $H$ cannot be strictly lc and hence $\mathrm{reg}_1(X)=0$. Fix a point $P\in X$. 
First of all, note that $H$ may be singular at the point $P\in X$ only if the tangent space $T_P X$ coincides with the hyperplane that corresponds to $H$. We may assume that the point $P$ belongs to the chart $x_0\neq 0$, and that $T_P X$ is given by the equation $x_4=0$. Then $X$ locally near $P$ is given by the equation
\begin{equation}
0 = f(x_1, \ldots, x_4) = x_4 + f_2(x_1, \ldots, x_4) + f_3(x_1, \ldots, x_4) + f_4(x_1, \ldots, x_4)
\end{equation}
where $\deg f_i=i$. 
By Lemma \ref{lem-striclty-lc-1-complement} the points on $X$ with the condition that $H=T_P X\cap X$ is strictly lc are such that 
 one of the following two conditions is satisfied: 
\begin{enumerate}
\item
$f_2(x_1, x_2, x_3, 0) = 0$, 
\item
$f_2 = l(x_1, x_2, x_3)^2$ where $l(x_1, x_2, x_3)$ is a linear form, and if we assume that $l=x_3$, then either $f_3(x_1, x_2, 0, 0)=0$, or $f_3(x_1, x_2, 0, 0)=m(x_1, x_2)^3$.
\end{enumerate}
Note that 
the first of the above conditions is a condition of codimension $6$ in the space of quartic hypersurfaces that contain $P$, while the second condition is of codimension at least $4$. 
By Lemma \ref{lem-codimention-n} we conclude that on a general quartic threefold, there are no such points. 
Thus, $H$ is plt, hence the pair $(X, H)$ is plt, and $\mathrm{reg}_1(X)=0$, which completes the proof.
\end{proof}

\begin{lemma}
\label{lem-no-cusps2}
For a general quartic hypersurface $X$ in $\mathbb{P}^4$, we have $\reg_2(X)=0$.
\end{lemma}
\begin{proof}
Since $\mathrm{reg}_1(X)=0$, we have $\mathrm{reg}_2(X)\geq 0$.  
Consider the pair $(X, 1/2D)$ where $D\in |-2K_X|$. Assume first
that $D=D_1+D_2$ is reducible where both $D_i$ are hyperplane
sections with $D_1\neq D_2$. 
By Lemma \ref{lem-no-cusps}, the pairs $(X, D_i)$ are plt for $i=1,2$. Hence the pair $(X, 1/2D)$ is plt as well, so the dual complex of the pair $(X, 1/2D)$ is empty.



Thus we may assume that $D$ is irreducible and reduced, and the pair $(X, 1/2D)$ is lc.  
In other words, $D$ is a restriction of a quadric $Q$ in $\mathbb{P}^4$ whose rank is at least $3$. 
Note that the intersection $X\cap Q$ is reduced. Indeed, if not, we have $D =a D'$ where $D'$ is a reduced divisor on $X$. Since $\mathrm{Pic}(X)\simeq \mathbb{Z}$, we have $a=2$ and $D'$ is a hyperplane section of $X$. Restrict a quadric  to the hyperplane $H$ such that $H\cap X=D$ to see that this is absurd. 

Let $X$ locally near $P$ be given by the equation
\begin{align*}
\label{eq-X}
0 = f(x_1, \ldots, x_4) =\ &x_4 + f_2(x_1, \ldots, x_4) + f_3(x_1, \ldots, x_4) + f_4(x_1, \ldots, x_4) \\
=\ &x_4 + 
f_{2,0} x_4^2 + x_4 f_{2,1}(x_1, x_2, x_3) + f_2(x_1, x_2, x_3, 0) \\
&+ f_{3,0} x_4^3 + x_4^2 f_{3,1}(x_1, x_2, x_3) + x_4 f_{3,2}(x_1, x_2, x_3) + f_3(x_1, x_2, x_3, 0) \\
&+f_{4,0} x_4^4 + x_4^3 f_{4,1}(x_1, x_2, x_3) + x_4^2 f_{4,2}(x_1, x_2, x_3) + x_4 f_{4,3}(x_1, x_2, x_3) + f_4(x_1, x_2, x_3,0)
\end{align*} 
where $f_j$ and $f_{i,j}$ are homogeneous polynomials of degree $j$. 
From this equation, we can express $x_4$ as follows:
\begin{align}
  -x_4=&f_2(x_1, x_2, x_3, 0) + f_3(x_1, x_2, x_3, 0) - f_{2,1}(x_1, x_2, x_3) f_2(x_1, x_2, x_3, 0) \nonumber\\ 
  +& f_{2,1}(x_1, x_2, x_3) f_3(x_1, x_2, x_3, 0) + f_{2,1}(x_1, x_2, x_3)^2 f_2(x_1, x_2, x_3, 0) - c_2 f_2(x_1, x_2, x_3, 0)^2 \nonumber\\ 
  -& f_2(x_1, x_2, x_3, 0) f_{3,2}(x_1, x_2, x_3) + f_4(x_1, x_2, x_3, 0) \nonumber\\ 
  +& \text{(terms of degree $\geq 5$)}. 
  \end{align}

\subsection{
Assume first that $D$ is singular at $P$}
Let $D$ locally near $P$ be given by the equation
\begin{equation}
\label{eq-quadric}
0 = g_2(x_1, \ldots, x_4) = g_{2,0} x_4^2 + x_4 g_{2,1}(x_1, x_2, x_3) + g_2(x_1, x_2, x_3, 0),
\end{equation}
where $g_j$ and $g_{2,j}$ are homogenous polynomials of degree $j$. Note that $g_2(x_1, x_2, x_3, 0)\neq 0$. Indeed, otherwise $D$ is reducible, which contradicts to our assumption. Put the expression of $x_4$ from \eqref{eq-expression-for-x4} into \eqref{eq-quadric} to obtain the equation of $D$ on $X$ near $P$:
\begin{align*}
0 =\ \widetilde{g}(x_1, x_2, x_3) =\ & g_2(x_1, x_2, x_3, 0) 
-  f_2(x_1, x_2, x_3, 0) g_{2,1}(x_1, x_2, x_3)  \\
&+  g_{2,0} f_2(x_1, x_2, x_3, 0)^2  + f_{2,1}(x_1, x_2, x_3) f_2(x_1, x_2, x_3,0)g_{2,1}(x_1, x_2, x_3) \\
&- f_3(x_1, x_2, x_3,0)g_{2,1}(x_1, x_2, x_3) +  \text{(terms of degree $\geq 4$)}.
\end{align*}
In particular, the quadratic term in $\widetilde{g}(x_1, x_2, x_3)$ is non-zero. Then by Lemma \ref{lem-sum-lct} the pair $(X, 1/2D)$ is klt, and hence the dual complex is empty.

\subsection{
Now assume that $D$ is smooth at $P$} 
Let $D$ locally near $P$ be given by
\begin{equation}
\label{eq-quadric2}
0 = g_2(x_1, \ldots, x_4) = x_4 + g_{2,0} x_4^2 + x_4 g_{2,1}(x_1, x_2, x_3) + g_2(x_1, x_2, x_3, 0)
\end{equation}
where $g_j$ and $g_{2,j}$ are homogenous polynomials of degree $j$. 
Put the expression of $x_4$ from \eqref{eq-quadric2} into \eqref{eq-expression-for-x4} to obtain the equation of $D$ on $X$ near $P$:
\begin{align*}
0 =\ & \widetilde{g}(x_1, x_2, x_3) = f_2(x_1, x_2, x_3, 0) - g_2(x_1, x_2, x_3, 0) \\
&+ f_3(x_1, x_2, x_3, 0) - f_{2,1}(x_1, x_2, x_3) g_2(x_1, x_2, x_3, 0) + g_2(x_1, x_2, x_3, 0) g_{2,1}(x_1, x_2, x_3) \\
&+ f_4(x_1, x_2, x_3,0) + f_{2,0} g_2(x_1, x_2, x_3, 0)^2 - g_2(x_1, x_2, x_3, 0) f_{3,2}(x_1, x_2, x_3)\\
&- g_{2,0} g_2(x_1, x_2, x_3, 0)^2 - g_2(x_1, x_2, x_3, 0)g_{2,1}^2(x_1, x_2, x_3) + f_{2,1}(x_1, x_2, x_3)g_2(x_1, x_2, x_3,0)g_{2,1}(x_1, x_2, x_3)\\
=&\widetilde{g}_2(x_1, x_2, x_3)+\widetilde{g}_3(x_1, x_2, x_3)+\widetilde{g}_4(x_1, x_2, x_3)+\text{(terms of degree $\geq 5$)}
\end{align*}
where $\widetilde{g}_j$ are homogeneous polynomials of degree $j$. We may assume that $\widetilde{g}_2$ is zero, so $g_2(x_1, x_2, x_3, 0) = - f_2(x_1, x_2, x_3, 0)$. Indeed, otherwise arguing as above we see that the dual complex is empty. Now we consider the term $\widetilde{g}_3$. 
Let $C_3$ be a subset of $\mathbb{P}^2$ defined by the equation $\widetilde{g}_3=0$. 
According to Lemma \ref{lem-kuwata} we may assume that $C_3$ belongs to the following list:
\begin{enumerate}
\item
the union of a double line and a line, 
\item
a triple line,
\item
the whole projective plane.
\end{enumerate}
Start with the last case. Then 
\[
f_3(x_1, x_2, x_3, 0) = f_2(x_1, x_2, x_3, 0)( f_{2,1}(x_1, x_2, x_3) - g_{2,1}(x_1, x_2, x_3) ) \\
\]
which is a codimension $7$ condition on the coefficients of $f$, hence for a general $X$ this case does not occur. In the second case, we have a codimension $4$ condition on the coefficients of $f$, hence for a general $X$ this case does not occur as well.

So it remains to consider the first case. In this case, we have a codimension $2$ condition on the coefficients of $f$. After a change of coordinates, we may assume that $\widetilde{g}_3=x_1^2x_2$. We have
\begin{align*}
0 =\ & \widetilde{g}(x_1, x_2, x_3) = x_1^2x_2
+ f_4(x_1, x_2, x_3,0) + f_{2,0} f_2(x_1, x_2, x_3, 0)^2 - f_2(x_1, x_2, x_3, 0) f_{3,2}(x_1, x_2, x_3)\\
&- g_{2,0} f_2(x_1, x_2, x_3, 0)^2 - f_2(x_1, x_2, x_3, 0)g_{2,1}^2(x_1, x_2, x_3) + f_{2,1}(x_1, x_2, x_3)f_2(x_1, x_2, x_3,0)g_{2,1}(x_1, x_2, x_3).
\end{align*}
Moreover, once we fix $\widetilde{g}_3$, we may assume that $g_{2,1}$ is fixed as well. 
Let us rewrite this equation as
\[
\widetilde{g} = x_1^2x_2 + a_0 x_2^4 + a_1 x_2^3 x_3 + a_2 x_2^2 x_3^2 + a_3 x_2 x_3^3 + a_4 x_3^4 + G, \quad \quad a_i \in \mathbb{C}
\]
where $G$ are the sum of higher terms with respect to the weights $w=(3,2,2)$. 
Note that if at least one of the $a_i$ is not equal to zero then by Corolllary \ref{cor-surface-x2y} we are done. On the other hand, the condition that all $a_i$ are equal to $0$ has codimension $5$ on the coefficients of $f$. So this cannot happen for a general $f$. This concludes the proof.
\end{proof}

The proof of Proposition \ref{quartic-coreg-0} follows from Lemmas \ref{lem-no-cusps}, \ref{lem-no-cusps2} and Theorem \ref{thm-1-2-complements}.

\section{Intersection of a quadric and a cubic}
\label{sec-23}
In this section, we consider smooth Fano threefolds in family \textnumero\,1.3. 
Put \[
\mathcal{P} = \mathbb{P}\mathrm{H}^0(\mathbb{P}^5, \OOO(2))\times \mathbb{P}\mathrm{H}^0(\mathbb{P}^5, \OOO(3)).
\]
Consider the product 
$
\mathcal{S} = \mathbb{P}^5\times \mathcal{P}
$
together with the natural projections 
\[
p\colon \mathcal{S} \to  \mathbb{P}^5\quad \text{and}\quad q\colon \mathcal{S} \to \mathcal{P}.
\] 
Put 
$\label{eq-dim-polynomials-on-quadric}
m = \dim \mathcal{P}$.  
For an element $F=(F_1, F_2)\in \mathrm{H}^0(\mathbb{P}^5, \OOO(2))\times \mathrm{H}^0(\mathbb{P}^5, \OOO(3))$ put $X_{F} = \{ F_1 = F_2 = 0 \}\subset \mathbb{P}^5$.
Define the incidence subvariety
\[
\mathcal{I}=\{ (P, F)\in\mathcal{S}\ |\ F_1(P)=F_2(P)=0\ \text{and}\ X_F\ \text{is smooth}\,\}.
\]
Note that $\dim p|_{\mathcal{I}}^{-1}(P) = m - 2$. In this section, we will need the following lemma.

\begin{lemma}
\label{lem-codimention-23}
In the notation as above, let $S\subset \mathcal{I}$ be a non-empty algebraic subset.
Assume that for any point $P\in \mathbb{P}^5$ the subset $p|_{S}^{-1}(P)$ has codimension at least $4$ in $p|_{\mathcal{I}}^{-1}(P)$. Then $q(S)$ is a proper subset in $\mathcal{P}$. 
\end{lemma}
\begin{proof}
We have $\dim p|_{S}^{-1}(P) \leq m - 2 - 4 = m-6$ where $m$ is defined above. It follows that 
\[
\dim q(S)\leq \dim S\leq m - 1.
\] 
Thus, $q(S)$ is a proper subset in $\mathbb{P}\mathrm{H}^0(\mathbb{P}^5, \OOO(2))\times \mathbb{P}\mathrm{H}^0(\mathbb{P}^5, \OOO(3))$. The claim follows.
\end{proof}

Let $f=0$ be the equation of a quadric $Q$, and $g=0$ be the equation of a cubic $Y$ in $\mathbb{P}^5$. Since $X=Q\cap Y$ is general, we may assume that $Q$ is smooth. Note that the equation of a cubic is not defined uniquely by $X$ since we can replace $g$ with $g+lf$ for arbitrary linear form $l$. Let $P\in X$ be a point. Near $P$ we can write
\begin{align}
\label{eq-23-quadric}
0 =& f(x_1, \ldots, x_5) = x_4 + f_2(x_1,\ldots, x_5),\\
\label{eq-23-cubic}
0 =& g(x_1, \ldots, x_5) = x_5 + g_2(x_1,\ldots, x_5) + g_3(x_1,\ldots, x_5).
\end{align}

\begin{proposition}
\label{prop-23-reg1}
For a general $X$, we have $\mathrm{reg}_1(X)=0$.
\end{proposition}
\begin{proof}
We use the above notation. Fix a point $P\in X$.  
Let $H$ be a hyperplane in $\mathbb{P}^5$. If $H$ intersects $T_P X$ transversally, then the pair $(X, H|_X)$ is log smooth at $P$. So we may assume that $H$ contains $T_P X$. Then $H$ is given by the equation $a x_4 + b x_5 = 0$ for $a,b\in\mathbb{C}$, see equations \eqref{eq-23-quadric} and \eqref{eq-23-cubic}. First assume that $H$ is given by the equation $x_4=0$. Then the intersection $X\cap H$ is given by the equations
\begin{align}
\label{eq-23-QT_PQ}
0 =& f(x_1, x_2, x_3, 0, x_5) =  f_2(x_1,x_2, x_3, 0, x_5),\notag\\
0 =& g(x_1, x_2, x_3, 0, x_5) = x_5 + g_2(x_1, x_2, x_3, 0, x_5) + g_3(x_1, x_2, x_3, 0, x_5).
\end{align}
Since $X$ is general, we may assume that $Q$ is smooth. Hence $f_2(x_1,x_2, x_3, 0, x_5)$ defines a quadric of rank~$4$. Express $x_5$ from the second equation and put it into the first equation to obtain the equation of $X\cap H$ near~$P$:
\begin{equation}
\label{eq-23-intersection-first-case}
0 = f_2(x_1, x_2, x_3, 0, 0) + \text{(terms of degree $\geq 3$)}.
\end{equation}
Then $f_2(x_1, x_2, x_3, 0, 0)=0$ defines a conic of rank at least $2$. Thus, by Lemma \ref{lem-corank-second-differenetial} the equation \eqref{eq-23-intersection-first-case} is an equation of a du Val singularity. In particular, the pair $(H|_X, 0)$ is klt, hence the pair $(X, H|_X)$ is plt, and its dual complex is one point. So we are done in this case.

If $H$ is given by the equation $ax_4+x_5 = 0$ for $a\in \mathbb{C}$, 
we have the following equations of $X\cap H$ near $P$: 
\begin{align}
\label{eq-23-QT_PQ'}
0 =& f(x_1, \ldots, x_4, 0) = x_4 + f_2(x_1,\ldots, x_4, -ax_4),\\
\label{eq-23-QT_PQ''}
0 =& g(x_1, \ldots, x_4, 0) = -ax_4 + g_2(x_1,\ldots, x_4, -ax_4 ) + g_3(x_1,\ldots, x_4, -ax_4 ) \notag\\
=& g_2(x_1,x_2,x_3,0,0) + (\widetilde{g_{2,1}}(x_1,x_2,x_3)-a)x_4 + \widetilde{g_{2,0}}x_4^2 + g_3(x_1,\ldots, x_4, 0) + \ldots
\end{align}
where $\widetilde{g_{i, j}}$ are homogeneous polynomials of degree $j$ whose coefficients depend on the coefficients of the polynomials $g_k$. 
Express $x_4$ from \eqref{eq-23-QT_PQ'} and put it into \eqref{eq-23-QT_PQ''} to obtain the equation of $X\cap H$ near~$P$:
\begin{align}
\label{eq-23-intersection-final-case}
0 = g_2(x_1,x_2, x_3, 0, 0) &+ a f_2(x_1,x_2,x_3, 0, 0) - \widetilde{g_{2,1}}(x_1,x_2,x_3) f_2(x_1,x_2,x_3, 0, 0)\notag\\
&+ g_3(x_1,x_2, x_3, 0, 0) + \text{(terms of degree $\geq 4$)}.
\end{align}
If the rank of $g_2$ is at least $2$, then by Lemma \ref{lem-corank-second-differenetial} we have an equation of a du Val singularity, and arguing as above we are done. So we may assume that $g_2$ is a square of a linear form, say, $g_2(x_1,x_2, x_3, 0, 0) =x_1^2$. This is a condition of codimension $3$ on the coefficients of $g$. Consider the third order term of \eqref{eq-23-intersection-final-case} restricted to $x_1=0$: 
\[
- g_{2,1}(0,x_2,x_3) f_2(0,x_2,x_3, 0, 0) + g_3(0,x_2, x_3, 0, 0).
\]
The condition that this term either vanishes, or is a cube of a linear form, is a condition of codimension at least $1$. In total, we get a condition of codimension at least $4$, so this cannot happen for a general $X$ by Lemma \ref{lem-codimention-23}. Then by Lemma \ref{lem-striclty-lc-1-complement}, we conclude that for a general $X$, the pair $(X, H|_X)$ is plt, and its dual complex is one point, so $\mathrm{reg}_1(X)=0$. This finishes the proof.
\end{proof}

\section{Intersection of three quadrics}
\label{sec-222}
In this section, we consider smooth Fano threefolds in family \textnumero\,1.4. 
Put 
\[
\mathcal{P}=\mathbb{P}\mathrm{H}^0(\mathbb{P}^6, \OOO(2))\times \mathbb{P}\mathrm{H}^0(\mathbb{P}^6, \OOO(2))\times \mathbb{P}\mathrm{H}^0(\mathbb{P}^6, \OOO(2)).
\] 
Consider the product 
$
\mathcal{S} = \mathbb{P}^6\times \mathcal{P}
$ 
together with the natural projections 
\[
p\colon \mathcal{S} \to  \mathbb{P}^6\quad \text{and} \quad q\colon \mathcal{S} \to \mathcal{P}
\] 
Put 
$\label{eq-dim-polynomials-on-quadric}
m = \dim \mathcal{P}. 
$
For an element $F=(F_1, F_2, F_3)\in \mathrm{H}^0(\mathbb{P}^6, \OOO(2))\times \mathrm{H}^0(\mathbb{P}^6, \OOO(2))\times \mathrm{H}^0(\mathbb{P}^6, \OOO(2))$, put $X_{F} = \{ F_1 = F_2 = F_3 = 0 \}\subset \mathbb{P}^6$.
Define the incidence subvariety
\[
\mathcal{I}=\{ (P, F)\in\mathcal{S}\ |\ F_1(P)=F_2(P)=F_3(P)=0\ \text{and}\ X_F\ \text{is smooth}\,\}.
\]
Note that $\dim p|_{\mathcal{I}}^{-1}(P) = m - 3$. In this section, we will need the following lemma.

\begin{lemma}
\label{lem-codimention-23}
In the notation as above, let $S\subset \mathcal{I}$ be a non-empty algebraic subset.
Assume that for any point $P\in \mathbb{P}^6$ the subset $p|_{S}^{-1}(P)$ has codimension at least $4$ in $p|_{\mathcal{I}}^{-1}(P)$. Then $q(S)$ is a proper subset in $\mathcal{S}$. 
\end{lemma}
\begin{proof}
We have $\dim p|_{S}^{-1}(P) \leq m - 3 - 4 = m - 7$ where $m$ is defined above. It follows that 
\[
\dim q(S)\leq \dim S\leq m - 1.
\] 
Thus, $q(S)$ is a proper subset in $\mathcal{S}$. The claim follows.
\end{proof}

Let $X$ be a smooth Fano threefold in the family \textnumero\,1.4. Then $X$ is a complete intersection of three quadrics $Q_1$, $Q_2$ and $Q_3$ in $\mathbb{P}^6$. Fix a point $P\in X$. Up to a change of coordinates, near $P$ we can write the local equations of $X=Q_1\cap Q_2\cap Q_3$ as follows:
\begin{align}
0 =& f(x_1, x_2, x_3, x_4, x_5, x_6) 
= x_4 + f_2(x_1, x_2, x_3, x_4, x_5, x_6) \notag \\
&= x_4  + f_2(x_1, x_2, x_3, 0, 0, 0)\notag \\ 
&+ x_4 l_{f,1} (x_1, x_2, x_3, x_4, x_5, x_6) + x_5 l_{f,2} (x_1, x_2, x_3, x_4, x_5, x_6) + x_6 l_{f,3} (x_1, x_2, x_3, x_4, x_5, x_6) \notag \\
&= x_4 + f_2(x_1, x_2, x_3, 0, 0, 0) + x_4 (l_{f,41}x_1+l_{f,42}x_2+l_{f,43}x_3+l_{f,44}x_4+l_{f,45}x_5+l_{f,46}x_6)\notag \\
&+ x_5 (l_{f,51}x_1+l_{f,52}x_2+l_{f,53}x_3+l_{f,54}x_4+l_{f,55}x_5+l_{f,56}x_6)\notag  \\
&+ x_6 (l_{f,61}x_1+l_{f,62}x_2+l_{f,63}x_3+l_{f,64}x_4+l_{f,65}x_5+l_{f,66}x_6),\notag
\end{align}
\begin{align}
0 =& g(x_1, x_2, x_3, x_4, x_5, x_6) = x_5 + g_2(x_1, x_2, x_3, x_4, x_5, x_6) \notag \\
&=x_5 + g_2(x_1, x_2, x_3, 0, 0, 0) \notag \\
&+x_4 l_{g,1} (x_1, x_2, x_3, x_4, x_5, x_6) + x_5 l_{g,2} (x_1, x_2, x_3, x_4, x_5, x_6) + x_6 l_{g,3} (x_1, x_2, x_3, x_4, x_5, x_6) \notag \\
&=x_5 + g_2(x_1, x_2, x_3, 0, 0, 0) + x_4 (l_{g,41}x_1+l_{g,42}x_2+l_{g,43}x_3+l_{g,44}x_4+l_{g,45}x_5+l_{g,46}x_6)\notag \\
&+ x_5 (l_{g,51}x_1+l_{g,52}x_2+l_{g,53}x_3+l_{g,54}x_4+l_{g,55}x_5+l_{g,56}x_6) \notag \\
&+ x_6 (l_{g,61}x_1+l_{g,62}x_2+l_{g,63}x_3+l_{g,64}x_4+l_{g,65}x_5+l_{g,66}x_6), \notag
\end{align}
\begin{align}
0 =& h(x_1, x_2, x_3, x_4, x_5, x_6) = x_6 + h_2(x_1, x_2, x_3, x_4, x_5, x_6) \notag \\
&=x_6  + h_2(x_1, x_2, x_3, 0, 0, 0) \notag \\
& + x_4 l_{h,1} (x_1, x_2, x_3, x_4, x_5, x_6) + x_5 l_{h,2} (x_1, x_2, x_3, x_4, x_5, x_6) + x_6 l_{h,3} (x_1, x_2, x_3, x_4, x_5, x_6) \notag \\
&=x_6 + h_2(x_1, x_2, x_3, 0, 0, 0) + x_4 (l_{h,41}x_1+l_{h,42}x_2+l_{h,43}x_3+l_{h,44}x_4+l_{h,45}x_5+l_{h,46}x_6) \notag \\
&+ x_5 (l_{h,51}x_1+l_{h,52}x_2+l_{h,53}x_3+l_{h,54}x_4+l_{h,55}x_5+l_{h,56}x_6) \notag \\
&+ x_6 (l_{h,61}x_1+l_{h,62}x_2+l_{h,63}x_3+l_{h,64}x_4+l_{h,65}x_5+l_{h,66}x_6), \notag
\end{align}
where $f_2$, $g_2$ and $h_2$ are homogeneous polynomials of degree $2$,  
$l_{f,i}$, $l_{g,i}$, $l_{h,i}$ are linear forms and $l_{f,ij}$, $l_{g,ij}$, $l_{h,ij}\in\mathbb{C}$ are such that $l_{f,ij}=l_{f,ji}$, $l_{g,ij}=l_{g,ji}$, $l_{h,ij}=l_{h,ji}$ for $4\leq i,j\leq 6$.
Up to degree $5$ we have
\begin{align}
x_4 =& - f_2(x_1, x_2, x_3, 0, 0, 0) + f_2(x_1, x_2, x_3, 0, 0, 0) (l_{f,41}x_1+l_{f,42}x_2+l_{f,43}x_3) \notag\\
&+ g_2(x_1, x_2, x_3, 0, 0, 0) (l_{f,51}x_1+l_{f,52}x_2+l_{f,53}x_3) 
+ h_2(x_1, x_2, x_3, 0, 0, 0) (l_{f,61}x_1+l_{f,62}x_2+l_{f,63}x_3)\notag\\
&- f_2(x_1, x_2, x_3, 0, 0, 0) (l_{f,44} f_2(x_1, x_2, x_3, 0, 0, 0) + l_{f,45} g_2(x_1, x_2, x_3, 0, 0, 0) + l_{f,46} h_2(x_1, x_2, x_3, 0, 0, 0))\notag\\
&- g_2(x_1, x_2, x_3, 0, 0, 0) (l_{f,54} f_2(x_1, x_2, x_3, 0, 0, 0) + l_{f,55} g_2(x_1, x_2, x_3, 0, 0, 0) + l_{f,56} h_2(x_1, x_2, x_3, 0, 0, 0))\notag\\
&- h_2(x_1, x_2, x_3, 0, 0, 0) (l_{f,64} f_2(x_1, x_2, x_3, 0, 0, 0) + l_{f,65} g_2(x_1, x_2, x_3, 0, 0, 0) + l_{f,66} h_2(x_1, x_2, x_3, 0, 0, 0))\notag\\
&+ (f_2(x_1, x_2, x_3, 0, 0, 0) (l_{f,41}x_1+l_{f,42}x_2+l_{f,43}x_3) 
+ g_2(x_1, x_2, x_3, 0, 0, 0) (l_{f,51}x_1+l_{f,52}x_2+l_{f,53}x_3) \notag\\
&+ h_2(x_1, x_2, x_3, 0, 0, 0) (l_{f,61}x_1+l_{f,62}x_2+l_{f,63}x_3))(l_{f,41}x_1+l_{f,42}x_2+l_{f,43}x_3)\notag\\
&+ (f_2(x_1, x_2, x_3, 0, 0, 0) (l_{g,41}x_1+l_{g,42}x_2+l_{g,43}x_3) 
+ g_2(x_1, x_2, x_3, 0, 0, 0) (l_{g,51}x_1+l_{g,52}x_2+l_{g,53}x_3) \notag\\
&+ h_2(x_1, x_2, x_3, 0, 0, 0) (l_{g,61}x_1+l_{g,62}x_2+l_{g,63}x_3))(l_{f,51}x_1+l_{f,52}x_2+l_{f,53}x_3)\notag\\
&+ (f_2(x_1, x_2, x_3, 0, 0, 0) (l_{h,41}x_1+l_{h,42}x_2+l_{h,43}x_3) 
+ g_2(x_1, x_2, x_3, 0, 0, 0) (l_{h,51}x_1+l_{h,52}x_2+l_{h,53}x_3) \notag\\
&+ h_2(x_1, x_2, x_3, 0, 0, 0) (l_{h,61}x_1+l_{h,62}x_2+l_{h,63}x_3))(l_{f,61}x_1+l_{f,62}x_2+l_{f,63}x_3).
\end{align}
We also can write similar expressions for $x_5$ and $x_6$.
Consider a net of quadrics spanned by $Q_1, Q_2, Q_3$. For a general $X$, we may assume that its discriminant curve $\Delta$ of degree $7$ is smooth. 

\begin{proposition}
\label{prop-222-reg1}
For a general $X$, we have $\mathrm{reg}_1(X)=0$.
\end{proposition}
\begin{proof}
Let $H$ be a hyperplane in $\mathbb{P}^6$. If $H$ intersects $T_P X$ transversally, then the pair $(X, H|_X)$ is log smooth at $P$. So we may assume that $H$ contains $T_P X$. Then $H$ is given by the equation $a x_4 + b x_5 + c x_6 = 0$. 
Express $x_4$, $x_5$ and $x_6$ from the above equations and put it into the equation of $H$ to obtain the following equation of $X\cap Q'$ (written up to degree $4$ terms):
\begin{align}
0 = a x_4 &+ b x_5 + c x_6 = - a f_2(x_1, x_2, x_3, 0, 0, 0) - b g_2(x_1, x_2, x_3, 0, 0, 0) - c h_2(x_1, x_2, x_3, 0, 0, 0)\notag\\
&+ f_2(x_1, x_2, x_3, 0, 0, 0) ( al_{f,1} (x_1, x_2, x_3, 0, 0, 0) + bl_{g,1} (x_1, x_2, x_3, 0, 0, 0) + cl_{h,1} (x_1, x_2, x_3, 0, 0, 0) )\notag\\
&+ g_2(x_1, x_2, x_3, 0, 0, 0) ( al_{f,2} (x_1, x_2, x_3, 0, 0, 0) + bl_{g,2} (x_1, x_2, x_3, 0, 0, 0) + cl_{h,2} (x_1, x_2, x_3, 0, 0, 0) )\notag\\
&+ h_2(x_1, x_2, x_3, 0, 0, 0) ( al_{f,3} (x_1, x_2, x_3, 0, 0, 0) + bl_{g,3} (x_1, x_2, x_3, 0, 0, 0) + cl_{h,3} (x_1, x_2, x_3, 0, 0, 0) ).\notag
\end{align}
Note that the vanishing of the quadratic term of this equation is a condition of codimension $4$, so it cannot happen for a general $X$. Also, the condition that the quadratic term is a square of a linear form is a condition of codimension $1$.  We may assume that $a=1$ and $b=c=0$ and the quadratic term is 
\[
f_2(x_1, x_2, x_3, 0, 0, 0)=x_1^2,
\]
and the cubic term is 
\[
x_1^2 l_{f,1} (x_1, x_2, x_3,0,0,0)
+ g_2(x_1, x_2, x_3, 0, 0, 0) l_{f,2} (x_1, x_2, x_3,0,0,0)
+ h_2(x_1, x_2, x_3, 0, 0, 0) l_{f,3} (x_1, x_2, x_3,0,0,0).
\]
We are going to use Proposition \ref{lem-striclty-lc-2-complement}.
Put $x_1=0$ to obtain
\[
g_2(0, x_2, x_3, 0, 0, 0) l_{f,2} (0, x_2, x_3,0,0,0) + h_2(0, x_2, x_3, 0, 0, 0) l_{f,3} (0, x_2, x_3,0,0,0).
\]
We may assume that both $l_{f,2}(0, x_2, x_3,0,0,0)$ and $l_{f,3}(0, x_2, x_3,0,0,0)$ do not vanish simultaneously (otherwise we get a condition of codimension $4$). 

\subsection{
Assume that $l_{f,2}(0, x_2, x_3,0,0,0)=0$}
Then the condition that the cubic term is a cube of a linear form implies that $g_2(0, x_2, x_3, 0, 0, 0)$ is proportional to $l_{f,2} (0, x_2, x_3,0,0,0)^2$ which is a condition of codimension $2$. In total, we get a codimension $5$ condition which cannot happen for general $X$. Hence we may assume that $l_{f,2}(0, x_2, x_3,0,0,0)\neq 0$. 

\subsection{
Assume that $l_{f,2}(0, x_2, x_3,0,0,0)$ and $l_{f,3}(0, x_2, x_3,0,0,0)$ are proportional}
Then we obtain $l_{f,2}(0, x_2, x_3,0,0,0)=x_2$ and $l_{f,3}(0, x_2, x_3,0,0,0)=ax_2$ for $0\neq a\in \mathbb{C}$. This is a condition of codimension $1$. If the cubic term is zero, we get
\[
g_2(0, x_2, x_3, 0, 0, 0) l_{f,2} (0, x_2, x_3,0,0,0) + a h_2(0, x_2, x_3, 0, 0, 0) 
\]
which is a condition of codimension $3$, so in total we get a codimension $4$ condition, so this case does not happen for general $X$. Assume that the cubic term is a cube of a linear form, so $g_2(0, x_2, x_3, 0, 0, 0)$ and $h_2(0, x_2, x_3, 0, 0, 0)$ are proportional to $x_2^2$. This is a condition of codimension $4$ which cannot happen for general $X$.

\subsection{
Assume that $l_{f,2} (0, x_2, x_3,0,0,0)$ and $l_{f,3} (0, x_2, x_3,0,0,0)$ are not proportional}
Then we may assume that $l_{f,2} (0, x_2, x_3,0,0,0)=x_2$ and $l_{f,3} (0, x_2, x_3,0,0,0)=x_3$. The cubic term is
\begin{align*}
&x_2 g_2(0, x_2, x_3, 0, 0, 0) + x_3 h_2(0, x_2, x_3, 0, 0, 0)\\
=& (a_1x_2^2+a_2x_2x_3+a_3x_3^2)x_2 + (b_1x_2^2+b_2x_2x_3+b_3x_3^2)x_3\\
=& a_1x_2^3+(a_2+b_1)x_2^2x_3+(a_3+b_2)x_2x_3^2 + b_3x_2^3.
\end{align*}
If the cubic term vanishes, we get a condition of codimension $4$ which cannot happen for general $X$. Assume that in this case the cubic term is a cube of a linear form, so it is equal to
\[
l(x_2,x_3)^3 = (\alpha x_2 + \beta x_3)^3 = \alpha^3 x_2^3 + 3 \alpha^2 \beta x_2^2x_3 + 3 \alpha \beta^2 x_2x_3^2 + \beta^3 x_3^3
\]
for some $\alpha, \beta\in \mathbb{C}$, 
which gives a condition of codimension $2$, so we get a codimension $3$ condition in total. 
We may assume that $l(x_2,x_3)=x_2-x_3$. 
We have to analyze the quartic term restricted to $f_2(x_1,x_2,x_3,0,0,0)=x_1^2$, $x_1=0$ and $x_2=x_3$: 
\begin{align*}
&- l_{f,55} g_2(x_1, x_2, x_3, 0, 0, 0)^2 - (l_{f,56}+l_{f,65}) h_2(x_1, x_2, x_3, 0, 0, 0)g_2(x_1, x_2, x_3, 0, 0, 0)\\
&- l_{f,66} h_2(x_1, x_2, x_3, 0, 0, 0))^2+ (g_2(x_1, x_2, x_3, 0, 0, 0) (l_{f,52}+l_{f,53}) \\
&+ h_2(x_1, x_2, x_3, 0, 0, 0) (l_{f,62}+l_{f,63}))(l_{f,42}+l_{f,43})x_2^2\\
&+ (g_2(x_1, x_2, x_3, 0, 0, 0) (l_{g,52}+l_{g,53}) \\
&+ h_2(x_1, x_2, x_3, 0, 0, 0) (l_{g,62}+l_{g,63}))(l_{f,52}+l_{f,53})x_2^2\\
&+ (g_2(x_1, x_2, x_3, 0, 0, 0) (l_{h,52}+l_{h,53}) \\
&+ h_2(x_1, x_2, x_3, 0, 0, 0) (l_{h,62}+l_{h,63}))(l_{f,62}+l_{f,63})x_2^2.
\end{align*}
If it is equal to $0$, then we have an additional codimension $1$ condition, hence in total we have a codimension $4$ condition, so this cannot happen for a general $X$.
\end{proof} 
 
\section{Fanos that are not blow-ups}
\label{sec-Fanos-that-are-not-blow-ups}
We treat the case of smooth Fano threefolds with Picard rank at least two. First we deal with the case of Fano threefolds that cannot be realized as the blow up of some other smooth variety. It is known that such threefolds belong to one of the following families: 
\begin{equation}
\label{Fano-to-treat}
2.2,\ 2.6,\ 2.8,\ 2.18,\ 2.24,\ 2.32,\ 2.34,\ 2.35,\ 2.36,\ 3.1,\ 3.2,\ 3.31.
\end{equation}
Note that the varieties $2.34, 2.35, 2.36, 3.31$ are toric, so by Remark \ref{remark:toric} it suffices to deal with the remaining cases. We will repeatedly use the following lemma.

\begin{lemma}
\label{lemma:blow-up}
Let $Y$ be a smooth threefold, and let $(Y, D')$ be a dlt
(respectively, lc) log CY pair. Let $Z=\sum Z_i$ is a disjoint union of (effective
integral irreducible) curves $Z_i$ and $P=\sum P_i$ be a union of points such that  
\begin{itemize}
\item
$Z$ and $P$ are disjoint,
\item
each $Z_i$ belongs to the smooth locus of some $D'_j$ where $D'_j$ is a component of $D'$ with coefficient~$1$,
\item
each $P_i$ belongs to the smooth locus of the
intersection $D'_j\cap D'_k$ for some $j\neq k$, where $D'_j$ and $D'_k$
are components of $D'$ with coefficient $1$.
\end{itemize}
Let $X=\mathrm{Bl}_{Z\cup P} (Y)$ be the blow up of the union of $Z$ and $P$ on $Y$. Consider the log pullback $(X, D)$ of $(Y, D')$ which
is a dlt (respectively, lc) log CY pair. Then $\mathcal{D}(Y, D')$ is
homeomorphic to $\mathcal{D}(X, D)$. In particular, if $\mathcal{D}(Y, D')=2$
 then $\mathrm{coreg}(Y)=\mathrm{coreg}(X)=0$.
\end{lemma}
\begin{proof}
The equality $\dim \mathcal{D}(Y, D')=\dim \mathcal{D}(X, D)$ follows from \cite[Proposition 11]{dFKX17} since the pairs $(X, D)$
and $(Y, D')$ are crepant birational. The last statement follows immediately.
\end{proof}

\begin{lemma}
\label{lem-double-cover}
Let $Y$ be a smooth threefold. Let $f\colon X\to Y$ be a finite morphism of degree $2$ with a smooth ramification divisor
$R\subset Y$, so in particular $X$ is smooth as well. Assume that there exists a boundary $D'$ on~$Y$ such that any component of $R$ is not contained in $D'$, the
pair $(Y, D'+R/2)$ is lc log CY. 
Then the pair $(X, D)$ is lc log CY where $D$ is defined by the formula
\[
K_X + D = f^*(K_Y+D'+R/2)\sim_{\mathbb{Q}} 0.
\]
Also, if the pair $(Y, D'+R/2)$ is snc, then $(X, D)$ is snc as well. Consequently, if $\mathrm{coreg}(Y, D'+R/2)=0$ then $\mathrm{coreg}(X, D)=0$.
\end{lemma}
\begin{proof}
Follows from \cite[20.2, 20.3]{CKM88}.
\end{proof}

We treat the Fano threefolds listed in \eqref{Fano-to-treat} starting from the case of Picard rank $2$. It turns out that the threefolds with larger anti-canonical degree $(-K_X)^3$ are easier to deal with, so we analyze them first.

\begin{lemma}
\label{lemma:2-32}
Let $X$ be a Fano threefold in the family \textnumero\,$2.32$.
Then $\coreg(X)=0$.
\end{lemma}
\begin{proof}
The variety $X$ is a smooth divisor of bidegree $(1,1)$ in
$\mathbb{P}^2\times\mathbb{P}^2$. Let $D=D_1+D_2+D_3+D_4$ be a boundary divisor on $X$ where $D_1$
and $D_2$ have bidegree $(1,0)$, $D_3$ and $D_4$ have bidegree $(0,1)$, and
all the divisors $D_i$ are general. One has $-K_X\sim D$. 
Note that 
$D_1\cap D_2\cap D_3\neq \emptyset$ 
and the pair $(X, D)$ has simple normal crossings
which shows that $\coreg(X)=0$.
\end{proof}

\begin{lemma}
\label{lemma:2-24}
Let $X$ be a Fano threefold in the family \textnumero\,$2.24$.
Then $\coreg(X)=0$.
\end{lemma}
\begin{proof}
The variety $X$ is a smooth divisor on $\mathbb{P}^2\times\mathbb{P}^2$
of bidegree $(1, 2)$. 
Let $D=D_1+D_2+D_3$ be a boundary divisor on $X$ where $D_1$
and $D_2$ have bidegree $(1,0)$, $D_3$ has bidegree $(0,1)$, and
all the divisors $D_i$ are general.
By adjunction,
one has $-K_X\sim D$. 
Note that 
$D_1\cap D_2\cap D_3\neq \emptyset$,
the pair $(X, D)$ has simple normal crossings 
which shows that $\coreg(X)=0$.
\end{proof}

\begin{lemma}
\label{lemma:2-18}
Let $X$ be a Fano threefold in the family \textnumero\,$2.18$.
Then $\coreg(X)=0$.
\end{lemma}
\begin{proof}
The variety $X$ is realized as a double cover $f\colon X\to Y=\mathbb{P}^1\times\mathbb{P}^2$ ramified in a smooth divisor $R$ of bidegree $(2, 2)$ on~$Y$. The Hurwitz formula yields
\[
K_X = f^* ( K_Y + R/2 ). 
\]
Choose a complement of $K_Y + R/2\sim(-1,-2)$ of the form $D'=D'_1+D'_2+D'_3$
where $D'_i$ are general divisors on $Y$ of bidegree $(1,0)$,
$(0,1)$ and $(0,1)$, respectively. By construction, the pair $(Y, D'+R/2)$ has simple normal crossings and $D'$ has a zero-dimensional stratum. We apply Lemma \ref{lem-double-cover} to conclude that $\coreg(X)=0$. 
\end{proof}

\begin{lemma}
\label{lemma:2-8}
Let $X$ be a Fano threefold in the family \textnumero\,$2.8$.
Then $\coreg(X)=0$.
\end{lemma}
\begin{proof}
The variety $X$ is realized as a double cover $f\colon X\to V_7$
ramified in a smooth divisor $R$ on $V_7$ where $R\sim -K_{V_7}$. Denote by $g\colon V_7\to\mathbb{P}^3$ the blow up in a point $P\in \mathbb{P}^3$, and denote the $g$-exceptional divisor by~$E$. The Hurwitz formula yields
\[
K_X = f^*(K_{V_7}+R/2)
\sim f^*(-2H+E)
\]
where $H$ is the strict transform of a general plane via the map
$g$. 
Consider a complement $D' = D'_1+D'_2+D'_3$ on $V_7$ where $D'_1$ and
$D'_2$ are strict transform of general planes passing through $P$, so $D'_1\sim D'_2\sim H-E$, and $D'_3=E$. By construction, the pair $(Y, D'+R/2)$ has simple normal crossings and $D'$ has a zero-dimensional stratum. Applying Lemma \ref{lem-double-cover}, we conclude that $\coreg(X)=0$.
\end{proof}

\begin{lemma}
\label{lemma:2-6}
Let $X$ be a Fano threefold in the family \textnumero\,$2.6$.
Then $\coreg(X)=0$.
\end{lemma}
\begin{proof}
First we consider the case when $X$ is a smooth divisor of bidegree $(2,2)$
in $\mathbb{P}^2\times\mathbb{P}^2$. Let $\pi_i\colon X\to \mathbb{P}^2$ be two natural projections. We construct a complement
of $K_X$ of the form $D=D_1+D_2$ where $D_1$, and $D_2$ are
smooth surfaces on $X$ of bidegree $(1,0)$ and $(0,1)$, respectively. Note that each $D_i$ is a del Pezzo surface of degree
$2$. Let us start with a general $D_1$. 
Note that $\pi_2|_{D_1}\colon D_1\to \mathbb{P}^2$ is a double cover ramified in a smooth quartic curve $C\subset \mathbb{P}^2$. 
Consider the projection $\pi_2\colon X \to \mathbb{P}^2$ which is a conic bundle with a discriminant curve $\Delta\subset \mathbb{P}^2$ of degree $6$. We can choose a line $L\subset \mathbb{P}^2$ such that $L$ is tangent to $C$, and $L$ intersects $\Delta$ transversally. Indeed, this follows from the fact that the dual curves $C^\vee$ and $\Delta^\vee$ in the dual plane $(\mathbb{P}^2)^\vee$ are distinct. Then $D_2=\pi_2^{-1}(L)$ is a divisor of bidegree $(0,1)$, and $D_1\cap D_2$ is a nodal curve of arithmetic genus $1$. Also note that $D_2$ is smooth. Applying Lemma \ref{lem-nodal-curve}, we see that $\coreg(X)=0$.

Now, we consider the case when $X$ is a double cover $f\colon
X\to Y$ where $Y$ is a Fano threefold $2.32$, that is, $Y $ is a smooth
divisor of bidegree $(1,1)$ in $\mathbb{P}^2\times\mathbb{P}^2$, and
$f$ is ramified in a smooth divisor $R\sim -K_Y$. The Hurwitz formula yields
\[
K_X = f^*(K_Y + R/2). 
\]
We construct a complement $D'=D'_1+D'_2$ of $K_Y + R/2\sim (-1,-1)$ where $D'_i$ are smooth, 
$D'_1$ has bidegree $(1,0)$ and $D'_2$ has bidegree $(0,1)$. Note that both $D'_i$ are
isomorphic to a Hirzebruch surface $\mathbb{F}_1$. We choose them in
such a way that their intersection is a union of two smooth
rational curves $L_1$ and $L_2$ intersecting transversally. 

To do this, start with a general $D'_1$ of bidegree $(1,0)$. Then $\pi_2|_{D'_1}\colon D'_1\to \mathbb{P}^2$ is the blow up of a point $P\in \mathbb{P}^2$. Consider a line $L\subset \mathbb{P}^2$ such that $P\in L$. Then the preimage $\pi_2|_{D'_1}^{-1}(L)$ is a reducible curve $C$ on $D'_1$ such that $C=L_1+L_2$ where $L_1$ is a $(-1)$-curve on $D'_1$, and $L_2$ is a fiber of the projection $D'_1=\mathbb{F}_1\to \mathbb{P}^1$. Put $D'_2=\pi_2^{-1}(L)$. Then it is clear that $D'_2$ has bidegree $(0, 1)$. Note that 
\[
N_{L_1/D'_1}=\oo(-1), \quad \quad L_1 \cdot D'_1 = 1 \quad \quad \text{and} \quad \quad N_{L_1/X}=\oo\oplus\oo
\] 
where the latter holds because $L_1$ is a fiber of a $\mathbb{P}^1$-bundle $\pi_1$. Applying Lemma \ref{lem-line-conic}, we conclude that $\dim \mathcal{D}(Y, D')=2$. Since $D'_1$ was chosen to be general, $R\cap D'_1$ is a smooth irreducible curve which is a positive section of the projection $D'_1=\mathbb{F}_1\to \mathbb{P}^1$. 
Applying Lemma \ref{lem-double-cover}, we conclude that $\coreg(X)=0$.
\end{proof}

\begin{lemma}
\label{lemma:2-2}
Let $X$ be a Fano threefold in the family \textnumero\,$2.2$. Then
$\coreg(X)=0$.
\end{lemma}
\begin{proof}
The variety $X$ is realized as a double cover $f\colon X\to Y=\mathbb{P}^1\times\mathbb{P}^2$ ramified in a smooth divisor $R$ of bidegree $(2, 4)$. The Hurwitz formula yields
\[
K_X = f^* ( K_Y + R/2 ) \sim f^* ( (-1,-1) ).
\]
Choose a complement of $K_Y + R/2$ of the form $D'=D'_1+D'_2$ where
$D'_1$ and $D'_2$ are divisors on $Y$ of bidegree $(1,0)$ and $(0,1)$, respectively. We show that we can choose $D'_1$ and $D'_2$ such that the intersection of their preimages $D_1=f^*(D'_1)$ and $D_2=f^*(D'_2)$ is the union of two smooth rational curves $L_1$ and $L_2$. First choose $D'_1$ to be general and observe that $D_1=f^*(D'_1)$ is a smooth del Pezzo surface
of degree $2$, and for $C=D_2|_{D_1}$ we have $C\sim -K_{D_1}$. Note that $C'=f(C)$ is a line on $D'_1=\mathbb{P}^2$. We can choose $D'_2$ in such a way that $C'$ is a general tangent line to a smooth quartic $R\cap D'_1$, and $D'_2$ is smooth. Then $C$ is a nodal irreducible curve of arithmetic genus $1$ on $D'_1$. Consequently, we have a complement of $K_X$ of the form $D=D_1+D_2$, so
\[
K_X + D = f^* ( K_Y + R/2 + D') \sim 0.
\]
Applying Lemma \ref{lem-nodal-curve} and Lemma \ref{lem-double-cover} we see that $\coreg(X)=0$.
\end{proof}

We treat Fano threefolds listed in \eqref{Fano-to-treat} with Picard rank $3$.

\begin{lemma}
\label{lemma:3-2}
Let $X$ be a Fano threefold in the family \textnumero\,$3.2$.
Then $\coreg(X)=0$.
\end{lemma}
\begin{proof}
We have that $X$ is a smooth divisor in a $\mathbb{P}^2$-bundle over
$\mathbb{P}^1\times\mathbb{P}^1$ of the form
\[
\mathbb{P}=\mathbb{P}_{\mathbb{P}^1\times\mathbb{P}^1}(\oo\oplus\oo(-1,-1)^{\oplus2})
\]
such that $X\sim 2L + 2F_1 + 3F_3$ where $F_1$ is the pullback of the divisor $l_1$ of bidegree $(1,0)$ on $\mathbb{P}^1\times\mathbb{P}^1$, $F_2$ is the pullback of the divisor $l_2$ of bidegree $(0,1)$ on $\mathbb{P}^1\times\mathbb{P}^1$, and $L$ is the tautological divisor. We have
\[
K_{\mathbb{P}} = -3L - 4F_1 - 4F_2, \quad \quad \quad \quad K_X = (-L - 2F_1 - F_2)|_X.
\]
Put $L_X=L|_X$. Note that $L_X$ is a bisection of the natural projection $\pi\colon X\to
\mathbb{P}^1\times\mathbb{P}^1$, that is, it intersects its general fiber
in two points. We have $L_X\simeq \mathbb{P}^1\times\mathbb{P}^1$. Consider the boundary $D$ on $X$ such that $D=D_1+D_2+D_3+D_4$ where $D_1=L_X$, $D_2\sim D_3\sim F_1|_X$ are general, and $D_4\sim F_2|_X$ is general as well. Then $K_X+D\sim 0$, the pair $(X, D)$ has simple normal crossings, and
\[
(K_X + L_X)|_{L_X} = K_{L_X} \sim -2l_1 - 2l_2.
\]
Observe that $F_1|_{L_X} \sim l_1$, and $F_2|_{L_X} \sim 2l_2$. Hence $D_1\cap D_2\cap D_3\neq \emptyset$, and thus $\mathrm{coreg}(X)=0$.
\end{proof}

\begin{lemma}
\label{lemma:3-1}
Let $X$ be a Fano threefold in the family \textnumero\,$3.1$.
Then $\coreg(X)=0$.
\end{lemma}
\begin{proof}
The variety $X$ is realized as a double cover $f\colon X\to
Y=\mathbb{P}^1\times\mathbb{P}^1\times\mathbb{P}^1$ ramified in a
smooth divisor $R$ of type $(2,2,2)$. The Hurwitz formula yields
\[
K_X = f^*(K_Y + R/2).
\]
Pick a complement of $K_Y+R/2\sim (-1,-1,-1)$ of the form $D' = D'_1+D'_2+D'_3$ where $D'_i$ are general divisors on $Y$
of the type $(1,0,0)$, $(0,1,0)$ and
$(0,0,1)$, respectively. Note that $D'_1\cap D'_2\cap D'_3\neq \emptyset$ and the pair $(Y, D')$ has simple normal crossings. Applying Lemma \ref{lem-double-cover}, we conclude that $\coreg(X)=0$.
\end{proof}

\section{Fano blow-ups with $\rho=2$}
\label{sec-Fano-with-big-rho}

In what follows, we will use the following well-known description of the normal bundle to a line on smooth Fano threefold, see e.g. \cite[Proposition 2.2.8.]{KPS18}. Namely, for any line $L\subset \mathbb{P}^3$ we have $N_{L/\mathbb{P}^3}=\oo(1)\oplus \oo(1)$. For any line $L$ on a smooth quadric $Q\subset \mathbb{P}^4$ we have $N_{L/Y}=\oo\oplus \oo(1)$. For a smooth del Pezzo threefold $Y$, we have $N_{L/Y}=\oo\oplus \oo$ for a general line $L\subset Y$, and $N_{L/Y}=\oo(-1)\oplus \oo(1)$ for special lines.

\begin{remark}
\label{rem-no-planes}
Observe that a del Pezzo threefold $Y$ with $\rho(Y)=1$ does not contain planes in the embedding given by the linear system $|-K_Y/2|$. Indeed, assume that $H\subset Y$ is a plane, so $H\simeq \mathbb{P}^2$. Then for two general hyperplane sections $H_1$ and $H_2$, we have $H\cdot H_1\cdot H_2=1$. 
On the other hand, let $H'$ be a generator of $\mathrm{Pic}(Y)$. Then $H\sim kH'$, and we have $1=H\cdot H_1\cdot H_2=kH'^3 = k d$, so $d=1$ and $k=1$. But this contradicts to the adjunction formula $-3H|_H=K_H=(K_X+H)|_H=-H|_H$. Arguing similarly, we see that $Y$ does not contain smooth quadric surfaces and quadric cones.
\end{remark}

We treat Fano threefolds that can be realized as a blow up of some other smooth variety. First we deal with the case of Picard number $2$.

\begin{lemma}
\label{corollary:blow-ups}
Let $X$ be a Fano threefold with Picard number $2$ that can be realized as a blow up of some other smooth variety. 
In the case when $X$ belongs to the family $2.1$, we assume that $X$ is general. 
Then~$\coreg(X)=~0$.
\end{lemma}
\begin{proof} 
We apply Lemma \ref{lemma:blow-up} and use its notation. In what follows, we construct the pairs $(Y, D')$ of coregularity $0$, and describe the subset $Z\subset Y$ such that $X=\mathrm{Bl}_Z Y$. Then the log pullback $(X, D)$ of $(Y, D')$ will also have coregularity $0$. 
\subsection{
Case
$Y=\mathbb{P}^3$}
\subsubsection{
    \label{2:30}
    $X$ is a Fano variety \textnumero\,2.30, }
    $Z$ is a plane conic,   
$D'=D'_1+D'_2+D'_3+D'_4$ where $D'_i$ are general planes such that
$Z\subset D'_1$.

\subsubsection{
    \label{2:28}
    $X$ is a Fano variety \textnumero\,2.28,}    
    $Z$ is a plane cubic curve,    
    $D'$ is as in the previous case, and $Z\subset D'_1$.
        
\subsubsection{
    \label{2:27}
    $X$ is a Fano variety \textnumero\,2.27,}     
    $Z$ is a twisted cubic curve,
    $D'=D'_1+D'_2+D'_3$ where $D'_1$ is a smooth quadric, $D'_2$ and $D'_3$ are general planes, and $Z\subset D'_1$.
    
\subsubsection{
    \label{2:25}
    $X$ is a Fano variety \textnumero\,2.25,}     
    $Z$ is a smooth elliptic curve of degree $4$ which is an intersection
of two quadrics, 
$D'$ as in the previous case, and $Z\subset D'_1$.

    \subsubsection{
    \label{2:15}
    $X$ is a Fano variety \textnumero\,2.15, }
    $Z$ is a smooth curve of degree $6$ which is an intersection of a
quadric and a smooth cubic surface, 
$D'=D'_1+D'_2+D'_3$ where $D'_1$ is either a smooth quadric or a quadric cone, $D'_2$ and $D'_3$
are general planes, and $Z\subset D'_1$.

\subsubsection{
    \label{2:12}
    $X$ is a Fano variety \textnumero\,2.12,}
    $Z$ is a smooth curve of degree $6$ and genus $3$ which is an intersection of $4$ cubic surfaces,    
    $D' = D'_1+D'_2$, where $D'_1$ is a
 cubic surface such that $Z\subset D'_1$, $D'_2$ is a hyperplane such
that $D'_1\cap D'_2$ is a union of a line $L$ and a conic~$C$. By Lemma \ref{lem-smooth-section}, we can choose $D'_1$ to be smooth. 
Since $N_{L/Y}=\oo(1)\oplus \oo(1)$, by Lemma \ref{reducible_cubic_surface} we see that $\dim\mathcal{D}(Y, D')=2$. Thus, 
$\mathrm{coreg}(Y, D')=0$. 

\subsubsection{
    \label{2:9}
    $X$ is a Fano variety \textnumero\,2.9,}     
    $Z$ is a smooth curve of degree $7$ and genus $5$ which is an intersection of cubic surfaces, 
    $D'$ is as in the
previous case, and $Z\subset D'_1$. By Lemma \ref{lem-smooth-section}, we may assume that $D'_1$ is smooth. By Lemma \ref{reducible_cubic_surface}, we see that $\dim\mathcal{D}(Y, D')=2$. Thus, 
$\mathrm{coreg}(Y, D')=0$. 

\subsubsection{
    \label{2:4}
    $X$ is a Fano variety \textnumero\,2.4,}  
    $Z$ is a smooth curve which is an intersection of two cubic surfaces,     
    $D'$ as in the previous case, and $Z\subset D'_1$. By Lemma \ref{lem-smooth-section}, we may assume that $D'_1$ is smooth. By Lemma \ref{reducible_cubic_surface}, we see that $\dim\mathcal{D}(Y, D')=2$. Thus, 
    $\mathrm{coreg}(Y, D')=0$. 

\subsection{
Case when $Y$ is a smooth quadric hypersurface
$\mathbb{P}^4$} 
\subsubsection{
    \label{2:31}
    $X$ is a Fano variety \textnumero\,2.31,}  
    $Z$ is a line,     
    $D'=D'_1+D'_2+D'_3$ where $D'_i$ are general hyperplane sections, and
$Z\subset D'_1$. By Lemma \ref{lem-smooth-section}, one can choose $D'_1$ such that it is smooth.

\subsubsection{
    \label{2:29}
    $X$ is a Fano variety \textnumero\,2.29,}     
    $Z$ is a conic,     
    $D'$ is as in the previous case, and $Z\subset D'_1$. By Lemma \ref{lem-smooth-section}, we can choose
$D'_1$ to be smooth.

\subsubsection{
    \label{2:23}
    $X$ is a Fano variety \textnumero\,2.23,}     
    $Z$ is a smooth curve which is an intersection of a hyperplane section and an element from $\OOO_Y(2)$,     
    $D'=D'_1+D'_2+D'_3$, where $D'_1$ is either a smooth
hyperplane section, or a hyperplane section which is a quadric cone, 
$D'_2$ and $D'_3$ are general hyperplane sections, and $Z\subset D'_1$.

\subsubsection{
    \label{2:21}
    $X$ is a Fano variety \textnumero\,2.21,} 
     $Z$ is a smooth rational quartic curve, 
$D'=D'_1+D'_2$ where $D'_1$ is a
smooth element from $\OOO_Y(2)$ such that $Z\subset D'_1$, $D'_2$ is a smooth hyperplane section
such that $D'_1$ intersects $D'_2$ in a union of a line $L$ and a smooth
rational cubic curve $C$. We can choose such $D'_1$ and $D'_2$ by Lemma \ref{lem-smooth-section}. Since $N_{L/Y}=\oo\oplus \oo(1)$, by Lemma \ref{lem-line-conic} we see that $\dim\mathcal{D}(Y, D')=2$. Thus, 
$\mathrm{coreg}(Y, D')=0$. 

\subsubsection{
    \label{2:13}
    $X$ is a Fano variety \textnumero\,2.13,}     
    $Z$ is a smooth curve of degree
$6$ and genus $2$ which is an intersection of quadrics, 
    $D'$ as in the previous case, $D'_1\cap D'_2$ is a union of a line $L$ and a smooth rational cubic curve $C$, and
$Z\subset D'_1$. We can choose $D'_1$ to be smooth by Lemma \ref{lem-smooth-section}. By Lemma \ref{lem-line-conic}, we see that $\dim\mathcal{D}(Y, D')=2$. Thus, 
$\mathrm{coreg}(Y, D')=0$.

\subsubsection{
    \label{2:7}
    $X$ is a Fano variety \textnumero\,2.7,}     
    $Z$ is a smooth curve which is an intersection two elements from $\OOO_Y(2)$,     
    $D'$ as in the previous case. We can choose $D'_1$ to be smooth by Lemma \ref{lem-smooth-section}. By Lemma \ref{lem-line-conic}, we see that $\dim\mathcal{D}(Y, D')=2$. Thus, 
    $\mathrm{coreg}(Y, D')=0$. 
    
\subsection{Case $Y=V_5$} That is, $Y$ is an intersection
of the Grassmannian $\mathrm{Gr}(2,5)\subset \mathbb{P}^9$ with a subspace of
codimension $3$.
\subsubsection{
    \label{2:26}
    $X$ is a Fano variety \textnumero\,2.26,}   
    $Z$ is a line,     
    $D'=D'_1+D'_2$ where $D'_1$ and $D'_2$ are smooth hyperplane sections intersecting in a union of a line $Z$ and a smooth rational quartic curve $C$ as constructed in Lemma \ref{lemma:Fano del
Pezzo}. 
By Lemma \ref{lem-line-conic}, we see that $\dim\mathcal{D}(Y, D')=2$. Thus, 
$\mathrm{coreg}(Y, D')=~0$. 

\subsubsection{
    \label{2:22}
    $X$ is a Fano variety \textnumero\,2.22,}     
    $Z$ is a conic,     
    $D'=D'_1+D'_2$ where $D'_1$ and $D'_2$ are smooth hyperplane sections intersecting in a union of a line $L$ and a smooth rational quartic curve $C$ as constructed in Lemma \ref{lemma:Fano del
Pezzo}, and $Z\subset D'_1$. We claim that a hyperplane section $D'_1$ can be chosen smooth. Let $\Pi=\mathbb{P}^2$ be a plane that contains $Z$. Since the Grassmannian is the intersection of quadrics and $Y$ does not contain planes by Remark \ref{rem-no-planes}, we have $Y\cap \Pi=Z$. Hence by Lemma \ref{lem-smooth-section} we can choose a smooth hyperplane section $D'_1$ such that $Z\subset D'_1$. Then we choose a smooth hyperplane section $D'_2$ which intersects $D'_1$ in a union of a line $L$ and a smooth rational curve $C$ as in Lemma \ref{lemma:Fano del Pezzo}. By Lemma \ref{lem-line-conic}, we see that $\dim\mathcal{D}(Y, D')=2$. Thus, 
$\mathrm{coreg}(Y, D')=0$. 
    
\subsubsection{
    \label{2:20}
    $X$ is a Fano variety \textnumero\,2.20,}     
    $Z$ is a twisted cubic,     
    $D'$ is as in the previous case, and $Z\subset D'_1$. We explain why we can choose such a smooth hyperplane section $D'_1$. Let $\Pi=\mathbb{P}^3$ be a projective subspace that contains $Z$. Since the Grassmannian is the intersection of quadrics, and $Y$ does not contain quadrics or quadric cones by Remark \ref{rem-no-planes}, we have that the intersection $Y\cap \Pi$ is either $Z$, or $Z\cup Z'$ where $Z'$ is a line. In both cases, using Lemma \ref{lem-smooth-section}, one can choose $D'_1$ such that it is smooth. Then we choose a smooth hyperplane section $D'_2$ which intersects $D'_1$ in a union of a line $L$ and a smooth rational curve $C$ as in Lemma \ref{lemma:Fano del Pezzo}. By Lemma \ref{lem-line-conic}, we see that $\dim\mathcal{D}(Y, D')=2$. Thus, 
    $\mathrm{coreg}(Y, D')=0$. 

\subsubsection{
    \label{2:14}
    $X$ is a Fano variety \textnumero\,2.14,}     
    $Z$ is an elliptic curve which is an intersection of two hyperplane sections,     
    $D'$ as in the previous case, and $Z\subset D'_1$. By Lemma \ref{lem-smooth-section}, we may assume that $D'_1$ is smooth. By Lemma \ref{lem-line-conic}, we see that $\dim\mathcal{D}(Y, D')=2$. Thus, $\dim\mathcal{D}(Y, D')=2$, and $\mathrm{coreg}(Y, D')=0$.

\subsection{Case $Y=V_4$} That is, $Y$ is an intersection of 
two quadrics in $\mathbb{P}^5$. 
\subsubsection{
    \label{2:19}
    $X$ is a Fano variety \textnumero\,2.19,}     
    $Z$ is a line,     
    $D'=D'_1+D'_2$ where $D'_1$ and $D'_2$ are smooth hyperplane sections intersecting in a union of a line $Z$ and a smooth rational cubic curve $C$ as constructed in Lemma \ref{lemma:Fano del
Pezzo}, so $Z\subset D'_1$. By Lemma \ref{lem-line-conic}, we see that $\dim\mathcal{D}(Y, D')=2$. Thus, 
$\mathrm{coreg}(Y, D')=0$.

\subsubsection{
    \label{2:16}
    $X$ is a Fano variety \textnumero\,2.16,}    
    $Z$ is a conic, 
    $D'=D'_1+D'_2$ where $D'_1$ and $D'_2$ are smooth hyperplane sections intersecting in a union of a line $L$ and a smooth rational cubic curve $C$ as constructed in Lemma \ref{lemma:Fano del
Pezzo}, and $Z\subset D'_1$. We explain why we can choose such a smooth hyperplane section $D'_1$. Let $\Pi=\mathbb{P}^2$ be a plane in $\mathbb{P}^5$ that contains $Z$. By Remark \ref{rem-no-planes}, we have that $\Pi\not\subset Y$. Note that $Y$ is defined by a pencil of quadrics whose restriction to $\Pi$ is a pencil of conics that contain $Z$ (in particular, this means that exactly one of the quadrics in the pencil contains $\Pi$). This shows that $\Pi\cap Y=Z$. By Lemma \ref{lem-smooth-section}, we can choose a hyperplane section $D'_1$ such that it is smooth and $Z\subset D'_1$. By Lemma \ref{lem-line-conic}, we see that $\dim\mathcal{D}(Y, D')=2$. Thus, 
    $\mathrm{coreg}(Y, D')=0$.

\subsubsection{
    \label{2:10}
    $X$ is a Fano variety \textnumero\,2.10,}     
    $Z$ is a smooth elliptic curve which is an intersection of two hyperplane
sections, 
    $D'$ as in the previous case, and $Z\subset D'_1$. By Lemma \ref{lem-line-conic}, we see that $\dim\mathcal{D}(Y, D')=2$. Thus, 
    $\mathrm{coreg}(Y, D')=0$.

\subsection{
Case $Y=V_3$} That is, $Y$ is a smooth cubic
hypersurface in $\mathbb{P}^4$. 
\subsubsection{
    \label{2:11}
    $X$ is a Fano variety \textnumero\,2.11,}     
    $Z$ is a line,     
    $D'=D'_1+D'_2$ where $D'_i$ are smooth hyperplane section intersecting in a union of a line $Z$ and a conic $C$ as constructed in Lemma \ref{lemma:Fano del
Pezzo}, and $Z\subset D'_1$. By Lemma \ref{lem-line-conic}, we see that $\dim\mathcal{D}(Y, D')=2$. Thus, $\mathrm{coreg}(Y, D')=0$.

\subsubsection{
    \label{2:5}
    $X$ is a Fano variety \textnumero\,2.5,}     
    $Z$ is a smooth plane cubic curve which is an intersection of two 
hyperplane sections,     
   $D'=D'_1+D'_2$ where $D'_i$ are smooth hyperplane section intersecting in a union of a line $L$ and a conic $C$ as constructed in Lemma \ref{lemma:Fano del
Pezzo}, and $Z\subset D'_1$. 
By Lemma \ref{lem-line-conic}, we see that $\dim\mathcal{D}(Y, D')=2$. Thus, $\dim\mathcal{D}(Y, D')=2$, and $\mathrm{coreg}(Y, D')=0$.

\subsection{Case $Y=V_2$} That is, $Y$ is a double cover of $\mathbb{P}^3$ ramified in a smooth quartic surface. 
        \label{2:3}
Then            $X$ is a Fano variety \textnumero\,2.3,     
    $Z$ is a smooth plane cubic curve which is an intersection of two hyperplane sections, 
            $D'=D'_1+D'_2$ where $D'_i$ are smooth hyperplane sections intersecting in a union of two smooth rational curves as constructed in Lemma \ref{lemma:Fano del
Pezzo}, and $Z\subset D'_1$. By Lemma \ref{lem-line-conic}, we see that $\dim\mathcal{D}(Y, D')=2$. Thus, 
$\mathrm{coreg}(Y, D')=0$.

\subsection{
Case $Y=V_1$} That is, $Y$ is a hypersurface of degree $6$ in $\mathbb{P}(1,1,1,2,3)$. 
    \label{cor-gen-2-1}
Then $X$ is a Fano variety \textnumero\,2.1, 
 $Z$ is an elliptic curve which is the intersection of two elements $H_1$ and $H_2$ such that $H_i\sim H$ where $-K_Y\sim 2H$,  
  $D'=D'_1+D'_2$ where $D'_i\sim H$ are smooth surfaces intersecting in a nodal curve $C$ as constructed in Lemma \ref{lemma:Fano del Pezzo1}. Here we use the assumption that $Y$ is general. More precisely, in Lemma \ref{lemma:Fano del Pezzo1} we have shown that for a general element in $|H|$, its anti-canonical linear system contains a nodal curve. Hence, starting from $D'_1=H_1$ and assuming it to be general we may find an element $D'_2\in |H|$ such that the intersection $D'_1\cap D'_2$ is a nodal curve. Applying Lemma \ref{lem-nodal-curve} we see that $\dim\mathcal{D}(Y, D')=2$. Thus, 
$\mathrm{coreg}(Y, D')=0$.

\

We have considered all smooth Fano threefolds with Picard number $2$ that can be realized as a blow up of some other smooth variety, and showed that they admit a boundary of coregularity $0$. This completes the proof.
\end{proof}

\section{Fano blow-ups with $\rho\geq 3$}
\label{sec-Fano-with-big-rho2}
It remains to consider Fano threefolds of Picard number at least~$3$ that can be realized as a blow up of some other threefold.

\begin{lemma}
\label{corollary:blow-ups3}
Let $X$ be a Fano threefold with Picard number at least $3$ and such that it can be realized as a blow up of some other smooth variety. Then~$\coreg(X)=0$.
\end{lemma}
\begin{proof}
We apply Lemma \ref{lemma:blow-up} and use its notation. In what follows, we construct the pairs $(Y, D')$ of coregularity $0$, and describe the subset $Z\subset Y$ such that $X=\mathrm{Bl}_Z Y$. Then the log pullback $(X, D)$ of $(Y, D')$ will also have coregularity $0$. 
\subsection{
Case
$Y=\mathbb{P}^3$.}

\subsubsection{
    \label{3:18}
    $X$ is a Fano variety \textnumero\,3.18,}     
    $Z=Z_1+Z_2$, $Z_1$ is a line and $Z_2$ is a plane conic,     
    $D'=D'_1+D'_2+D'_3+D'_4$ where $D'_i$ are general planes such that $Z_1\subset D'_1$ and $Z_2\subset D'_2$.
    
\subsubsection{
    \label{3:14}
    $X$ is a Fano variety \textnumero\,3.14,}     
    $Z$ is a smooth cubic curve contained in a plane $\Pi$, and $P$ is point such that $P\not\in \Pi$,     
    $D'$ is as in the previous case such that $Z\subset D'_1$ and $P\in D'_2\cap D'_3$.
    
\subsubsection{
    \label{3:12}
    $X$ is a Fano variety \textnumero\,3.12,}     
    $Z=Z_1+Z_2$, $Z_1$ is a line and $Z_2$ is a twisted cubic curve,     
    $D'=D'_1+D'_2+D'_3$ where $D'_1$ is a smooth quadric, $D'_2$ and $D'_3$ are general planes such that $Z_1\subset D'_2$ and  with $Z_2\subset D'_1$.

\subsubsection{
    \label{3:6}
    $X$ is a Fano variety \textnumero\,3.6,}      
    $Z=Z_1+Z_2$ where $Z_1$ is an elliptic curve of degree $4$ which is an intersection of two quadrics, $Z_2$ is a line,     
    $D'$ is as in the previous case such that $Z_1\subset D'_1$ and $Z_2\subset D'_2$.
    
\subsubsection{
    \label{4:6}
     $X$ is a Fano variety \textnumero\,4.6,}      
     $Z=Z_1+Z_2+Z_3$ where $Z_i$ are lines, 
$D'=D'_1+D'_2+D'_3+D'_4$ where $D'_i$ are general planes such that $Z_1\subset D'_1$, $Z_2\subset D'_2$ and
$Z_3\subset D'_3$.

\subsection{
Case when $Y$ is a smooth quadric hypersurface in
$\mathbb{P}^4$} 
\subsubsection{
\label{3:20}
$X$ is a Fano variety \textnumero\,3.20,} 
$Z=Z_1+Z_2$ where $Z_i$ are lines, 
$D'=D'_1+D'_2+D'_3$ where $D'_i$
are general hyperplane sections such that
$Z_1\subset D'_1$ and $Z_2\subset D'_2$. Note that we can choose smooth $D'_i$ with this property. 

\subsubsection{
\label{3:19}
$X$ is a Fano variety \textnumero\,3.19,} 
$P=P_1+P_2$ where $P_i$ are points that do not lie on a line that belongs to $Y$, 
$D'=D'_1+D'_2+D'_3$ where $D'_i$
are general hyperplane sections such that $P_1\in D'_1\cap
D'_2$ and $P_2\in D'_2\cap D'_3$.

\subsubsection{
\label{3:15}
$X$ is a Fano variety \textnumero\,3.15,}
$Z=Z_1+Z_2$ where $Z_1$ is a line, $Z_2$ is a plane conic, 
$D'=D'_1+D'_2+D'_3$ where $D'_i$
are general hyperplane sections such that $Z_1\subset D'_1$, $Z_2\subset D'_2$. By Lemma \ref{lem-smooth-section}, we can choose smooth $D'_i$ with these conditions.

\subsubsection{
\label{3:10}
$X$ is a Fano variety \textnumero\,3.10,}
$Z=Z_1+Z_2$ where $Z_i$ are plane conics, 
$D'=D'_1+D'_2+D'_3$ where $D'_i$
are general hyperplane sections such that $Z_1\subset D'_1$ and $Z_2\subset D'_2$. By Lemma \ref{lem-smooth-section}, we can choose smooth $D'_i$ with these conditions.

\subsubsection{
\label{4:4}
$X$ is a Fano variety \textnumero\,4.4.} In this case, we apply Lemma \ref{lemma:quadrics} twice. 
Let $P_1$ and $P_2$ be the points that do not lie on a line in $Y$, $Z$ be a conic that passes through $P_1$ and $P_2$. 
Put $D'=D'_1+D'_2+D'_3$ where $D'_i$
are hyperplane sections. To obtain $X$, we first
we blow up the points $P_1$ and $P_2$, and then blow up the strict transform
of the conic $Z$. We can choose general smooth $D'_i$ with the conditions
that $P_1\subset D'_1\cap D'_3$, $P_2\subset D'_1\cap D'_3$ and $Z_3\subset D'_3$. By Lemma \ref{lem-smooth-section}, we can choose smooth $D'_3$ with these conditions.

\subsection{
Case when $Y$ is a smooth divisor of bidegree $(1,1)$
in $\mathbb{P}^2\times\mathbb{P}^2$} 
We denote by $\pi_1$ and $\pi_2$ the natural projections from $Y$ to the first and the second copy of $\mathbb{P}^2$, respectively. Note that $\pi_i$ are $\mathbb{P}^1$-bundles.  
\subsubsection{
\label{3:24}
$X$ is a Fano variety \textnumero\,3.24,} 
$Z$ is a curve
which is a fiber of the projection $\pi_1$, 
$D'=D'_1+D'_2+D'_3$ where $D'_i$ are general such that
$D'_1$ has bidegree $(1, 1)$, $D'_2$ has bidegree $(1,0)$, $D'_3$ has bidegree $(0, 1)$, and
$Z\subset D'_2$. By Lemma \ref{lem-smooth-section}, we can choose $D'_2$ to be smooth.

\subsubsection{
\label{3:13}
$X$ is a Fano variety \textnumero\,3.13,} 
$Z$ is a curve of
bidegree $(2, 2)$ such that the composition of the embedding $Z\subset Y$ with the projections $\pi_i$ 
is also an embedding, so $\pi_i(C)$ are irreducible conics, 
$D'=D'_1+D'_2+D'_3$ where $D'_i$ are general such that
$D'_1$ has bidegree $(0, 2)$, $D'_2$ has bidegree $(1,0)$, $D'_3$ has bidegree $(1, 0)$, and $Z\subset D'_1$. 

\subsubsection{
\label{3:7}
$X$ is a Fano variety \textnumero\,3.7,} 
$Z$ is a smooth elliptic curve that is an intersection of two divisors $H_1$ and $H_2$ such that each $H_i$ has bidegree $(1,1)$, 
$D'=D'_1+D'_2+D'_3$ where $D'_i$ are general such that $D'_1$ has bidegree $(1,1)$, $D'_2$ has bidegree $(1,0)$, $D'_3$ has bidegree $(0,1)$ such that $Z=D'_1\cap H_2$. By Lemma \ref{lem-smooth-section}, we can choose $D'_1$ to be smooth.

\subsubsection{
\label{4:7}
$X$ is a Fano variety \textnumero\,4.7,} 
$Z=Z_1+Z_2$ where $Z_1$ is a curve which is a fiber of the projection $\pi_1$, $Z_2$ is a curve which is a fiber of the projection $\pi_2$, 
$D'=D'_1+D'_2+D'_3$ where $D'_i$ are general such that
$D'_1$ has bidegree $(1, 1)$, $D'_2$ has bidegree $(1,0)$, $D'_3$ has bidegree $(0, 1)$, and $Z_1\subset D'_2$ and $Z_2\subset D'_3$. By Lemma \ref{lem-smooth-section}, we can choose $D'_i$ to be smooth.

\subsection{
Case $Y=\mathbb{P}^1\times \mathbb{P}^2$}
We denote by $\pi_1$ and $\pi_2$ the natural projections from $Y$ to $\mathbb{P}^1$ and $\mathbb{P}^2$, respectively. 
\subsubsection{
\label{3:22}
        $X$ is a Fano variety \textnumero\,3.22,}         
        $Z$ is a conic in a fiber of $\pi_2$,         
$D'=D'_1+D'_2+D'_3+D'_4+D'_5$ where $D'_1$ and $D'_2$ have bidegree $(1,0)$,
$D'_3$, $D'_4$ and $D'_5$ have bidegree $(0,1)$, and $Z\subset D'_1$.
                
\subsubsection{
\label{3:21}
        $X$ is a Fano variety \textnumero\,3.21,} 
$Z$ is a curve of bidegree $(2, 1)$, 
$D'=D'_1+D'_2+D'_3+D'_4+D'_5$ where $D'_1$ and $D'_2$ have bidegree $(1,0)$,
$D'_3$, $D'_4$ and $D'_5$ have bidegree $(0,1)$, and $Z\subset D'_3$.

\subsubsection{
\label{3:17}
        $X$ is a Fano variety \textnumero\,3.17,}         
        $Z$ is a smooth curve of genus $0$ which is a complete intersection of two divisors of bidegree $(1, 1)$,         
        $D'=D'_1+D'_2+D'_3+D'_4$ where $D'_1$ has bidegree $(1, 1)$, $D'_2$ has bidegree $(1, 0)$, $D'_3$ and $D'_4$ have bidegree $(0, 1)$, and $Z\subset D'_1$. By Lemma \ref{lem-smooth-section}, we can choose $D'_1$ to be smooth.
        
\subsubsection{
\label{3:8}
        $X$ is a Fano variety \textnumero\,3.8,}         
        $Z$ is a smooth rational curve which is an intersection of divisors of bidegree $(0, 2)$ and $(1, 2)$,         
        $D'=D'_1+D'_2+D'_3$ where $D'_1$ has bidegree $(1, 0)$, $D'_2$ has bidegree $(0, 1)$, $D'_3$ has bidegree $(0, 2)$ such that $Z\subset D'_3$.    
        
\subsubsection{
\label{3:5}
        $X$ is a Fano variety \textnumero\,3.5,} 
                $Z$ is a curve of
bidegree $(5, 2)$ such that the composition of the embedding $Z\subset Y$ with the projection $\pi_2$ is an embedding,
        $D'=D'_1+D'_2+D'_3+D'_4$ where $D'_1$ and $D'_2$ have bidegree $(1, 0)$, $D'_3$
has bidegree $(0,1)$, $D'_4$ has bidegree $(0,2)$, so we pick $D'_4$ such that $Z\subset D'_4$. 

\subsubsection{
\label{3:3}
        $X$ is a Fano variety \textnumero\,3.3,}         
        $Z$ is a smooth curve of genus $3$ which is a complete intersection of two divisors of bidegree $(1, 2)$,         
        $D'=D'_1+D'_2+D'_3$ where $D'_1$ has bidegree $(1, 2)$, $D'_2$ has bidegree $(1, 0)$, $D'_3$ has bidegree $(0, 1)$ such that $Z\subset D'_1$. 
   
\subsubsection{
        \label{4:5}
        $X$ is a Fano variety \textnumero\,4.5,}          
        $Z=Z_1+Z_2$ where $Z_1$ is a curve of bidegree $(2, 1)$, $Z_2$ is a curve of bidegree $(1, 0)$,         
        $D'=D'_1+D'_2+D'_3+D'_4+D'_5$ where $D'_1$ and $D'_2$ have bidegree $(1,0)$,
$D'_3$, $D'_4$ and $D'_5$ have bidegree $(0,1)$ such that $Z_1\subset D'_3$ and $Z_2\subset D'_4$.

\subsection{
Case $Y=\mathbb{P}^1\times\mathbb{P}^1\times\mathbb{P}^1$} 
\subsubsection{
\label{4:13}
        $X$ is a Fano variety \textnumero\,4.13,}         
        $Z$ is a curve of
type $(1, 1, 3)$, 
$D'=D'_1+D'_2+D'_3+D'_4$ where $D'_1$ and $D'_2$ have type $(1, 1, 0)$,
$D'_3$ and $D'_4$ have type $(0, 0, 1)$ such that $Z\subset D'_1$.

\subsubsection{
\label{4:8}
        $X$ is a Fano variety \textnumero\,4.8,}        
        $Z$ is a curve of
type $(0, 1, 1)$, $D'=D'_1+D'_2+D'_3+D'_4$ where $D'_1$ and $D'_2$ have type $(0, 1, 1)$,
$D'_3$ and $D'_4$ have type $(1, 0, 0)$ such that $Z\subset D'_1$.

\subsubsection{
\label{4:3}
        $X$ is a Fano variety \textnumero\,4.3,}         
         $Z$ is a curve of
type $(1, 1, 2)$,
$D'=D'_1+D'_2+D'_3+D'_4$ where $D'_1$ and $D'_2$ have type $(1, 1, 0)$,
$D'_3$ and $D'_4$ have type $(0, 0, 1)$ such that $Z\subset D'_1$.

\subsubsection{
\label{4:1}
        $X$ is a Fano variety \textnumero\,4.1,}       
        $Z$ is a curve which is an intersection of two divisors of type $(1,1,1)$,        
$D'=D'_1+D'_2+D'_3+D'_4$ where $D'_2$ has type $(1,0,0)$, $D'_3$ has type $(0,1,0)$, $D'_4$ has type $(0, 0, 1)$, and $D'_1$ has type $(1, 1, 1)$ with the condition that $Z\subset D'_1$. By Lemma \ref{lem-smooth-section}, we can choose $D'_1$ to be smooth.

\subsection{
Case $Y=V_7$} That is, $Y$ is the blow up of a point $P_1$ in
$\mathbb{P}^3$.
In this case, we apply Lemma \ref{lemma:blow-up}
twice to conclude that $\coreg(X)=0$. First, we construct a pair $(\mathbb{P}^3, D'')$ of coregularity zero, then we observe that its log pullbacks $(Y, D')$ and $(X, D)$ also have coregularity zero.
\subsubsection{
\label{3:23}
        $X$ is a Fano variety \textnumero\,3.23,}         
        $Z'$ is a conic on $\mathbb{P}^3$ passing through $P_1$,        
        $D''=D''_1+D''_2+D''_3+D''_4$ where
$D''_i$ are general planes with the property $P_1\in D''_1\cap D''_2$ and
$Z'\subset D''_1$.

\subsubsection{
\label{3:16}
        $X$ is a Fano variety \textnumero\,3.16,}         
        $Z'$ is a twisted cubic on $\mathbb{P}^3$ passing
through $P_1$, 
$D''=D''_1+D''_2+D''_3$ where
$D''_1$ is a smooth quadric, $D''_2$ and $D''_3$ are general planes with the
property $P_1\in D''_1\cap D''_2$ such that $Z'\subset D''_1$.

\subsubsection{
\label{3:11}
        $X$ is a Fano variety \textnumero\,3.11,}         
        $Z'$ is a smooth elliptic curve on $\mathbb{P}^3$ which is an intersection of two quadrics 
passing through $P_1$, 
$D''=D''_1+D''_2+D''_3$ where
$D''_1$ is a smooth quadric, $D''_2$ and $D''_3$ are general planes with the
property $P_1\in D''_1\cap D''_2$, and such that $Z'\subset D''_1$.

\subsection{
\label{lemma:3-4}
Case when $Y$ is a Fano threefold in the family \textnumero\,$2.2$} That is, $Y$ is a double cover $f\colon Y\to \mathbb{P}^1\times\mathbb{P}^2$ ramified in a smooth divisor of bidegree $(2,4)$.
Then $X$ is a Fano threefold in the family \textnumero\,$3.4$, $Z$ is a smooth fiber of the composition $Y \to \mathbb{P}^1 \times \mathbb{P}^2 \to
\mathbb{P}^2$ of the double cover $f$ with the projection onto the second factor.
According to the proof of Lemma \ref{lemma:2-2}, there is a boundary $D'=D'_1+D'_2$ on $Y$
of coregularity~$0$, and we can pick $D'_2$ such that $Z\subset
D'_2$. 

\subsection{
\label{lemma:3-9}
Case 
$Y=\mathbb{P}_{\mathbb{P}^2}(\oo\oplus\oo(2))$} 
Then $X$ is a Fano threefold in the family \textnumero\,$3.9$, 
$Z$ is a smooth quartic
curve that is contained in a projective plane which is equivalent
to a tautological divisor $H$ on $Y$. We have
$
K_Y \sim -2H - F
$
where $F$ is the preimage of a line via the projection $Y=\mathbb{P}_{\mathbb{P}^2}(\oo\oplus\oo(2))\to \mathbb{P}^2$. 
Put $D'=D'_1+D'_2+D'_3$ where $D'_1\sim H$ such that $Z\subset D'_1$, $D'_2\sim H$ is general and  
$D'_3\sim F$. 

\subsection{
\label{lemma:4-2}
Case $Y=\mathbb{P}_{\mathbb{P}^1\times\mathbb{P}^1}(\OOO\oplus\OOO(1, 1))$} 
Then $X$ is a Fano threefold in the family \textnumero\,$4.2$, 
$Z$ is an elliptic curve such that is contained in a surface equivalent to a tautological
divisor $H$ on $Y$. We have
$
K_Y \sim -2H - F_1 - F_2
$
where $F_1$ and $F_2$ are the preimages of two distinct rulings via the projection $Y=\mathbb{P}_{\mathbb{P}^1\times\mathbb{P}^1}(\OOO\oplus\OOO(1, 1))\to\mathbb{P}^1\times\mathbb{P}^1$. Put $D'=D'_1+D'_2+D'_3+D'_4$ where $D'_1\sim H$ such that $Z\subset D'_1$, $D'_2\sim H$ is general,
$D'_3\sim F_1$ and $D'_4\sim F_2$ are general as well. 

\subsection{
\label{5:1}
Case when $Y$ is a Fano threefold in the family
\textnumero\,$5.1$} 
That is, $Y$ is the blow up of a conic $C$ on a smooth
quadric hypersurface $Q\subset \mathbb{P}^4$. In this case, we apply
Lemma \ref{lemma:blow-up} twice. Consider a pair $(Q,
D'')$ of coregularity $0$ where $D''=D''_1+D''_2+D''_3$, $D'_i$ are general
hyperplane sections with the property $C\subset D'_1\cap D'_2$. Put $D'=D'_1+D'_2+D'_3+D'_4$ where $D'_1$, $D''_2$ and $D'_3$ are strict preimages of $D''_1$, $D''_2$ and $D''_3$, respectively, and 
$D'_4=E$ is the exceptional divisor of the blow up $Y\to Q$. Then the
log pullback $(Y, D')$ also has coregularity $0$. Now, let
$Z=Z_1+Z_2+Z_3$ be the union of three fibers of the blow up $Y\to Q$, so $Z\subset
D'_4$. 

\

We have considered all smooth Fano threefolds with Picard number at least $3$ that can be realized as a blow up of some other smooth variety, and showed that they admit a boundary of coregularity $0$. This completes the proof.
\end{proof}

\newpage
\section{The table}
\Small{
We comment on the notation that we use in the table. In the column ``Coreg.'', we write down the value of the coregularity of the corresponding Fano variety. When we write ``gen.'' we mean the general element in the corresponding deformation family. In particular, in the families 1.2 and 1.5 we work only with the first description of the corresponding Fano variety given in the column ``Brief description''. For the definition of $\mathrm{coreg}_1$, that is, the first coregularity, we refer to Definition \ref{defin-regularity}. 
}
\label{sec-the-table}
\small{
\begin{longtable}{|c|c|p{9cm}|c|c|}
\caption{Coregularity of smooth Fano threefolds}\label{table:Fanos}\\
\hline Family & $-K_X^3$ &  Brief description & Coreg. & Proof \\
\hline $1.1$ & $2$ & a hypersurface in $\mathbb{P}(1,1,1,1,3)$ of
degree $6$ & $\mathrm{gen.}\geq 1$ & Proposition \ref{cor-gen-sextic-double-solid} \\
\hline $1.2$ & $4$ & a hypersurface in $\mathbb{P}^4$ of degree
$4$ or\hfill\break a double cover of smooth quadric in
$\mathbb{P}^{4}$
branched over a surface of degree $8$ & $\mathrm{gen.}\geq 1$ & Proposition \ref{quartic-coreg-0} \\
\hline $1.3$ & $6$ & a complete intersection of a quadric and a
cubic in
$\mathbb{P}^{5}$ & $\mathrm{gen.coreg}_1=2$ & Proposition \ref{prop-23-reg1} \\
\hline $1.4$ & $8$ & a complete intersection of three quadrics
$\mathbb{P}^{6}$ & $\mathrm{gen.coreg}_1=2$ & Proposition \ref{prop-222-reg1} \\
\hline $1.5$ & $10$ & a section of
$\mathrm{Gr}(2,5)\subset\mathbb{P}^9$ by
quadric and linear subspace of dimension~$7$ or \hfill\break
a double cover of 1-15 with branch locus an anti-canonical divisor
& $\mathrm{gen.}\leq 1$ &  Lemma~\ref{lemma:1-5} \\
\hline $1.6$ & $12$ & a section of the Hermitian symmetric space
$M=G/P\subset\mathbb{P}^{15}$\hfill\break of type DIII  by a
linear
subspace of dimension~$8$ & $\mathrm{gen.}\,0$ & Lemma~\ref{lemma:1-6} \\
\hline $1.7$ & $14$ & a section of
$\mathrm{Gr}(2,6)\subset\mathbb{P}^{14}$
by a linear subspace of codimension~$5$ & $\mathrm{gen.}\,0$ & Lemma~\ref{lemma:1-7} \\
\hline $1.8$ & $16$ & a section of the Hermitian symmetric space
$M=G/P\subset \mathbb{P}^{19}$\hfill\break of type CI  by a linear
subspace
of dimension~$10$ & gen. $0$ & Lemma~\ref{lemma:1-8} \\
\hline $1.9$ & $18$ & a section of the $5$-dimensional rational
homogeneous contact\hfill\break manifold
$G_2/P\subset\mathbb{P}^{13}$  by
a linear subspace of dimension~$11$ & gen. $0$ & Lemma~\ref{lemma:1-9} \\
\hline $1.10$ & $22$ & a zero locus of three sections of the rank
$3$ vector bundle $\bigwedge^2\mathcal{Q}$,\hfill\break where
$\mathcal{Q}$ is
the universal quotient bundle on $\mathrm{Gr}(7,3)$ & $\mathrm{gen.}\,0$ &
Lemma~\ref{lemma:1-10} \\
\hline $1.11$ & $8$ & $V_{1}$ that is a hypersurface in
$\mathbb{P}(1,1,1,2,3)$ of degree $6$ & $\mathrm{gen.}\,0$ &
Lemma~\ref{lemma:Fano del Pezzo1} \\
\hline $1.12$ & $16$ & $V_{2}$ that is a hypersurface in
$\mathbb{P}(1,1,1,1,2)$ of degree $4$ & $0$ & Lemma~\ref{lemma:Fano
del Pezzo} \\
\hline $1.13$ & $24$ & $V_{3}$ that is a hypersurface in
$\mathbb{P}^{4}$ of degree $3$ & $0$ & Lemma~\ref{lemma:Fano del
Pezzo} \\
\hline $1.14$ & $32$ & $V_{4}$ that is a complete intersection of two
quadrics in $\mathbb{P}^{5}$ & $0$ & Lemma~\ref{lemma:Fano del Pezzo}
\\
\hline $1.15$ & $40$ & $V_{5}$ that is a section of
$\mathrm{Gr}(2,5)\subset\mathbb{P}^9$ by linear subspace of
codimension $3$ & $0$ & Lemma~\ref{lemma:Fano del Pezzo} \\
\hline $1.16$ & $54$ & $Q$ that is a hypersurface in $\mathbb{P}^{4}$
of degree $2$ & $0$ & Lemma~\ref{lemma:quadrics} \\
\hline $1.17$ & $64$ & $\mathbb{P}^{3}$ & $0$ & Lemma~\ref{lemma:Pn} \\
\hline $2.1$ & $4$ & a blow up of the Fano threefold $V_1$ along an
elliptic curve\hfill\break that is an intersection of two divisors
from $|-\frac{1}{2}K_{V_1}|$ & gen. $0$ & Lemma 
\ref{cor-gen-2-1} \\%
\hline $2.2$ & $6$ & a double cover of
$\mathbb{P}^1\times\mathbb{P}^2$
whose branch locus is a divisor of bidegree $(2, 4)$ & $0$ & Lemma
\ref{lemma:2-2} \\%
\hline $2.3$ & $8$ & the blow up of the Fano threefold $V_2$ along an
elliptic curve\hfill\break that is an intersection of two divisors
from $|-\frac{1}{2}K_{V_2}|$ & $0$ & Lemma
\ref{2:3} \\
\hline $2.4$ & $10$ & the blow up of $\mathbb{P}^3$ along an
intersection of two cubics & $0$ & Lemma \ref{2:4}
\\
\hline $2.5$ & $12$ & the blow up of the threefold
$V_3\subset\mathbb{P}^4$ along a plane cubic & $0$ & Lemma
\ref{2:5} \\
\hline $2.6$ & $12$ & a divisor on
$\mathbb{P}^2\times\mathbb{P}^2$ of bidegree $(2, 2)$
or\hfill\break a double cover of $W$ whose branch locus
is a surface in $|-K_W|$  & $0$ & Lemma \ref{lemma:2-6} \\
\hline $2.7$ & $14$ & the blow up of $Q$ along the intersection of
two divisors from $|\mathcal{O}_Q (2)|$ & $0$ & Lemma
\ref{2:7} \\
\hline $2.8$ & $14$ & a double cover of $V_7$ whose branch locus
is a surface
in $|-K_{V_7}|$ & $0$ & Lemma \ref{lemma:2-8} \\
\hline $2.9$ & $16$ & the blow up of $\mathbb{P}^3$ along a curve
of degree $7$ and genus~$5$\hfill\break which is an intersection
of cubics & $0$ & Lemma \ref{2:9} \\
\hline $2.10$ & $16$ & the blow up of $V_4\subset\mathbb{P}^5$
along an elliptic curve\hfill\break which is an intersection of
two hyperplane
sections & $0$ & Lemma \ref{2:10} \\
\hline $2.11$ & $18$ & the blow up of $V_3$ along a line & $0$ &
Lemma \ref{2:11} \\
\hline $2.12$ & $20$ & the blow up of $\mathbb{P}^3$ along a curve
of degree $6$ and genus~$3$\hfill\break which is an intersection
of cubics & $0$ & Lemma \ref{2:12} \\
\hline $2.13$ & $20$ & the blow up of $Q\subset\mathbb{P}^4$ along
a curve of degree $6$ and genus $2$ & $0$ & Lemma \ref{2:13} \\
\hline $2.14$ & $20$ & the blow up of $V_5\subset\mathbb{P}^6$ along
an elliptic curve\hfill\break which is an intersection of two
hyperplane sections & $0$ & Lemma \ref{2:14} \\
\hline $2.15$ & $22$ & the blow up of $\mathbb{P}^3$ along the
intersection of a quadric and a cubic surfaces & $0$ & Lemma
\ref{2:15} \\
\hline $2.16$ & $22$ & the blow up of $V_4\subset\mathbb{P}^5$
along a conic & $0$ & Lemma \ref{2:16} \\
\hline $2.17$ & $24$ & the blow up of $Q\subset\mathbb{P}^4$ along
an elliptic curve of degree~$5$ & $0$ & Lemma \ref{corollary:blow-ups} \\
\hline $2.18$ & $24$ & a double cover of
$\mathbb{P}^1\times\mathbb{P}^2$ whose branch locus is a divisor of
bidegree $(2, 2)$ & $0$ & Lemma \ref{lemma:2-18} \\
\hline $2.19$ & $26$ & the blow up of $V_4\subset\mathbb{P}^5$ along a
line & $0$ & Lemma \ref{2:19} \\
\hline $2.20$ & $26$ & the blow up of $V_5\subset\mathbb{P}^6$
along a twisted cubic & $0$ & Lemma \ref{2:20} \\
\hline $2.21$ & $28$ & the blow up of $Q\subset\mathbb{P}^4$ along
a twisted quartic & $0$ & Lemma \ref{2:21} \\
\hline $2.22$ & $30$ & the blow up of $V_5\subset\mathbb{P}^6$
along a conic & $0$ & Lemma \ref{2:22} \\
\hline $2.23$ & $30$ & the blow up of $Q\subset\mathbb{P}^4$ along a
curve of degree $4$ that is an intersection of a surface in
$|\mathcal{O}_{\mathbb{P}^{4}}(1)\vert_{Q}|$ and a surface in
$|\mathcal{O}_{\mathbb{P}^{4}}(2)\vert_{Q}|$ & $0$ & Lemma
\ref{2:23} \\
\hline $2.24$ & $30$ & a divisor on $\mathbb{P}^2\times\mathbb{P}^2$
of bidegree $(1, 2)$ & $0$ & Lemma \ref{lemma:2-24} \\
\hline $2.25$ & $32$ & the blow up of $\mathbb{P}^3$ along an elliptic
curve which is an intersection of two quadrics & $0$ & Lemma
\ref{2:25} \\
\hline $2.26$ & $34$ & the blow up of the threefold
$V_5\subset\mathbb{P}^6$ along a line & $0$ & Lemma
\ref{2:26} \\
\hline $2.27$ & $38$ & the blow up of $\mathbb{P}^3$ along a twisted
cubic  & $0$ & Lemma \ref{2:27} \\
\hline $2.28$ & $40$ & the blow up of $\mathbb{P}^3$ along a plane
cubic & $0$ & Lemma \ref{2:28} \\
\hline $2.29$ & $40$ & the blow up of $Q\subset\mathbb{P}^4$ along a
conic & $0$ & Lemma \ref{2:29} \\
\hline $2.30$ & $46$ & the blow up of $\mathbb{P}^3$ along a conic   &
$0$ & Lemma \ref{2:30} \\
\hline $2.31$ & $46$ & the blow up of $Q\subset\mathbb{P}^4$ along a
line & $0$ & Lemma \ref{2:31} \\
\hline $2.32$ & $48$ & $W$ that is a divisor on
$\mathbb{P}^2\times\mathbb{P}^2$ of bidegree $(1, 1)$  & $0$ &
Lemma~\ref{lemma:2-32} \\
\hline $2.33$ & $54$ & the blow up of $\mathbb{P}^3$ along a line  &
$0$ & Remark~\ref{remark:toric} \\
\hline $2.34$ & $54$ & $\mathbb{P}^1\times\mathbb{P}^2$  & $0$ &
Remark~\ref{remark:toric} \\
\hline $2.35$ & $56$ &
$V_7\cong\mathbb{P}(\mathcal{O}_{\mathbb{P}^2}\oplus\mathcal{O}_{\mathbb{P}^2}(1))$
 & $0$ & Remark~\ref{remark:toric} \\
\hline $2.36$ & $62$ &
$\mathbb{P}(\mathcal{O}_{\mathbb{P}^2}\oplus\mathcal{O}_{\mathbb{P}^2}(2))$
 & $0$ & Remark~\ref{remark:toric} \\
\hline $3.1$ & $12$ & a double cover of
$\mathbb{P}^1\times\mathbb{P}^1\times\mathbb{P}^1$ branched in a
divisor of type $(2, 2, 2)$ & $0$ & Lemma \ref{lemma:3-1} \\
\hline $3.2$ & $14$ & a divisor on a $\mathbb{P}^{2}$-bundle
$\mathbb{P}(\mathcal{O}_{\mathbb{P}^1\times\mathbb{P}^1}\oplus\mathcal{O}_{\mathbb{P}^1\times\mathbb{P}^1}(-1,-1)\oplus\mathcal{O}_{\mathbb{P}^1\times\mathbb{P}^1}(-1,-1))$
such that $X\in|L^{{}\otimes
2}\otimes\mathcal{O}_{\mathbb{P}^{1}\times\mathbb{P}^{1}}(2,3)|$,
where $L$ is the tautological line bundle & $0$ & Lemma
\ref{lemma:3-2} \\
\hline $3.3$ & $18$ & a divisor on
$\mathbb{P}^1\times\mathbb{P}^1\times\mathbb{P}^2$ of type $(1,
1, 2)$ & $0$ & Lemma \ref{3:3} \\
\hline $3.4$ & $18$ & the blow up of the Fano threefold $Y$ from the
family $\textnumero\ 2.18$ along a smooth fiber of the
composition $Y\to\mathbb{P}^1\times\mathbb{P}^2\to\mathbb{P}^2$ of the
double cover with the projection & $0$ & Lemma \ref{lemma:3-4}
\\
\hline $3.5$ & $20$ & the blow up of $\mathbb{P}^1\times\mathbb{P}^2$
along a curve $C$ of bidegree $(5, 2)$\hfill\break such that the
composition  $C\hookrightarrow\mathbb{P}^1\times\mathbb{P}^2\to\mathbb{P}^2$
is an embedding & $0$ & Lemma \ref{3:5} \\
\hline $3.6$ & $22$ & the blow up of $\mathbb{P}^3$ along a disjoint
union of a line and an elliptic curve of degree~$4$ & $0$ & Lemma
\ref{3:6} \\
\hline $3.7$ & $24$ & the blow up of the threefold $W$ along an
elliptic curve\hfill\break that is an intersection of two  divisors
from $|-\frac{1}{2}K_W|$  & $0$ & Lemma \ref{3:7}
\\
\hline $3.8$ & $24$ & a divisor in
$|(\alpha\circ\pi_1)^*(\mathcal{O}_{\mathbb{P}^2}(1))\otimes\pi_2^*(\mathcal{O}_{\mathbb{P}^2}(2))|$,
where $\pi_{1}\colon\mathbb{F}_1\times\mathbb{P}^2\to\mathbb{F}_1$
and $\pi_{2}\colon\mathbb{F}_1\times\mathbb{P}^2\to\mathbb{P}^2$ are
projections, and $\alpha\colon\mathbb{F}_1\to\mathbb{P}^2$ is a blow
up of a point & $0$ & Lemma
\ref{3:8} \\
\hline $3.9$ & $26$ & the blow up of a cone $W_4\subset\mathbb{P}^6$
over the Veronese surface  $R_4\subset\mathbb{P}^5$\hfill\break with
center in a disjoint union of the vertex and a quartic on
$R_4\cong\mathbb{P}^2$ & $0$ & Lemma
\ref{lemma:3-9} \\
\hline $3.10$ & $26$ & the blow up of $Q\subset\mathbb{P}^4$ along a
disjoint union of two conics & $0$ & Lemma
\ref{3:10} \\
\hline $3.11$ & $28$ & the blow up of the threefold $V_7$ along an
elliptic curve\hfill\break that is an intersection of  two divisors
from $|-\frac{1}{2}K_{V_7}|$ & $0$ & Lemma
\ref{3:11} \\
\hline $3.12$ & $28$ & the blow up of $\mathbb{P}^3$ along a disjoint
union of a line and a twisted cubic & $0$ & Lemma
\ref{3:12} \\
\hline $3.13$ & $30$ & the blow up of
$W\subset\mathbb{P}^2\times\mathbb{P}^2$ along a curve $C$ of
bidegree $(2, 2)$\hfill\break such that
$\pi_{1}(C)\subset\mathbb{P}^2$ and
$\pi_{2}(C)\subset\mathbb{P}^{2}$ are irreducible
conics,\hfill\break where $\pi_{1}\colon W\to\mathbb{P}^2$ and
$\pi_{2}\colon W\to\mathbb{P}^2$ are
natural projections & $0$ & Lemma \ref{3:13} \\
\hline $3.14$ & $32$ & the blow up of $\mathbb{P}^3$ along a
disjoint union of a plane cubic curve that is
contained in a plane
$\Pi\subset\mathbb{P}^{3}$ and a point that is not contained in $\Pi$
& $0$ & Lemma \ref{3:14} \\
\hline $3.15$ & $32$ & the blow up of $Q\subset\mathbb{P}^4$ along a
disjoint union of a line and a conic & $0$ & Lemma
\ref{3:15} \\
\hline $3.16$ & $34$ & the blow up of $V_7$ along a proper
transform via the blow up $\alpha\colon
V_7\to\mathbb{P}^3$ of a twisted cubic
passing through the center of the blow up $\alpha$ & $0$ & Lemma
\ref{3:16} \\
\hline $3.17$ & $36$ & a divisor on
$\mathbb{P}^1\times\mathbb{P}^1\times\mathbb{P}^2$ of type $(1,
1, 1)$ & $0$ & Lemma \ref{3:3} \\
\hline $3.18$ & $36$ & the blow up of $\mathbb{P}^3$ along a disjoint
union of a line and a conic & $0$ & Lemma
\ref{3:18} \\
\hline $3.19$ & $38$ & the blow up of $Q\subset\mathbb{P}^4$ at two
non-collinear points & $0$ & Lemma \ref{3:19} \\
\hline $3.20$ & $38$ & the blow up of $Q\subset\mathbb{P}^4$ along a
disjoint union of two lines & $0$ & Lemma
\ref{3:20} \\
\hline $3.21$ & $38$ & the blow up of $\mathbb{P}^1\times\mathbb{P}^2$
along a curve of bidegree  $(2, 1)$ & $0$ & Lemma
\ref{3:21} \\
\hline $3.22$ & $40$ & the blow up of $\mathbb{P}^1\times\mathbb{P}^2$
along a conic in a fiber of the projection
$\mathbb{P}^{1}\times\mathbb{P}^2\to\mathbb{P}^1$ & $0$ & Lemma
\ref{3:22} \\
\hline $3.23$ & $42$ & the blow up of $V_7$ along a proper transform
via the blow up $\alpha\colon V_7\to\mathbb{P}^3$ of an
irreducible conic passing through the center of the blow up $\alpha$ &
$0$ & Lemma \ref{3:23} \\
\hline $3.24$ & $42$ & $W\times_{\mathbb{P}^2}\mathbb{F}_1$, where
$W\to\mathbb{P}^2$ is a $\mathbb{P}^1$-bundle and
$\mathbb{F}_1\to\mathbb{P}^2$ is the
blow up & $0$ & Lemma \ref{3:24} \\
\hline $3.25$ & $44$ & the blow up of $\mathbb{P}^3$ along a disjoint
union of two lines & $0$ & Remark~\ref{remark:toric} \\
\hline $3.26$ & $46$ & the blow up of $\mathbb{P}^3$ with center in a
disjoint union of a point and a line  & $0$ &
Remark~\ref{remark:toric} \\
\hline $3.27$ & $48$ &
$\mathbb{P}^1\times\mathbb{P}^1\times\mathbb{P}^1$  & $0$ &
Remark~\ref{remark:toric} \\
\hline $3.28$ & $48$ & $\mathbb{P}^1\times\mathbb{F}_1$  & $0$ &
Remark~\ref{remark:toric} \\
\hline $3.29$ & $50$ & the blow up of the Fano threefold $V_7$
along a line in $E\cong\mathbb{P}^2$,\hfill\break where $E$ is the
exceptional divisor of the
blow up $V_7\to\mathbb{P}^3$  & $0$ & Remark~\ref{remark:toric} \\
\hline $3.30$ & $50$ & the blow up of $V_7$ along a proper transform
via the blow up $\alpha\colon V_7\to\mathbb{P}^3$ of a
line that passes through the center of the blow up $\alpha$   & $0$ &
Remark~\ref{remark:toric} \\
\hline $3.31$ & $52$ & the blow up of a cone over a smooth quadric in
$\mathbb{P}^3$ at the vertex & $0$ & Remark~\ref{remark:toric} \\
\hline $4.1$ & $24$ & divisor on
$\mathbb{P}^1\times\mathbb{P}^1\times\mathbb{P}^1\times\mathbb{P}^1$
of multidegree $(1, 1, 1, 1)$ & $0$ & Lemma \ref{4:1} \\
\hline $4.2$ & $28$ & the blow up of the cone over a smooth quadric
$S\subset\nolinebreak\mathbb{P}^3$\hfill\break along a disjoint union
of the vertex and an elliptic curve on $S$  & $0$ & Lemma
\ref{lemma:4-2} \\
\hline $4.3$ & $30$ & the blow up of
$\mathbb{P}^1\times\mathbb{P}^1\times\mathbb{P}^1$ along a curve of
type $(1, 1, 2)$  & $0$ & Lemma \ref{4:3} \\
\hline $4.4$ & $32$ & the blow up of the smooth Fano threefold $Y$
from the family $\textnumero\ 3.19$ along the proper
transform of a conic on the quadric $Q\subset\mathbb{P}^4$\hfill\break
that passes through the both centers of the blow up $Y\to Q$ & $0$ &
Lemma \ref{4:4} \\
\hline $4.5$ & $32$ & the blow up of $\mathbb{P}^1\times\mathbb{P}^2$
along a disjoint union of\hfill\break two irreducible curves of
bidegree $(2, 1)$ and $(1, 0)$ & $0$ & Lemma
\ref{4:5} \\
\hline $4.6$ & $34$ & the blow up of $\mathbb{P}^3$ along a disjoint
union of three lines & $0$ & Lemma \ref{4:6} \\
\hline $4.7$ & $36$ & the blow up of
$W\subset\mathbb{P}^2\times\mathbb{P}^2$ along a disjoint union
of\hfill\break two curves of bidegree $(0, 1)$ and $(1, 0)$ & &
Lemma \ref{4:7} \\
\hline $4.8$ & $38$ & the blow up of
$\mathbb{P}^1\times\mathbb{P}^1\times\mathbb{P}^1$ along a curve of
type $(0, 1, 1)$  & $0$ & Lemma \ref{4:8} \\
\hline $4.9$ & $40$ & the blow up of the smooth Fano threefold $Y$
from the family
$\textnumero\ 3.25$ along a curve that is contracted
by the blow
up $Y \to\mathbb{P}^3$  &  $0$ & Remark~\ref{remark:toric} \\
\hline $4.10$ & $42$ & $\mathbb{P}^1\times S_7$  & $0$ &
Remark~\ref{remark:toric} \\
\hline $4.11$ & $44$ & the blow up of
$\mathbb{P}^1\times\mathbb{F}_1$ along a curve $C\cong\mathbb{P}^1$ such
that $C$ is contained in a fiber $F\cong\mathbb{F}_{1}$ of
the projection
$\mathbb{P}^1\times\mathbb{F}_{1}\to\mathbb{P}^1$ and $C\cdot C=-1$ on
$F$ & $0$ & Remark~\ref{remark:toric} \\
\hline $4.12$ & $46$ & the blow up of the smooth Fano threefold
$Y$ from the family $\textnumero\ 2.33$ along two curves that are
contracted by the blow up
$Y\to\mathbb{P}^3$ & $0$ & Remark~\ref{remark:toric} \\
\hline $4.13$ & $26$ & the blow up of
$\mathbb{P}^1\times\mathbb{P}^1\times\mathbb{P}^1$ along a curve of
type $(1, 1, 3)$ & $0$ & Lemma \ref{4:13} \\
\hline $5.1$ & $28$ & the blow up of the smooth Fano threefold $Y$
from the family $\textnumero\ 2.29$ along three curves that are
contracted by the
blow up $Y\to Q$ & $0$ & Lemma \ref{5:1} \\
\hline $5.2$ & $36$ & the blow up of the smooth Fano threefold $Y$
in the family $3.25$ along two curves $C_{1}\ne
C_{2}$ that are contracted by the blow up $\phi\colon
Y\to\nolinebreak\mathbb{P}^3$ and that are contained in the
same exceptional divisor of the blow up $\phi$  &
$0$ & Remark~\ref{remark:toric} \\
\hline $5.3$ & $36$ & $\mathbb{P}^1\times S_6$  & $0$ &
Remark~\ref{remark:toric} \\
\hline $6.1$ & $30$ & $\mathbb{P}^1\times S_5$  & $0$ &
Corollary~\ref{corollary:P1 x del Pezzo} \\
\hline $7.1$ & $24$ & $\mathbb{P}^1\times S_4$  & $0$ &
Corollary~\ref{corollary:P1 x del Pezzo} \\
\hline $8.1$ & $18$ & $\mathbb{P}^1\times S_3$  & $0$ &
Corollary~\ref{corollary:P1 x del Pezzo} \\
\hline $9.1$ & $12$ & $\mathbb{P}^1\times S_2$  & $0$ &
Corollary~\ref{corollary:P1 x del Pezzo} \\
\hline $10.1$ & $6$ & $\mathbb{P}^1\times S_1$  & gen. $0$ &
Corollary~\ref{lemma:coreg P1xS1} \\
\hline
\end{longtable}
}

\end{document}